\documentclass[12pt]{amsart}

\usepackage{times,bm,amsfonts,amsmath,amssymb,mathrsfs,tikz,hyperref,dsfont,csquotes,enumerate} 
\hypersetup{pdfborder={0000}, colorlinks=true, linkcolor=blue,citecolor=citegreen}
\definecolor{citegreen}{rgb}{0,0.75,0}
\usepackage[babel=true,kerning=true]{microtype}

\ifx\figforTeXisloaded\relax \else\global\let\figforTeXisloaded=\relax\fi
\catcode`\@=11
\ifx\ctr@ln@m\undefined\else%
    \immediate\write16{*** Fig4TeX WARNING : \string\ctr@ln@m\space already defined.}\fi
\def\ctr@ln@m#1{\ifx#1\undefined\else%
    \immediate\write16{*** Fig4TeX WARNING : \string#1 already defined.}\fi}
\ctr@ln@m\ctr@ld@f
\def\ctr@ld@f#1#2{\ctr@ln@m#2#1#2}
\ctr@ld@f\def\ctr@ln@w#1#2{\ctr@ln@m#2\csname#1\endcsname#2}
{\catcode`\/=0 \catcode`/\=12 /ctr@ld@f/gdef/BS@{\}}
\ctr@ld@f\def\ctr@lcsn@m#1{\expandafter\ifx\csname#1\endcsname\relax\else%
    \immediate\write16{*** Fig4TeX WARNING : \BS@\expandafter\string#1\space already defined.}\fi}
\ctr@ld@f\edef\colonc@tcode{\the\catcode`\:}
\ctr@ld@f\edef\semicolonc@tcode{\the\catcode`\;}
\ctr@ld@f\def\t@stc@tcodech@nge{{\let\c@tcodech@nged=\z@%
    \ifnum\colonc@tcode=\the\catcode`\:\else\let\c@tcodech@nged=\@ne\fi%
    \ifnum\semicolonc@tcode=\the\catcode`\;\else\let\c@tcodech@nged=\@ne\fi%
    \ifx\c@tcodech@nged\@ne%
    \immediate\write16{}
    \immediate\write16{!!!=============================================================!!!}
    \immediate\write16{ Fig4TeX WARNING:}
    \immediate\write16{ The category code of some characters has been changed, which will}
    \immediate\write16{ result in an error (message "Runaway argument?").}
    \immediate\write16{ This probably comes from another package that changed the category}
    \immediate\write16{ code after Fig4TeX was loaded. If that proves to be exact, the}
    \immediate\write16{ solution is to exchange the loading commands on top of your file}
    \immediate\write16{ so that Fig4TeX is loaded last. For example, in LaTeX, we should}
    \immediate\write16{ say :}
    \immediate\write16{\BS@ usepackage[french]{babel}}
    \immediate\write16{\BS@ usepackage{fig4tex}}
    \immediate\write16{!!!=============================================================!!!}
    \immediate\write16{}
    \fi}}
\ctr@ld@f\def\FigforTeX{F\kern-.05em i\kern-.05em g\kern-.1em\raise-.14em\hbox{4}\kern-.19em\TeX}
\ctr@ld@f\def\W@rnmesoldA#1{\W@rnmesold}
\ctr@ld@f\def\W@rnmesoldAB#1(#2){\W@rnmesold}
\ctr@ld@f\def\W@rnmesold{%
    \immediate\write16{}
    \immediate\write16{!!!=============================================================!!!}
    \immediate\write16{ Fig4TeX WARNING:}
    \immediate\write16{ The file to be compiled is not compatible with the current version}
    \immediate\write16{ of Fig4TeX. To fix that, upgrade the source file (mainly change \BS@ ps*}
    \immediate\write16{ macros by \BS@ fig* macros), or use fig4tex184.tex instead (\BS@ input fig4tex184}
    \immediate\write16{ or \BS@ usepackage{fig4tex184}).}
    \immediate\write16{!!!=============================================================!!!}
    \immediate\write16{}}
\ctr@ln@m\psbeginfig\let\psbeginfig\W@rnmesoldA
\ctr@ln@m\psset\let\psset\W@rnmesoldAB
\ctr@ln@m\pssetdefault\let\pssetdefault\W@rnmesoldAB
\ctr@ln@m\pssetupdate\let\pssetupdate\W@rnmesoldA
\ctr@ln@w{newdimen}\epsil@n\epsil@n=0.00005pt
\ctr@ln@w{newdimen}\Cepsil@n\Cepsil@n=0.005pt
\ctr@ln@w{newdimen}\dcq@\dcq@=254pt
\ctr@ln@w{newdimen}\PI@\PI@=3.141592pt
\ctr@ln@w{newdimen}\DemiPI@deg\DemiPI@deg=90pt
\ctr@ln@w{newdimen}\PI@deg\PI@deg=180pt
\ctr@ln@w{newdimen}\DePI@deg\DePI@deg=360pt
\ctr@ld@f\chardef\t@n=10
\ctr@ld@f\chardef\c@nt=100
\ctr@ld@f\chardef\@lxxiv=74
\ctr@ld@f\chardef\@xci=91
\ctr@ld@f\mathchardef\@nMnCQn=9949
\ctr@ld@f\chardef\@vi=6
\ctr@ld@f\chardef\@xxx=30
\ctr@ld@f\chardef\@lvi=56
\ctr@ld@f\chardef\@@lxxi=71
\ctr@ld@f\chardef\@lxxxv=85
\ctr@ld@f\mathchardef\@@mmmmlxviii=4068
\ctr@ld@f\mathchardef\@ccclx=360
\ctr@ld@f\mathchardef\@dccxx=720
\ctr@ln@w{newcount}\p@rtent \ctr@ln@w{newcount}\f@ctech \ctr@ln@w{newcount}\result@tent
\ctr@ln@w{newdimen}\v@lmin \ctr@ln@w{newdimen}\v@lmax \ctr@ln@w{newdimen}\v@leur
\ctr@ln@w{newdimen}\result@t\ctr@ln@w{newdimen}\result@@t
\ctr@ln@w{newdimen}\mili@u \ctr@ln@w{newdimen}\c@rre \ctr@ln@w{newdimen}\delt@
\ctr@ld@f\def\degT@rd{0.017453 }  
\ctr@ld@f\def\rdT@deg{57.295779 } 
\ctr@ln@m\v@leurseule
{\catcode`p=12 \catcode`t=12 \gdef\v@leurseule#1pt{#1}}
\ctr@ld@f\def\repdecn@mb#1{\expandafter\v@leurseule\the#1\space}
\ctr@ld@f\def\arct@n#1(#2,#3){{\v@lmin=#2\v@lmax=#3%
    \maxim@m{\mili@u}{-\v@lmin}{\v@lmin}\maxim@m{\c@rre}{-\v@lmax}{\v@lmax}%
    \delt@=\mili@u\m@ech\mili@u%
    \ifdim\c@rre>\@nMnCQn\mili@u\divide\v@lmax\tw@\c@lATAN\v@leur(\z@,\v@lmax)
    \else%
    \maxim@m{\mili@u}{-\v@lmin}{\v@lmin}\maxim@m{\c@rre}{-\v@lmax}{\v@lmax}%
    \m@ech\c@rre%
    \ifdim\mili@u>\@nMnCQn\c@rre\divide\v@lmin\tw@
    \maxim@m{\mili@u}{-\v@lmin}{\v@lmin}\c@lATAN\v@leur(\mili@u,\z@)%
    \else\c@lATAN\v@leur(\delt@,\v@lmax)\fi\fi%
    \ifdim\v@lmin<\z@\v@leur=-\v@leur\ifdim\v@lmax<\z@\advance\v@leur-\PI@%
    \else\advance\v@leur\PI@\fi\fi%
    \global\result@t=\v@leur}#1=\result@t}
\ctr@ld@f\def\m@ech#1{\ifdim#1>1.646pt\divide\mili@u\t@n\divide\c@rre\t@n\m@ech#1\fi}
\ctr@ld@f\def\c@lATAN#1(#2,#3){{\v@lmin=#2\v@lmax=#3\v@leur=\z@\delt@=\tw@ pt%
    \un@iter{0.785398}{\v@lmax<}%
    \un@iter{0.463648}{\v@lmax<}%
    \un@iter{0.244979}{\v@lmax<}%
    \un@iter{0.124355}{\v@lmax<}%
    \un@iter{0.062419}{\v@lmax<}%
    \un@iter{0.031240}{\v@lmax<}%
    \un@iter{0.015624}{\v@lmax<}%
    \un@iter{0.007812}{\v@lmax<}%
    \un@iter{0.003906}{\v@lmax<}%
    \un@iter{0.001953}{\v@lmax<}%
    \un@iter{0.000976}{\v@lmax<}%
    \un@iter{0.000488}{\v@lmax<}%
    \un@iter{0.000244}{\v@lmax<}%
    \un@iter{0.000122}{\v@lmax<}%
    \un@iter{0.000061}{\v@lmax<}%
    \un@iter{0.000030}{\v@lmax<}%
    \un@iter{0.000015}{\v@lmax<}%
    \global\result@t=\v@leur}#1=\result@t}
\ctr@ld@f\def\un@iter#1#2{%
    \divide\delt@\tw@\edef\dpmn@{\repdecn@mb{\delt@}}%
    \mili@u=\v@lmin%
    \ifdim#2\z@%
      \advance\v@lmin-\dpmn@\v@lmax\advance\v@lmax\dpmn@\mili@u%
      \advance\v@leur-#1pt%
    \else%
      \advance\v@lmin\dpmn@\v@lmax\advance\v@lmax-\dpmn@\mili@u%
      \advance\v@leur#1pt%
    \fi}
\ctr@ld@f\def\c@ssin#1#2#3{\expandafter\ifx\csname COS@\number#3\endcsname\relax\c@lCS{#3pt}%
    \expandafter\xdef\csname COS@\number#3\endcsname{\repdecn@mb\result@t}%
    \expandafter\xdef\csname SIN@\number#3\endcsname{\repdecn@mb\result@@t}\fi%
    \edef#1{\csname COS@\number#3\endcsname}\edef#2{\csname SIN@\number#3\endcsname}}
\ctr@ld@f\def\c@lCS#1{{\mili@u=#1\p@rtent=\@ne%
    \relax\ifdim\mili@u<\z@\red@ng<-\else\red@ng>+\fi\f@ctech=\p@rtent%
    \relax\ifdim\mili@u<\z@\mili@u=-\mili@u\f@ctech=-\f@ctech\fi\c@@lCS}}
\ctr@ld@f\def\c@@lCS{\v@lmin=\mili@u\c@rre=-\mili@u\advance\c@rre\DemiPI@deg\v@lmax=\c@rre%
    \mili@u\@@lxxi\mili@u\divide\mili@u\@@mmmmlxviii%
    \edef\v@larg{\repdecn@mb{\mili@u}}\mili@u=-\v@larg\mili@u%
    \edef\v@lmxde{\repdecn@mb{\mili@u}}%
    \c@rre\@@lxxi\c@rre\divide\c@rre\@@mmmmlxviii%
    \edef\v@largC{\repdecn@mb{\c@rre}}\c@rre=-\v@largC\c@rre%
    \edef\v@lmxdeC{\repdecn@mb{\c@rre}}%
    \fctc@s\mili@u\v@lmin\global\result@t\p@rtent\v@leur%
    \let\t@mp=\v@larg\let\v@larg=\v@largC\let\v@largC=\t@mp%
    \let\t@mp=\v@lmxde\let\v@lmxde=\v@lmxdeC\let\v@lmxdeC=\t@mp%
    \fctc@s\c@rre\v@lmax\global\result@@t\f@ctech\v@leur}
\ctr@ld@f\def\fctc@s#1#2{\v@leur=#1\relax\ifdim#2<\@lxxxv\p@\cosser@h\else\sinser@t\fi}
\ctr@ld@f\def\cosser@h{\advance\v@leur\@lvi\p@\divide\v@leur\@lvi%
    \v@leur=\v@lmxde\v@leur\advance\v@leur\@xxx\p@%
    \v@leur=\v@lmxde\v@leur\advance\v@leur\@ccclx\p@%
    \v@leur=\v@lmxde\v@leur\advance\v@leur\@dccxx\p@\divide\v@leur\@dccxx}
\ctr@ld@f\def\sinser@t{\v@leur=\v@lmxdeC\p@\advance\v@leur\@vi\p@%
    \v@leur=\v@largC\v@leur\divide\v@leur\@vi}
\ctr@ld@f\def\red@ng#1#2{\relax\ifdim\mili@u#1#2\DemiPI@deg\advance\mili@u#2-\PI@deg%
    \p@rtent=-\p@rtent\red@ng#1#2\fi}
\ctr@ld@f\def\pr@c@lCS#1#2#3{\ctr@lcsn@m{COS@\number#3 }%
    \expandafter\xdef\csname COS@\number#3\endcsname{#1}%
    \expandafter\xdef\csname SIN@\number#3\endcsname{#2}}
\pr@c@lCS{1}{0}{0}
\pr@c@lCS{0.7071}{0.7071}{45}\pr@c@lCS{0.7071}{-0.7071}{-45}
\pr@c@lCS{0}{1}{90}          \pr@c@lCS{0}{-1}{-90}
\pr@c@lCS{-1}{0}{180}        \pr@c@lCS{-1}{0}{-180}
\pr@c@lCS{0}{-1}{270}        \pr@c@lCS{0}{1}{-270}
\ctr@ld@f\def\invers@#1#2{{\v@leur=#2\maxim@m{\v@lmax}{-\v@leur}{\v@leur}%
    \f@ctech=\@ne\m@inv@rs%
    \multiply\v@leur\f@ctech\edef\v@lv@leur{\repdecn@mb{\v@leur}}%
    \p@rtentiere{\p@rtent}{\v@leur}\v@lmin=\p@\divide\v@lmin\p@rtent%
    \inv@rs@\multiply\v@lmax\f@ctech\global\result@t=\v@lmax}#1=\result@t}
\ctr@ld@f\def\m@inv@rs{\ifdim\v@lmax<\p@\multiply\v@lmax\t@n\multiply\f@ctech\t@n\m@inv@rs\fi}
\ctr@ld@f\def\inv@rs@{\v@lmax=-\v@lmin\v@lmax=\v@lv@leur\v@lmax%
    \advance\v@lmax\tw@ pt\v@lmax=\repdecn@mb{\v@lmin}\v@lmax%
    \delt@=\v@lmax\advance\delt@-\v@lmin\ifdim\delt@<\z@\delt@=-\delt@\fi%
    \ifdim\delt@>\epsil@n\v@lmin=\v@lmax\inv@rs@\fi}
\ctr@ld@f\def\minim@m#1#2#3{\relax\ifdim#2<#3#1=#2\else#1=#3\fi}
\ctr@ld@f\def\maxim@m#1#2#3{\relax\ifdim#2>#3#1=#2\else#1=#3\fi}
\ctr@ld@f\def\p@rtentiere#1#2{#1=#2\divide#1by65536 }
\ctr@ld@f\def\r@undint#1#2{{\v@leur=#2\divide\v@leur\t@n\p@rtentiere{\p@rtent}{\v@leur}%
    \v@leur=\p@rtent pt\global\result@t=\t@n\v@leur}#1=\result@t}
\ctr@ld@f\def\sqrt@#1#2{{\v@leur=#2%
    \minim@m{\v@lmin}{\p@}{\v@leur}\maxim@m{\v@lmax}{\p@}{\v@leur}%
    \f@ctech=\@ne\m@sqrt@\sqrt@@%
    \mili@u=\v@lmin\advance\mili@u\v@lmax\divide\mili@u\tw@\multiply\mili@u\f@ctech%
    \global\result@t=\mili@u}#1=\result@t}
\ctr@ld@f\def\m@sqrt@{\ifdim\v@leur>\dcq@\divide\v@leur\c@nt\v@lmax=\v@leur%
    \multiply\f@ctech\t@n\m@sqrt@\fi}
\ctr@ld@f\def\sqrt@@{\mili@u=\v@lmin\advance\mili@u\v@lmax\divide\mili@u\tw@%
    \c@rre=\repdecn@mb{\mili@u}\mili@u%
    \ifdim\c@rre<\v@leur\v@lmin=\mili@u\else\v@lmax=\mili@u\fi%
    \delt@=\v@lmax\advance\delt@-\v@lmin\ifdim\delt@>\epsil@n\sqrt@@\fi}
\ctr@ld@f\def\extrairelepremi@r#1\de#2{\expandafter\lepremi@r#2@#1#2}
\ctr@ld@f\def\lepremi@r#1,#2@#3#4{\def#3{#1}\def#4{#2}\ignorespaces}
\ctr@ld@f\def\@cfor#1:=#2\do#3{%
  \edef\@fortemp{#2}%
  \ifx\@fortemp\empty\else\@cforloop#2,\@nil,\@nil\@@#1{#3}\fi}
\ctr@ln@m\@nextwhile
\ctr@ld@f\def\@cforloop#1,#2\@@#3#4{%
  \def#3{#1}%
  \ifx#3\Fig@nnil\let\@nextwhile=\Fig@fornoop\else#4\relax\let\@nextwhile=\@cforloop\fi%
  \@nextwhile#2\@@#3{#4}}

\ctr@ld@f\def\@ecfor#1:=#2\do#3{%
  \def\@@cfor{\@cfor#1:=}%
  \edef\@@@cfor{#2}%
  \expandafter\@@cfor\@@@cfor\do{#3}}
\ctr@ld@f\def\Fig@nnil{\@nil}
\ctr@ld@f\def\Fig@fornoop#1\@@#2#3{}
\ctr@ln@m\list@@rg
\ctr@ld@f\def\trtlis@rg#1#2{\def\list@@rg{#1}%
    \@ecfor\p@rv@l:=\list@@rg\do{\expandafter#2\p@rv@l|}}
\ctr@ld@f\def\trtlis@rgtok#1{\let@xte={}\let\n@xt\addt@t@xt\addt@t@xt #1}
\ctr@ln@m\M@cro
\ctr@ln@m\n@xt
\ctr@ld@f\def\addt@t@xt#1{\if#1|\let\n@xt\relax\else%
    \if#1,\expandafter\M@cro\the\let@xte|\let@xte={}%
    \else\let@xte=\expandafter{\the\let@xte #1}\fi\fi\n@xt}
\ctr@ln@w{newbox}\b@xvisu
\ctr@ln@w{newtoks}\let@xte
\ctr@ln@w{newif}\ifitis@K
\ctr@ln@w{newcount}\s@mme
\ctr@ln@w{newcount}\l@mbd@un \ctr@ln@w{newcount}\l@mbd@de
\ctr@ln@w{newcount}\superc@ntr@l\superc@ntr@l=\@ne        
\ctr@ln@w{newcount}\typec@ntr@l\typec@ntr@l=\superc@ntr@l 
\ctr@ln@w{newdimen}\v@lX  \ctr@ln@w{newdimen}\v@lY  \ctr@ln@w{newdimen}\v@lZ
\ctr@ln@w{newdimen}\v@lXa \ctr@ln@w{newdimen}\v@lYa \ctr@ln@w{newdimen}\v@lZa
\ctr@ln@w{newdimen}\unit@\unit@=\p@ 
\ctr@ld@f\def\unit@util{pt}
\ctr@ld@f\def\ptT@ptps{0.996264 }
\ctr@ld@f\def\ptpsT@pt{1.00375 }
\ctr@ld@f\def\ptT@unit@{1} 
\ctr@ld@f\def\setunit@#1{\def\unit@util{#1}\setunit@@#1:\invers@{\result@t}{\unit@}%
    \edef\ptT@unit@{\repdecn@mb\result@t}}
\ctr@ld@f\def\setunit@@#1#2:{\ifcat#1a\unit@=\@ne#1#2\else\unit@=#1#2\fi}
\ctr@ld@f\def\d@fm@cdim#1#2{{\v@leur=#2\v@leur=\ptT@unit@\v@leur\xdef#1{\repdecn@mb\v@leur}}}
\ctr@ln@w{newif}\ifBdingB@x\BdingB@xtrue
\ctr@ln@w{newdimen}\c@@rdXmin \ctr@ln@w{newdimen}\c@@rdYmin  
\ctr@ln@w{newdimen}\c@@rdXmax \ctr@ln@w{newdimen}\c@@rdYmax
\ctr@ld@f\def\b@undb@x#1#2{\ifBdingB@x%
    \relax\ifdim#1<\c@@rdXmin\global\c@@rdXmin=#1\fi%
    \relax\ifdim#2<\c@@rdYmin\global\c@@rdYmin=#2\fi%
    \relax\ifdim#1>\c@@rdXmax\global\c@@rdXmax=#1\fi%
    \relax\ifdim#2>\c@@rdYmax\global\c@@rdYmax=#2\fi\fi}
\ctr@ld@f\def\b@undb@xP#1{{\Figg@tXY{#1}\b@undb@x{\v@lX}{\v@lY}}}
\ctr@ld@f\def\ellBB@x#1;#2,#3(#4,#5,#6){{\s@uvc@ntr@l\et@tellBB@x%
    \setc@ntr@l{2}\figptell-2::#1;#2,#3(#4,#6)\b@undb@xP{-2}%
    \figptell-2::#1;#2,#3(#5,#6)\b@undb@xP{-2}%
    \c@ssin{\C@}{\S@}{#6}\v@lmin=\C@ pt\v@lmax=\S@ pt%
    \mili@u=#3\v@lmin\delt@=#2\v@lmax\arct@n\v@leur(\delt@,\mili@u)%
    \mili@u=-#3\v@lmax\delt@=#2\v@lmin\arct@n\c@rre(\delt@,\mili@u)%
    \v@leur=\rdT@deg\v@leur\advance\v@leur-\DePI@deg%
    \c@rre=\rdT@deg\c@rre\advance\c@rre-\DePI@deg%
    \v@lmin=#4pt\v@lmax=#5pt%
    \loop\ifdim\v@leur<\v@lmax\ifdim\v@leur>\v@lmin%
    \edef\@ngle{\repdecn@mb\v@leur}\figptell-2::#1;#2,#3(\@ngle,#6)%
    \b@undb@xP{-2}\fi\advance\v@leur\PI@deg\repeat%
    \loop\ifdim\c@rre<\v@lmax\ifdim\c@rre>\v@lmin%
    \edef\@ngle{\repdecn@mb\c@rre}\figptell-2::#1;#2,#3(\@ngle,#6)%
    \b@undb@xP{-2}\fi\advance\c@rre\PI@deg\repeat%
    \resetc@ntr@l\et@tellBB@x}\ignorespaces}
\ctr@ld@f\def\initb@undb@x{\c@@rdXmin=\maxdimen\c@@rdYmin=\maxdimen%
    \c@@rdXmax=-\maxdimen\c@@rdYmax=-\maxdimen}
\ctr@ld@f\def\c@ntr@lnum#1{%
    \relax\ifnum\typec@ntr@l=\@ne%
    \ifnum#1<\z@%
    \immediate\write16{*** Forbidden point number (#1). Abort.}\end\fi\fi%
    \set@bjc@de{#1}}
\ctr@ln@m\objc@de
\ctr@ld@f\def\set@bjc@de#1{\edef\objc@de{@BJ\ifnum#1<\z@ M\romannumeral-#1\else\romannumeral#1\fi}}
\s@mme=\m@ne\loop\ifnum\s@mme>-19
  \set@bjc@de{\s@mme}\ctr@lcsn@m\objc@de\ctr@lcsn@m{\objc@de T}
\advance\s@mme\m@ne\repeat
\s@mme=\@ne\loop\ifnum\s@mme<6
  \set@bjc@de{\s@mme}\ctr@lcsn@m\objc@de\ctr@lcsn@m{\objc@de T}
\advance\s@mme\@ne\repeat
\ctr@ld@f\def\setc@ntr@l#1{\ifnum\superc@ntr@l>#1\typec@ntr@l=\superc@ntr@l%
    \else\typec@ntr@l=#1\fi}
\ctr@ld@f\def\resetc@ntr@l#1{\global\superc@ntr@l=#1\setc@ntr@l{#1}}
\ctr@ld@f\def\s@uvc@ntr@l#1{\edef#1{\the\superc@ntr@l}}
\ctr@ln@m\c@lproscal
\ctr@ld@f\def\c@lproscalDD#1[#2,#3]{{\Figg@tXY{#2}%
    \edef\Xu@{\repdecn@mb{\v@lX}}\edef\Yu@{\repdecn@mb{\v@lY}}\Figg@tXY{#3}%
    \global\result@t=\Xu@\v@lX\global\advance\result@t\Yu@\v@lY}#1=\result@t}
\ctr@ld@f\def\c@lproscalTD#1[#2,#3]{{\Figg@tXY{#2}\edef\Xu@{\repdecn@mb{\v@lX}}%
    \edef\Yu@{\repdecn@mb{\v@lY}}\edef\Zu@{\repdecn@mb{\v@lZ}}%
    \Figg@tXY{#3}\global\result@t=\Xu@\v@lX\global\advance\result@t\Yu@\v@lY%
    \global\advance\result@t\Zu@\v@lZ}#1=\result@t}
\ctr@ld@f\def\c@lprovec#1{%
    \det@rmC\v@lZa(\v@lX,\v@lY,\v@lmin,\v@lmax)%
    \det@rmC\v@lXa(\v@lY,\v@lZ,\v@lmax,\v@leur)%
    \det@rmC\v@lYa(\v@lZ,\v@lX,\v@leur,\v@lmin)%
    \Figv@ctCreg#1(\v@lXa,\v@lYa,\v@lZa)}
\ctr@ld@f\def\det@rm#1[#2,#3]{{\Figg@tXY{#2}\Figg@tXYa{#3}%
    \delt@=\repdecn@mb{\v@lX}\v@lYa\advance\delt@-\repdecn@mb{\v@lY}\v@lXa%
    \global\result@t=\delt@}#1=\result@t}
\ctr@ld@f\def\det@rmC#1(#2,#3,#4,#5){{\global\result@t=\repdecn@mb{#2}#5%
    \global\advance\result@t-\repdecn@mb{#3}#4}#1=\result@t}
\ctr@ld@f\def\getredf@ctDD#1(#2,#3){{\maxim@m{\v@lXa}{-#2}{#2}\maxim@m{\v@lYa}{-#3}{#3}%
    \maxim@m{\v@lXa}{\v@lXa}{\v@lYa}
    \ifdim\v@lXa>\@xci pt\divide\v@lXa\@xci%
    \p@rtentiere{\p@rtent}{\v@lXa}\advance\p@rtent\@ne\else\p@rtent=\@ne\fi%
    \global\result@tent=\p@rtent}#1=\result@tent\ignorespaces}
\ctr@ld@f\def\getredf@ctTD#1(#2,#3,#4){{\maxim@m{\v@lXa}{-#2}{#2}\maxim@m{\v@lYa}{-#3}{#3}%
    \maxim@m{\v@lZa}{-#4}{#4}\maxim@m{\v@lXa}{\v@lXa}{\v@lYa}%
    \maxim@m{\v@lXa}{\v@lXa}{\v@lZa}
    \ifdim\v@lXa>\@lxxiv pt\divide\v@lXa\@lxxiv%
    \p@rtentiere{\p@rtent}{\v@lXa}\advance\p@rtent\@ne\else\p@rtent=\@ne\fi%
    \global\result@tent=\p@rtent}#1=\result@tent\ignorespaces}
\ctr@ln@m\getredf@ctB
\ctr@ld@f\def\getredf@ctBDD#1{\getredf@ctDD#1(\v@lX,\v@lY)}
\ctr@ld@f\def\getredf@ctBTD#1{\getredf@ctTD#1(\v@lX,\v@lY,\v@lZ)}
\ctr@ld@f\def\FigptintercircB@zDD#1:#2:#3,#4[#5,#6,#7,#8]{{\s@uvc@ntr@l\et@tfigptintercircB@zDD%
    \setc@ntr@l{2}\figvectPDD-1[#5,#8]\Figg@tXY{-1}\getredf@ctDD\f@ctech(\v@lX,\v@lY)%
    \mili@u=#4\unit@\divide\mili@u\f@ctech\c@rre=\repdecn@mb{\mili@u}\mili@u%
    \figptBezierDD-5::#3[#5,#6,#7,#8]%
    \v@lmin=#3\p@\v@lmax=\v@lmin\advance\v@lmax0.1\p@%
    \loop\edef\T@{\repdecn@mb{\v@lmax}}\figptBezierDD-2::\T@[#5,#6,#7,#8]%
    \figvectPDD-1[-5,-2]\n@rmeucCDD{\delt@}{-1}\ifdim\delt@<\c@rre\v@lmin=\v@lmax%
    \advance\v@lmax0.1\p@\repeat%
    \loop\mili@u=\v@lmin\advance\mili@u\v@lmax%
    \divide\mili@u\tw@\edef\T@{\repdecn@mb{\mili@u}}\figptBezierDD-2::\T@[#5,#6,#7,#8]%
    \figvectPDD-1[-5,-2]\n@rmeucCDD{\delt@}{-1}\ifdim\delt@>\c@rre\v@lmax=\mili@u%
    \else\v@lmin=\mili@u\fi\v@leur=\v@lmax\advance\v@leur-\v@lmin%
    \ifdim\v@leur>\epsil@n\repeat\figptcopyDD#1:#2/-2/%
    \resetc@ntr@l\et@tfigptintercircB@zDD}\ignorespaces}
\ctr@ln@m\figptinterlines
\ctr@ld@f\def\inters@cDD#1:#2[#3,#4;#5,#6]{{\s@uvc@ntr@l\et@tinters@cDD%
    \setc@ntr@l{2}\vecunit@{-1}{#4}\vecunit@{-2}{#6}%
    \Figg@tXY{-1}\setc@ntr@l{1}\Figg@tXYa{#3}%
    \edef\A@{\repdecn@mb{\v@lX}}\edef\B@{\repdecn@mb{\v@lY}}%
    \v@lmin=\B@\v@lXa\advance\v@lmin-\A@\v@lYa%
    \Figg@tXYa{#5}\setc@ntr@l{2}\Figg@tXY{-2}%
    \edef\C@{\repdecn@mb{\v@lX}}\edef\D@{\repdecn@mb{\v@lY}}%
    \v@lmax=\D@\v@lXa\advance\v@lmax-\C@\v@lYa%
    \delt@=\A@\v@lY\advance\delt@-\B@\v@lX%
    \invers@{\v@leur}{\delt@}\edef\v@ldelta{\repdecn@mb{\v@leur}}%
    \v@lXa=\A@\v@lmax\advance\v@lXa-\C@\v@lmin%
    \v@lYa=\B@\v@lmax\advance\v@lYa-\D@\v@lmin%
    \v@lXa=\v@ldelta\v@lXa\v@lYa=\v@ldelta\v@lYa%
    \setc@ntr@l{1}\Figp@intregDD#1:{#2}(\v@lXa,\v@lYa)%
    \resetc@ntr@l\et@tinters@cDD}\ignorespaces}
\ctr@ld@f\def\inters@cTD#1:#2[#3,#4;#5,#6]{{\s@uvc@ntr@l\et@tinters@cTD%
    \setc@ntr@l{2}\figvectNVTD-1[#4,#6]\figvectNVTD-2[#6,-1]\figvectPTD-1[#3,#5]%
    \r@pPSTD\v@leur[-2,-1,#4]\edef\v@lcoef{\repdecn@mb{\v@leur}}%
    \figpttraTD#1:{#2}=#3/\v@lcoef,#4/\resetc@ntr@l\et@tinters@cTD}\ignorespaces}
\ctr@ld@f\def\r@pPSTD#1[#2,#3,#4]{{\Figg@tXY{#2}\edef\Xu@{\repdecn@mb{\v@lX}}%
    \edef\Yu@{\repdecn@mb{\v@lY}}\edef\Zu@{\repdecn@mb{\v@lZ}}%
    \Figg@tXY{#3}\v@lmin=\Xu@\v@lX\advance\v@lmin\Yu@\v@lY\advance\v@lmin\Zu@\v@lZ%
    \Figg@tXY{#4}\v@lmax=\Xu@\v@lX\advance\v@lmax\Yu@\v@lY\advance\v@lmax\Zu@\v@lZ%
    \invers@{\v@leur}{\v@lmax}\global\result@t=\repdecn@mb{\v@leur}\v@lmin}%
    #1=\result@t}
\ctr@ln@m\n@rminf
\ctr@ld@f\def\n@rminfDD#1#2{{\Figg@tXY{#2}\maxim@m{\v@lX}{\v@lX}{-\v@lX}%
    \maxim@m{\v@lY}{\v@lY}{-\v@lY}\maxim@m{\global\result@t}{\v@lX}{\v@lY}}%
    #1=\result@t}
\ctr@ld@f\def\n@rminfTD#1#2{{\Figg@tXY{#2}\maxim@m{\v@lX}{\v@lX}{-\v@lX}%
    \maxim@m{\v@lY}{\v@lY}{-\v@lY}\maxim@m{\v@lZ}{\v@lZ}{-\v@lZ}%
    \maxim@m{\v@lX}{\v@lX}{\v@lY}\maxim@m{\global\result@t}{\v@lX}{\v@lZ}}%
    #1=\result@t}
\ctr@ln@m\n@rmeucC
\ctr@ld@f\def\n@rmeucCDD#1#2{\Figg@tXY{#2}\divide\v@lX\f@ctech\divide\v@lY\f@ctech%
    #1=\repdecn@mb{\v@lX}\v@lX\v@lX=\repdecn@mb{\v@lY}\v@lY\advance#1\v@lX}
\ctr@ld@f\def\n@rmeucCTD#1#2{\Figg@tXY{#2}%
    \divide\v@lX\f@ctech\divide\v@lY\f@ctech\divide\v@lZ\f@ctech%
    #1=\repdecn@mb{\v@lX}\v@lX\v@lX=\repdecn@mb{\v@lY}\v@lY\advance#1\v@lX%
    \v@lX=\repdecn@mb{\v@lZ}\v@lZ\advance#1\v@lX}
\ctr@ln@m\n@rmeucSV
\ctr@ld@f\def\n@rmeucSVDD#1#2{{\Figg@tXY{#2}%
    \v@lXa=\repdecn@mb{\v@lX}\v@lX\v@lYa=\repdecn@mb{\v@lY}\v@lY%
    \advance\v@lXa\v@lYa\sqrt@{\global\result@t}{\v@lXa}}#1=\result@t}
\ctr@ld@f\def\n@rmeucSVTD#1#2{{\Figg@tXY{#2}\v@lXa=\repdecn@mb{\v@lX}\v@lX%
    \v@lYa=\repdecn@mb{\v@lY}\v@lY\v@lZa=\repdecn@mb{\v@lZ}\v@lZ%
    \advance\v@lXa\v@lYa\advance\v@lXa\v@lZa\sqrt@{\global\result@t}{\v@lXa}}#1=\result@t}
\ctr@ln@m\n@rmeuc
\ctr@ld@f\def\n@rmeucDD#1#2{{\Figg@tXY{#2}\getredf@ctDD\f@ctech(\v@lX,\v@lY)%
    \divide\v@lX\f@ctech\divide\v@lY\f@ctech%
    \v@lXa=\repdecn@mb{\v@lX}\v@lX\v@lYa=\repdecn@mb{\v@lY}\v@lY%
    \advance\v@lXa\v@lYa\sqrt@{\global\result@t}{\v@lXa}%
    \global\multiply\result@t\f@ctech}#1=\result@t}
\ctr@ld@f\def\n@rmeucTD#1#2{{\Figg@tXY{#2}\getredf@ctTD\f@ctech(\v@lX,\v@lY,\v@lZ)%
    \divide\v@lX\f@ctech\divide\v@lY\f@ctech\divide\v@lZ\f@ctech%
    \v@lXa=\repdecn@mb{\v@lX}\v@lX%
    \v@lYa=\repdecn@mb{\v@lY}\v@lY\v@lZa=\repdecn@mb{\v@lZ}\v@lZ%
    \advance\v@lXa\v@lYa\advance\v@lXa\v@lZa\sqrt@{\global\result@t}{\v@lXa}%
    \global\multiply\result@t\f@ctech}#1=\result@t}
\ctr@ln@m\vecunit@
\ctr@ld@f\def\vecunit@DD#1#2{{\Figg@tXY{#2}\getredf@ctDD\f@ctech(\v@lX,\v@lY)%
    \divide\v@lX\f@ctech\divide\v@lY\f@ctech%
    \Figv@ctCreg#1(\v@lX,\v@lY)\n@rmeucSV{\v@lYa}{#1}%
    \invers@{\v@lXa}{\v@lYa}\edef\v@lv@lXa{\repdecn@mb{\v@lXa}}%
    \v@lX=\v@lv@lXa\v@lX\v@lY=\v@lv@lXa\v@lY%
    \Figv@ctCreg#1(\v@lX,\v@lY)\multiply\v@lYa\f@ctech\global\result@t=\v@lYa}}
\ctr@ld@f\def\vecunit@TD#1#2{{\Figg@tXY{#2}\getredf@ctTD\f@ctech(\v@lX,\v@lY,\v@lZ)%
    \divide\v@lX\f@ctech\divide\v@lY\f@ctech\divide\v@lZ\f@ctech%
    \Figv@ctCreg#1(\v@lX,\v@lY,\v@lZ)\n@rmeucSV{\v@lYa}{#1}%
    \invers@{\v@lXa}{\v@lYa}\edef\v@lv@lXa{\repdecn@mb{\v@lXa}}%
    \v@lX=\v@lv@lXa\v@lX\v@lY=\v@lv@lXa\v@lY\v@lZ=\v@lv@lXa\v@lZ%
    \Figv@ctCreg#1(\v@lX,\v@lY,\v@lZ)\multiply\v@lYa\f@ctech\global\result@t=\v@lYa}}
\ctr@ld@f\def\vecunitC@TD[#1,#2]{\Figg@tXYa{#1}\Figg@tXY{#2}%
    \advance\v@lX-\v@lXa\advance\v@lY-\v@lYa\advance\v@lZ-\v@lZa\c@lvecunitTD}
\ctr@ld@f\def\vecunitCV@TD#1{\Figg@tXY{#1}\c@lvecunitTD}
\ctr@ld@f\def\c@lvecunitTD{\getredf@ctTD\f@ctech(\v@lX,\v@lY,\v@lZ)%
    \divide\v@lX\f@ctech\divide\v@lY\f@ctech\divide\v@lZ\f@ctech%
    \v@lXa=\repdecn@mb{\v@lX}\v@lX%
    \v@lYa=\repdecn@mb{\v@lY}\v@lY\v@lZa=\repdecn@mb{\v@lZ}\v@lZ%
    \advance\v@lXa\v@lYa\advance\v@lXa\v@lZa\sqrt@{\v@lYa}{\v@lXa}%
    \invers@{\v@lXa}{\v@lYa}\edef\v@lv@lXa{\repdecn@mb{\v@lXa}}%
    \v@lX=\v@lv@lXa\v@lX\v@lY=\v@lv@lXa\v@lY\v@lZ=\v@lv@lXa\v@lZ}
\ctr@ln@m\figgetangle
\ctr@ld@f\def\figgetangleDD#1[#2,#3,#4]{\ifGR@cri{\s@uvc@ntr@l\et@tfiggetangleDD\setc@ntr@l{2}%
    \figvectPDD-1[#2,#3]\figvectPDD-2[#2,#4]\vecunit@{-1}{-1}%
    \c@lproscalDD\delt@[-2,-1]\figvectNVDD-1[-1]\c@lproscalDD\v@leur[-2,-1]%
    \arct@n\v@lmax(\delt@,\v@leur)\v@lmax=\rdT@deg\v@lmax%
    \ifdim\v@lmax<\z@\advance\v@lmax\DePI@deg\fi\xdef#1{\repdecn@mb{\v@lmax}}%
    \resetc@ntr@l\et@tfiggetangleDD}\ignorespaces\fi}
\ctr@ld@f\def\figgetangleTD#1[#2,#3,#4,#5]{\ifGR@cri{\s@uvc@ntr@l\et@tfiggetangleTD\setc@ntr@l{2}%
    \figvectPTD-1[#2,#3]\figvectPTD-2[#2,#5]\figvectNVTD-3[-1,-2]%
    \figvectPTD-2[#2,#4]\figvectNVTD-4[-3,-1]%
    \vecunit@{-1}{-1}\c@lproscalTD\delt@[-2,-1]\c@lproscalTD\v@leur[-2,-4]%
    \arct@n\v@lmax(\delt@,\v@leur)\v@lmax=\rdT@deg\v@lmax%
    \ifdim\v@lmax<\z@\advance\v@lmax\DePI@deg\fi\xdef#1{\repdecn@mb{\v@lmax}}%
    \resetc@ntr@l\et@tfiggetangleTD}\ignorespaces\fi}    
\ctr@ld@f\def\figgetdist#1[#2,#3]{\ifGR@cri{\s@uvc@ntr@l\et@tfiggetdist\setc@ntr@l{2}%
    \figvectP-1[#2,#3]\n@rmeuc{\v@lX}{-1}\v@lX=\ptT@unit@\v@lX\xdef#1{\repdecn@mb{\v@lX}}%
    \resetc@ntr@l\et@tfiggetdist}\ignorespaces\fi}
\ctr@ld@f\def\figget#1=#2[#3]{\keln@mun#1|%
    \def\n@mref{a}\ifx\l@debut\n@mref\figgetangle#2[#3]\else
    \def\n@mref{d}\ifx\l@debut\n@mref\figgetdist#2[#3]\else
    \W@rnmeskwd{figget}{#1}\fi\fi\ignorespaces}
\ctr@ld@f\def\Figg@tT#1{\c@ntr@lnum{#1}%
    {\expandafter\expandafter\expandafter\extr@ctT\csname\objc@de\endcsname:%
     \ifnum\B@@ltxt=\z@\ptn@me{#1}\else\csname\objc@de T\endcsname\fi}}
\ctr@ld@f\def\extr@ctT#1,#2,#3/#4:{\def\B@@ltxt{#3}}
\ctr@ld@f\def\Figg@tXY#1{\c@ntr@lnum{#1}%
    \expandafter\expandafter\expandafter\extr@ctC\csname\objc@de\endcsname:}
\ctr@ln@m\extr@ctC
\ctr@ld@f\def\extr@ctCDD#1/#2,#3,#4:{\v@lX=#2\v@lY=#3}
\ctr@ld@f\def\extr@ctCTD#1/#2,#3,#4:{\v@lX=#2\v@lY=#3\v@lZ=#4}
\ctr@ld@f\def\Figg@tXYa#1{\c@ntr@lnum{#1}%
    \expandafter\expandafter\expandafter\extr@ctCa\csname\objc@de\endcsname:}
\ctr@ln@m\extr@ctCa
\ctr@ld@f\def\extr@ctCaDD#1/#2,#3,#4:{\v@lXa=#2\v@lYa=#3}
\ctr@ld@f\def\extr@ctCaTD#1/#2,#3,#4:{\v@lXa=#2\v@lYa=#3\v@lZa=#4}
\ctr@ln@m\t@xt@
\ctr@ld@f\def\figinit#1{\t@stc@tcodech@nge\initpr@lim\Figinit@#1,:\initpss@ttings\ignorespaces}
\ctr@ld@f\def\Figinit@#1,#2:{\setunit@{#1}\def\t@xt@{#2}\ifx\t@xt@\empty\else\Figinit@@#2:\fi}
\ctr@ld@f\def\Figinit@@#1#2:{\if#12 \else\Figs@tproj{#1}\initTD@\fi}
\ctr@ln@w{newif}\ifTr@isDim
\ctr@ld@f\def\UnD@fined{UNDEFINED}
\ctr@ln@m\@utoFN
\ctr@ln@m\@utoFInDone
\ctr@ln@m\disob@unit
\ctr@ld@f\def\initpr@lim{\initb@undb@x\figsetmark{}\figsetptname{$A_{##1}$}\def\Sc@leFact{1}%
    \initDD@\figsetroundcoord{yes}\GR@critrue\expandafter\setupd@te\D@FTupdate:%
    \edef\disob@unit{\UnD@fined}\edef\t@rgetpt{\UnD@fined}\gdef\@utoFInDone{1}\gdef\@utoFN{0}}
\ctr@ld@f\def\initDD@{\Tr@isDimfalse%
    \ifPDFm@ke%
     \let\Ps@rcerc=\Ps@rcercBz%
     \let\Ps@rell=\Ps@rellBz%
    \fi
    \let\c@lDCUn=\c@lDCUnDD%
    \let\c@lDCDeux=\c@lDCDeuxDD%
    \let\c@ldefproj=\relax%
    \let\c@lproscal=\c@lproscalDD%
    \let\c@lprojSP=\relax%
    \let\extr@ctC=\extr@ctCDD%
    \let\extr@ctCa=\extr@ctCaDD%
    \let\extr@ctCF=\extr@ctCFDD%
    \let\Figp@intreg=\Figp@intregDD%
    \let\Figpts@xes=\Figpts@xesDD%
    \let\getredf@ctB=\getredf@ctBDD%
    \let\n@rmeucSV=\n@rmeucSVDD\let\n@rmeuc=\n@rmeucDD\let\n@rmeucC\n@rmeucCDD\let\n@rminf=\n@rminfDD%
    \let\pr@dMatV=\pr@dMatVDD%
    \let\Q@@xes=\Q@@xesDD%
    \let\vecunit@=\vecunit@DD%
    \let\figcoord=\figcoordDD%
    \let\figgetangle=\figgetangleDD%
    \let\figpt=\figptDD%
    \let\figptBezier=\figptBezierDD%
    \let\figptbary=\figptbaryDD%
    \let\figptcirc=\figptcircDD%
    \let\figptcircumcenter=\figptcircumcenterDD%
    \let\figptcopy=\figptcopyDD%
    \let\figptcurvcenter=\figptcurvcenterDD%
    \let\figptell=\figptellDD%
    \let\figptendnormal=\figptendnormalDD%
    \let\figptinterlineplane=\figptinterlineplaneDD%
    \let\figptinterlines=\inters@cDD%
    \let\figptorthocenter=\figptorthocenterDD%
    \let\figptorthoprojline=\figptorthoprojlineDD%
    \let\figptorthoprojplane=\figptorthoprojplaneDD%
    \let\figptrot=\figptrotDD%
    \let\figptscontrol=\figptscontrolDD%
    \let\figptsintercirc=\figptsintercircDD%
    \let\figptsinterlinell=\figptsinterlinellDD%
    \let\figptsorthoprojline=\figptsorthoprojlineDD%
    \let\figptorthoprojplane=\figptorthoprojplaneDD%
    \let\figptsrot=\figptsrotDD%
    \let\figptssym=\figptssymDD%
    \let\figptstra=\figptstraDD%
    \let\figptsym=\figptsymDD%
    \let\figpttraC=\figpttraCDD%
    \let\figpttra=\figpttraDD%
    \let\figptvisilimSL=\figptvisilimSLDD%
    \let\figsetobdist=\figsetobdistDD%
    \let\figsettarget=\figsettargetDD%
    \let\figsetview=\figsetviewDD%
    \let\figvectDBezier=\figvectDBezierDD%
    \let\figvectN=\figvectNDD%
    \let\figvectNV=\figvectNVDD%
    \let\figvectP=\figvectPDD%
    \let\figvectU=\figvectUDD%
    \let\figdrawarccircP=\Q@arccircPDD%
    \let\figdrawarccirc=\Q@arccircDD%
    \let\figdrawarcell=\Q@arcellDD%
    \let\figdrawarcellPA=\Q@arcellPADD%
    \let\figdrawarrowBezier=\Q@arrowBezierDD%
    \let\figdrawarrowcircP=\Q@arrowcircPDD%
    \let\figdrawarrowcirc=\Q@arrowcircDD%
    \let\figdrawarrowhead=\Q@arrowheadDD%
    \let\figdrawarrow=\Q@arrowDD%
    \let\figdrawBezier=\Q@BezierDD%
    \let\figdrawcirc=\Q@circDD%
    \let\figdrawcurve=\Q@curveDD%
    \let\figdrawnormal=\Q@normalDD%
    }
\ctr@ld@f\def\initTD@{\Tr@isDimtrue\initb@undb@xTD\newt@rgetptfalse\newdis@bfalse%
    \let\c@lDCUn=\c@lDCUnTD%
    \let\c@lDCDeux=\c@lDCDeuxTD%
    \let\c@ldefproj=\c@ldefprojTD%
    \let\c@lproscal=\c@lproscalTD%
    \let\extr@ctC=\extr@ctCTD%
    \let\extr@ctCa=\extr@ctCaTD%
    \let\extr@ctCF=\extr@ctCFTD%
    \let\Figp@intreg=\Figp@intregTD%
    \let\Figpts@xes=\Figpts@xesTD%
    \let\getredf@ctB=\getredf@ctBTD%
    \let\n@rmeucSV=\n@rmeucSVTD\let\n@rmeuc=\n@rmeucTD\let\n@rmeucC\n@rmeucCTD\let\n@rminf=\n@rminfTD%
    \let\pr@dMatV=\pr@dMatVTD%
    \let\Q@@xes=\Q@@xesTD%
    \let\vecunit@=\vecunit@TD%
    \let\figcoord=\figcoordTD%
    \let\figgetangle=\figgetangleTD%
    \let\figpt=\figptTD%
    \let\figptBezier=\figptBezierTD%
    \let\figptbary=\figptbaryTD%
    \let\figptcirc=\figptcircTD%
    \let\figptcircumcenter=\figptcircumcenterTD%
    \let\figptcopy=\figptcopyTD%
    \let\figptcurvcenter=\figptcurvcenterTD%
    \let\figptinterlineplane=\figptinterlineplaneTD%
    \let\figptinterlines=\inters@cTD%
    \let\figptorthocenter=\figptorthocenterTD%
    \let\figptorthoprojline=\figptorthoprojlineTD%
    \let\figptorthoprojplane=\figptorthoprojplaneTD%
    \let\figptrot=\figptrotTD%
    \let\figptscontrol=\figptscontrolTD%
    \let\figptsintercirc=\figptsintercircTD%
    \let\figptsorthoprojline=\figptsorthoprojlineTD%
    \let\figptsorthoprojplane=\figptsorthoprojplaneTD%
    \let\figptsrot=\figptsrotTD%
    \let\figptssym=\figptssymTD%
    \let\figptstra=\figptstraTD%
    \let\figptsym=\figptsymTD%
    \let\figpttraC=\figpttraCTD%
    \let\figpttra=\figpttraTD%
    \let\figptvisilimSL=\figptvisilimSLTD%
    \let\figsetobdist=\figsetobdistTD%
    \let\figsettarget=\figsettargetTD%
    \let\figsetview=\figsetviewTD%
    \let\figvectDBezier=\figvectDBezierTD%
    \let\figvectN=\figvectNTD%
    \let\figvectNV=\figvectNVTD%
    \let\figvectP=\figvectPTD%
    \let\figvectU=\figvectUTD%
    \let\figdrawarccircP=\Q@arccircPTD%
    \let\figdrawarccirc=\Q@arccircTD%
    \let\figdrawarcell=\Q@arcellTD%
    \let\figdrawarcellPA=\Q@arcellPATD%
    \let\figdrawarrowBezier=\Q@arrowBezierTD%
    \let\figdrawarrowcircP=\Q@arrowcircPTD%
    \let\figdrawarrowcirc=\Q@arrowcircTD%
    \let\figdrawarrowhead=\Q@arrowheadTD%
    \let\figdrawarrow=\Q@arrowTD%
    \let\figdrawBezier=\Q@BezierTD%
    \let\figdrawcirc=\Q@circTD%
    \let\figdrawcurve=\Q@curveTD%
    }
\ctr@ld@f\def\un@v@ilable#1{\immediate\write16{*** The macro #1 is not available in the current context.}}
\ctr@ld@f\def\figinsert#1{{\def\t@xt@{#1}\relax%
    \ifx\t@xt@\empty\ifnum\@utoFInDone>\z@\Figinsert@\DefGIfilen@me,:\fi%
    \else\expandafter\FiginsertNu@#1 :\fi}\ignorespaces}
\ctr@ld@f\def\FiginsertNu@#1 #2:{\def\t@xt@{#1}\relax\ifx\t@xt@\empty\def\t@xt@{#2}%
    \ifx\t@xt@\empty\ifnum\@utoFInDone>\z@\Figinsert@\DefGIfilen@me,:\fi%
    \else\FiginsertNu@#2:\fi\else\expandafter\FiginsertNd@#1 #2:\fi}
\ctr@ld@f\def\FiginsertNd@#1#2:{\ifcat#1a\Figinsert@#1#2,:\else%
    \ifnum\@utoFInDone>\z@\Figinsert@\DefGIfilen@me,#1#2,:\fi\fi}
\ctr@ln@m\Sc@leFact
\ctr@ld@f\def\Figinsert@#1,#2:{\def\t@xt@{#2}\ifx\t@xt@\empty\xdef\Sc@leFact{1}\else%
    \X@rgdeux@#2\xdef\Sc@leFact{\@rgdeux}\fi%
    \Figdisc@rdLTS{#1}{\t@xt@}\@psfgetbb{\t@xt@}%
    \v@lX=\@psfllx\p@\v@lX=\ptpsT@pt\v@lX\v@lX=\Sc@leFact\v@lX%
    \v@lY=\@psflly\p@\v@lY=\ptpsT@pt\v@lY\v@lY=\Sc@leFact\v@lY%
    \b@undb@x{\v@lX}{\v@lY}%
    \v@lX=\@psfurx\p@\v@lX=\ptpsT@pt\v@lX\v@lX=\Sc@leFact\v@lX%
    \v@lY=\@psfury\p@\v@lY=\ptpsT@pt\v@lY\v@lY=\Sc@leFact\v@lY%
    \b@undb@x{\v@lX}{\v@lY}%
    \ifPDFm@ke\Figinclud@PDF{\t@xt@}{\Sc@leFact}\else%
    \v@lX=\c@nt pt\v@lX=\Sc@leFact\v@lX\edef\F@ct{\repdecn@mb{\v@lX}}%
    \ifx\TeXturesonMacOSltX\special{postscriptfile #1 vscale=\F@ct\space hscale=\F@ct}%
    \else\includegraphics{#1}\fi\fi%
    \message{[\t@xt@]}\ignorespaces}
\ctr@ld@f\def\Figdisc@rdLTS#1#2{\expandafter\Figdisc@rdLTS@#1 :#2}
\ctr@ld@f\def\Figdisc@rdLTS@#1 #2:#3{\def#3{#1}\relax\ifx#3\empty\expandafter\Figdisc@rdLTS@#2:#3\fi}
\ctr@ld@f\def\figinsertE#1{\FiginsertE@#1,:\ignorespaces}
\ctr@ld@f\def\FiginsertE@#1,#2:{{\def\t@xt@{#2}\ifx\t@xt@\empty\xdef\Sc@leFact{1}\else%
    \X@rgdeux@#2\xdef\Sc@leFact{\@rgdeux}\fi%
    \Figdisc@rdLTS{#1}{\t@xt@}\pdfximage{\t@xt@}%
    \setbox\Gb@x=\hbox{\pdfrefximage\pdflastximage}%
    \v@lX=\z@\v@lY=-\Sc@leFact\dp\Gb@x\b@undb@x{\v@lX}{\v@lY}%
    \advance\v@lX\Sc@leFact\wd\Gb@x\advance\v@lY\Sc@leFact\dp\Gb@x%
    \advance\v@lY\Sc@leFact\ht\Gb@x\b@undb@x{\v@lX}{\v@lY}%
    \v@lX=\Sc@leFact\wd\Gb@x\pdfximage width \v@lX {\t@xt@}%
    \rlap{\pdfrefximage\pdflastximage}\message{[\t@xt@]}}\ignorespaces}
\ctr@ld@f\def\X@rgdeux@#1,{\edef\@rgdeux{#1}}
\ctr@ln@m\figpt
\ctr@ld@f\def\figptDD#1:#2(#3,#4){\ifGR@cri\c@ntr@lnum{#1}%
    {\v@lX=#3\unit@\v@lY=#4\unit@\Fig@dmpt{#2}{\z@}}\ignorespaces\fi}
\ctr@ld@f\def\Fig@dmpt#1#2{\def\t@xt@{#1}\ifx\t@xt@\empty\def\B@@ltxt{\z@}%
    \else\expandafter\gdef\csname\objc@de T\endcsname{#1}\def\B@@ltxt{\@ne}\fi%
    \expandafter\xdef\csname\objc@de\endcsname{\ifitis@vect@r\C@dCl@svect%
    \else\C@dCl@spt\fi,\z@,\B@@ltxt/\the\v@lX,\the\v@lY,#2}}
\ctr@ld@f\def\C@dCl@spt{P}
\ctr@ld@f\def\C@dCl@svect{V}
\ctr@ln@m\c@@rdYZ
\ctr@ln@m\c@@rdY
\ctr@ld@f\def\figptTD#1:#2(#3,#4){\ifGR@cri\c@ntr@lnum{#1}%
    \def\c@@rdYZ{#4,0,0}\extrairelepremi@r\c@@rdY\de\c@@rdYZ%
    \extrairelepremi@r\c@@rdZ\de\c@@rdYZ%
    {\v@lX=#3\unit@\v@lY=\c@@rdY\unit@\v@lZ=\c@@rdZ\unit@\Fig@dmpt{#2}{\the\v@lZ}%
    \b@undb@xTD{\v@lX}{\v@lY}{\v@lZ}}\ignorespaces\fi}
\ctr@ln@m\Figp@intreg
\ctr@ld@f\def\Figp@intregDD#1:#2(#3,#4){\c@ntr@lnum{#1}%
    {\result@t=#4\v@lX=#3\v@lY=\result@t\Fig@dmpt{#2}{\z@}}\ignorespaces}
\ctr@ld@f\def\Figp@intregTD#1:#2(#3,#4){\c@ntr@lnum{#1}%
    \def\c@@rdYZ{#4,\z@,\z@}\extrairelepremi@r\c@@rdY\de\c@@rdYZ%
    \extrairelepremi@r\c@@rdZ\de\c@@rdYZ%
    {\v@lX=#3\v@lY=\c@@rdY\v@lZ=\c@@rdZ\Fig@dmpt{#2}{\the\v@lZ}%
    \b@undb@xTD{\v@lX}{\v@lY}{\v@lZ}}\ignorespaces}
\ctr@ln@m\figptBezier
\ctr@ld@f\def\figptBezierDD#1:#2:#3[#4,#5,#6,#7]{\ifGR@cri{\s@uvc@ntr@l\et@tfigptBezierDD%
    \FigptBezier@#3[#4,#5,#6,#7]\Figp@intregDD#1:{#2}(\v@lX,\v@lY)%
    \resetc@ntr@l\et@tfigptBezierDD}\ignorespaces\fi}
\ctr@ld@f\def\figptBezierTD#1:#2:#3[#4,#5,#6,#7]{\ifGR@cri{\s@uvc@ntr@l\et@tfigptBezierTD%
    \FigptBezier@#3[#4,#5,#6,#7]\Figp@intregTD#1:{#2}(\v@lX,\v@lY,\v@lZ)%
    \resetc@ntr@l\et@tfigptBezierTD}\ignorespaces\fi}
\ctr@ld@f\def\FigptBezier@#1[#2,#3,#4,#5]{\setc@ntr@l{2}%
    \edef\T@{#1}\v@leur=\p@\advance\v@leur-#1pt\edef\UNmT@{\repdecn@mb{\v@leur}}%
    \figptcopy-4:/#2/\figptcopy-3:/#3/\figptcopy-2:/#4/\figptcopy-1:/#5/%
    \l@mbd@un=-4 \l@mbd@de=-\thr@@\p@rtent=\m@ne\c@lDecast%
    \l@mbd@un=-4 \l@mbd@de=-\thr@@\p@rtent=-\tw@\c@lDecast%
    \l@mbd@un=-4 \l@mbd@de=-\thr@@\p@rtent=-\thr@@\c@lDecast\Figg@tXY{-4}}
\ctr@ln@m\c@lDCUn
\ctr@ld@f\def\c@lDCUnDD#1#2{\Figg@tXY{#1}\v@lX=\UNmT@\v@lX\v@lY=\UNmT@\v@lY%
    \Figg@tXYa{#2}\advance\v@lX\T@\v@lXa\advance\v@lY\T@\v@lYa%
    \Figp@intregDD#1:(\v@lX,\v@lY)}
\ctr@ld@f\def\c@lDCUnTD#1#2{\Figg@tXY{#1}\v@lX=\UNmT@\v@lX\v@lY=\UNmT@\v@lY\v@lZ=\UNmT@\v@lZ%
    \Figg@tXYa{#2}\advance\v@lX\T@\v@lXa\advance\v@lY\T@\v@lYa\advance\v@lZ\T@\v@lZa%
    \Figp@intregTD#1:(\v@lX,\v@lY,\v@lZ)}
\ctr@ld@f\def\c@lDecast{\relax\ifnum\l@mbd@un<\p@rtent\c@lDCUn{\l@mbd@un}{\l@mbd@de}%
    \advance\l@mbd@un\@ne\advance\l@mbd@de\@ne\c@lDecast\fi}
\ctr@ld@f\def\figptmap#1:#2=#3/#4/#5/{\ifGR@cri{\s@uvc@ntr@l\et@tfigptmap%
    \setc@ntr@l{2}\figvectP-1[#4,#3]\Figg@tXY{-1}%
    \pr@dMatV/#5/\figpttra#1:{#2}=#4/1,-1/%
    \resetc@ntr@l\et@tfigptmap}\ignorespaces\fi}
\ctr@ln@m\pr@dMatV
\ctr@ld@f\def\pr@dMatVDD/#1,#2;#3,#4/{\v@lXa=#1\v@lX\advance\v@lXa#2\v@lY%
    \v@lYa=#3\v@lX\advance\v@lYa#4\v@lY\Figv@ctCreg-1(\v@lXa,\v@lYa)}
\ctr@ld@f\def\pr@dMatVTD/#1,#2,#3;#4,#5,#6;#7,#8,#9/{%
    \v@lXa=#1\v@lX\advance\v@lXa#2\v@lY\advance\v@lXa#3\v@lZ%
    \v@lYa=#4\v@lX\advance\v@lYa#5\v@lY\advance\v@lYa#6\v@lZ%
    \v@lZa=#7\v@lX\advance\v@lZa#8\v@lY\advance\v@lZa#9\v@lZ%
    \Figv@ctCreg-1(\v@lXa,\v@lYa,\v@lZa)}
\ctr@ln@m\figptbary
\ctr@ld@f\def\figptbaryDD#1:#2[#3;#4]{\ifGR@cri{\edef\list@num{#3}\extrairelepremi@r\p@int\de\list@num%
    \s@mme=\z@\@ecfor\c@ef:=#4\do{\advance\s@mme\c@ef}%
    \edef\listec@ef{#4,0}\extrairelepremi@r\c@ef\de\listec@ef%
    \Figg@tXY{\p@int}\divide\v@lX\s@mme\divide\v@lY\s@mme%
    \multiply\v@lX\c@ef\multiply\v@lY\c@ef%
    \@ecfor\p@int:=\list@num\do{\extrairelepremi@r\c@ef\de\listec@ef%
           \Figg@tXYa{\p@int}\divide\v@lXa\s@mme\divide\v@lYa\s@mme%
           \multiply\v@lXa\c@ef\multiply\v@lYa\c@ef%
           \advance\v@lX\v@lXa\advance\v@lY\v@lYa}%
    \Figp@intregDD#1:{#2}(\v@lX,\v@lY)}\ignorespaces\fi}
\ctr@ld@f\def\figptbaryTD#1:#2[#3;#4]{\ifGR@cri{\edef\list@num{#3}\extrairelepremi@r\p@int\de\list@num%
    \s@mme=\z@\@ecfor\c@ef:=#4\do{\advance\s@mme\c@ef}%
    \edef\listec@ef{#4,0}\extrairelepremi@r\c@ef\de\listec@ef%
    \Figg@tXY{\p@int}\divide\v@lX\s@mme\divide\v@lY\s@mme\divide\v@lZ\s@mme%
    \multiply\v@lX\c@ef\multiply\v@lY\c@ef\multiply\v@lZ\c@ef%
    \@ecfor\p@int:=\list@num\do{\extrairelepremi@r\c@ef\de\listec@ef%
           \Figg@tXYa{\p@int}\divide\v@lXa\s@mme\divide\v@lYa\s@mme\divide\v@lZa\s@mme%
           \multiply\v@lXa\c@ef\multiply\v@lYa\c@ef\multiply\v@lZa\c@ef%
           \advance\v@lX\v@lXa\advance\v@lY\v@lYa\advance\v@lZ\v@lZa}%
    \Figp@intregTD#1:{#2}(\v@lX,\v@lY,\v@lZ)}\ignorespaces\fi}
\ctr@ld@f\def\figptbaryR#1:#2[#3;#4]{\ifGR@cri{%
    \v@leur=\z@\@ecfor\c@ef:=#4\do{\maxim@m{\v@lmax}{\c@ef pt}{-\c@ef pt}%
    \ifdim\v@lmax>\v@leur\v@leur=\v@lmax\fi}%
    \ifdim\v@leur<\p@\f@ctech=\@M\else\ifdim\v@leur<\t@n\p@\f@ctech=\@m\else%
    \ifdim\v@leur<\c@nt\p@\f@ctech=\c@nt\else\ifdim\v@leur<\@m\p@\f@ctech=\t@n\else%
    \f@ctech=\@ne\fi\fi\fi\fi%
    \def\listec@ef{0}%
    \@ecfor\c@ef:=#4\do{\sc@lec@nvRI{\c@ef pt}\edef\listec@ef{\listec@ef,\the\s@mme}}%
    \extrairelepremi@r\c@ef\de\listec@ef\figptbary#1:#2[#3;\listec@ef]}\ignorespaces\fi}
\ctr@ld@f\def\sc@lec@nvRI#1{\v@leur=#1\p@rtentiere{\s@mme}{\v@leur}\advance\v@leur-\s@mme\p@%
    \multiply\v@leur\f@ctech\p@rtentiere{\p@rtent}{\v@leur}%
    \multiply\s@mme\f@ctech\advance\s@mme\p@rtent}
\ctr@ln@m\figptcirc
\ctr@ld@f\def\figptcircDD#1:#2:#3;#4(#5){\ifGR@cri{\s@uvc@ntr@l\et@tfigptcircDD%
    \c@lptellDD#1:{#2}:#3;#4,#4(#5)\resetc@ntr@l\et@tfigptcircDD}\ignorespaces\fi}
\ctr@ld@f\def\figptcircTD#1:#2:#3,#4,#5;#6(#7){\ifGR@cri{\s@uvc@ntr@l\et@tfigptcircTD%
    \setc@ntr@l{2}\c@lExtAxes#3,#4,#5(#6)\figptellP#1:{#2}:#3,-4,-5(#7)%
    \resetc@ntr@l\et@tfigptcircTD}\ignorespaces\fi}
\ctr@ln@m\figptcircumcenter
\ctr@ld@f\def\figptcircumcenterDD#1:#2[#3,#4,#5]{\ifGR@cri{\s@uvc@ntr@l\et@tfigptcircumcenterDD%
    \setc@ntr@l{2}\figvectNDD-5[#3,#4]\figptbaryDD-3:[#3,#4;1,1]%
                  \figvectNDD-6[#4,#5]\figptbaryDD-4:[#4,#5;1,1]%
    \resetc@ntr@l{2}\inters@cDD#1:{#2}[-3,-5;-4,-6]%
    \resetc@ntr@l\et@tfigptcircumcenterDD}\ignorespaces\fi}
\ctr@ld@f\def\figptcircumcenterTD#1:#2[#3,#4,#5]{\ifGR@cri{\s@uvc@ntr@l\et@tfigptcircumcenterTD%
    \setc@ntr@l{2}\figvectNTD-1[#3,#4,#5]%
    \figvectPTD-3[#3,#4]\figvectNVTD-5[-1,-3]\figptbaryTD-3:[#3,#4;1,1]%
    \figvectPTD-4[#4,#5]\figvectNVTD-6[-1,-4]\figptbaryTD-4:[#4,#5;1,1]%
    \resetc@ntr@l{2}\inters@cTD#1:{#2}[-3,-5;-4,-6]%
    \resetc@ntr@l\et@tfigptcircumcenterTD}\ignorespaces\fi}
\ctr@ln@m\figptcopy
\ctr@ld@f\def\figptcopyDD#1:#2/#3/{\ifGR@cri{\Figg@tXY{#3}%
    \Figp@intregDD#1:{#2}(\v@lX,\v@lY)}\ignorespaces\fi}
\ctr@ld@f\def\figptcopyTD#1:#2/#3/{\ifGR@cri{\Figg@tXY{#3}%
    \Figp@intregTD#1:{#2}(\v@lX,\v@lY,\v@lZ)}\ignorespaces\fi}
\ctr@ln@m\figptcurvcenter
\ctr@ld@f\def\figptcurvcenterDD#1:#2:#3[#4,#5,#6,#7]{\ifGR@cri{\s@uvc@ntr@l\et@tfigptcurvcenterDD%
    \setc@ntr@l{2}\c@lcurvradDD#3[#4,#5,#6,#7]\edef\Sprim@{\repdecn@mb{\result@t}}%
    \figptBezierDD-1::#3[#4,#5,#6,#7]\figpttraDD#1:{#2}=-1/\Sprim@,-5/%
    \resetc@ntr@l\et@tfigptcurvcenterDD}\ignorespaces\fi}
\ctr@ld@f\def\figptcurvcenterTD#1:#2:#3[#4,#5,#6,#7]{\ifGR@cri{\s@uvc@ntr@l\et@tfigptcurvcenterTD%
    \setc@ntr@l{2}\figvectDBezierTD -5:1,#3[#4,#5,#6,#7]%
    \figvectDBezierTD -6:2,#3[#4,#5,#6,#7]\vecunit@TD{-5}{-5}%
    \edef\Sprim@{\repdecn@mb{\result@t}}\figvectNVTD-1[-6,-5]%
    \figvectNVTD-5[-5,-1]\c@lproscalTD\v@leur[-6,-5]%
    \invers@{\v@leur}{\v@leur}\v@leur=\Sprim@\v@leur\v@leur=\Sprim@\v@leur%
    \figptBezierTD-1::#3[#4,#5,#6,#7]\edef\Sprim@{\repdecn@mb{\v@leur}}%
    \figpttraTD#1:{#2}=-1/\Sprim@,-5/\resetc@ntr@l\et@tfigptcurvcenterTD}\ignorespaces\fi}
\ctr@ld@f\def\c@lcurvradDD#1[#2,#3,#4,#5]{{\figvectDBezierDD -5:1,#1[#2,#3,#4,#5]%
    \figvectDBezierDD -6:2,#1[#2,#3,#4,#5]\vecunit@DD{-5}{-5}%
    \edef\Sprim@{\repdecn@mb{\result@t}}\figvectNVDD-5[-5]\c@lproscalDD\v@leur[-6,-5]%
    \invers@{\v@leur}{\v@leur}\v@leur=\Sprim@\v@leur\v@leur=\Sprim@\v@leur%
    \global\result@t=\v@leur}}
\ctr@ln@m\figptell
\ctr@ld@f\def\figptellDD#1:#2:#3;#4,#5(#6,#7){\ifGR@cri{\s@uvc@ntr@l\et@tfigptell%
    \c@lptellDD#1::#3;#4,#5(#6)\figptrotDD#1:{#2}=#1/#3,#7/%
    \resetc@ntr@l\et@tfigptell}\ignorespaces\fi}
\ctr@ld@f\def\c@lptellDD#1:#2:#3;#4,#5(#6){\c@ssin{\C@}{\S@}{#6}\v@lmin=\C@ pt\v@lmax=\S@ pt%
    \v@lmin=#4\v@lmin\v@lmax=#5\v@lmax%
    \edef\Xc@mp{\repdecn@mb{\v@lmin}}\edef\Yc@mp{\repdecn@mb{\v@lmax}}%
    \setc@ntr@l{2}\figvectC-1(\Xc@mp,\Yc@mp)\figpttraDD#1:{#2}=#3/1,-1/}
\ctr@ld@f\def\figptellP#1:#2:#3,#4,#5(#6){\ifGR@cri{\s@uvc@ntr@l\et@tfigptellP%
    \setc@ntr@l{2}\figvectP-1[#3,#4]\figvectP-2[#3,#5]%
    \v@leur=#6pt\c@lptellP{#3}{-1}{-2}\figptcopy#1:{#2}/-3/%
    \resetc@ntr@l\et@tfigptellP}\ignorespaces\fi}
\ctr@ln@m\@ngle
\ctr@ld@f\def\c@lptellP#1#2#3{\edef\@ngle{\repdecn@mb\v@leur}\c@ssin{\C@}{\S@}{\@ngle}%
    \figpttra-3:=#1/\C@,#2/\figpttra-3:=-3/\S@,#3/}
\ctr@ln@m\figptendnormal
\ctr@ld@f\def\figptendnormalDD#1:#2:#3,#4[#5,#6]{\ifGR@cri{\s@uvc@ntr@l\et@tfigptendnormal%
    \Figg@tXYa{#5}\Figg@tXY{#6}%
    \advance\v@lX-\v@lXa\advance\v@lY-\v@lYa%
    \setc@ntr@l{2}\Figv@ctCreg-1(\v@lX,\v@lY)\vecunit@{-1}{-1}\Figg@tXY{-1}%
    \delt@=#3\unit@\maxim@m{\delt@}{\delt@}{-\delt@}\edef\l@ngueur{\repdecn@mb{\delt@}}%
    \v@lX=\l@ngueur\v@lX\v@lY=\l@ngueur\v@lY%
    \delt@=\p@\advance\delt@-#4pt\edef\l@ngueur{\repdecn@mb{\delt@}}%
    \figptbaryR-1:[#5,#6;#4,\l@ngueur]\Figg@tXYa{-1}%
    \advance\v@lXa\v@lY\advance\v@lYa-\v@lX%
    \setc@ntr@l{1}\Figp@intregDD#1:{#2}(\v@lXa,\v@lYa)\resetc@ntr@l\et@tfigptendnormal}%
    \ignorespaces\fi}
\ctr@ld@f\def\figptexcenter#1:#2[#3,#4,#5]{\ifGR@cri{\let@xte={-}
    \Figptexinsc@nter#1:#2[#3,#4,#5]}\ignorespaces\fi}
\ctr@ld@f\def\figptincenter#1:#2[#3,#4,#5]{\ifGR@cri{\let@xte={}
    \Figptexinsc@nter#1:#2[#3,#4,#5]}\ignorespaces\fi}
\ctr@ld@f
\ctr@ld@f\def\Figptexinsc@nter#1:#2[#3,#4,#5]{%
    \figgetdist\LA@[#4,#5]\figgetdist\LB@[#3,#5]\figgetdist\LC@[#3,#4]%
    \figptbaryR#1:{#2}[#3,#4,#5;\the\let@xte\LA@,\LB@,\LC@]}
\ctr@ln@m\figptinterlineplane
\ctr@ld@f\def\figptinterlineplaneDD{\un@v@ilable{figptinterlineplane}}
\ctr@ld@f\def\figptinterlineplaneTD#1:#2[#3,#4;#5,#6]{\ifGR@cri{\s@uvc@ntr@l\et@tfigptinterlineplane%
    \setc@ntr@l{2}\figvectPTD-1[#3,#5]\vecunit@TD{-2}{#6}%
    \r@pPSTD\v@leur[-2,-1,#4]\edef\v@lcoef{\repdecn@mb{\v@leur}}%
    \figpttraTD#1:{#2}=#3/\v@lcoef,#4/\resetc@ntr@l\et@tfigptinterlineplane}\ignorespaces\fi}
\ctr@ln@m\figptorthocenter
\ctr@ld@f\def\figptorthocenterDD#1:#2[#3,#4,#5]{\ifGR@cri{\s@uvc@ntr@l\et@tfigptorthocenterDD%
    \setc@ntr@l{2}\figvectNDD-3[#3,#4]\figvectNDD-4[#4,#5]%
    \resetc@ntr@l{2}\inters@cDD#1:{#2}[#5,-3;#3,-4]%
    \resetc@ntr@l\et@tfigptorthocenterDD}\ignorespaces\fi}
\ctr@ld@f\def\figptorthocenterTD#1:#2[#3,#4,#5]{\ifGR@cri{\s@uvc@ntr@l\et@tfigptorthocenterTD%
    \setc@ntr@l{2}\figvectNTD-1[#3,#4,#5]%
    \figvectPTD-2[#3,#4]\figvectNVTD-3[-1,-2]%
    \figvectPTD-2[#4,#5]\figvectNVTD-4[-1,-2]%
    \resetc@ntr@l{2}\inters@cTD#1:{#2}[#5,-3;#3,-4]%
    \resetc@ntr@l\et@tfigptorthocenterTD}\ignorespaces\fi}
\ctr@ln@m\figptorthoprojline
\ctr@ld@f\def\figptorthoprojlineDD#1:#2=#3/#4,#5/{\ifGR@cri{\s@uvc@ntr@l\et@tfigptorthoprojlineDD%
    \setc@ntr@l{2}\figvectPDD-3[#4,#5]\figvectNVDD-4[-3]\resetc@ntr@l{2}%
    \inters@cDD#1:{#2}[#3,-4;#4,-3]\resetc@ntr@l\et@tfigptorthoprojlineDD}\ignorespaces\fi}
\ctr@ld@f\def\figptorthoprojlineTD#1:#2=#3/#4,#5/{\ifGR@cri{\s@uvc@ntr@l\et@tfigptorthoprojlineTD%
    \setc@ntr@l{2}\figvectPTD-1[#4,#3]\figvectPTD-2[#4,#5]\vecunit@TD{-2}{-2}%
    \c@lproscalTD\v@leur[-1,-2]\edef\v@lcoef{\repdecn@mb{\v@leur}}%
    \figpttraTD#1:{#2}=#4/\v@lcoef,-2/\resetc@ntr@l\et@tfigptorthoprojlineTD}\ignorespaces\fi}
\ctr@ln@m\figptorthoprojplane
\ctr@ld@f\def\figptorthoprojplaneDD{\un@v@ilable{figptorthoprojplane}}
\ctr@ld@f\def\figptorthoprojplaneTD#1:#2=#3/#4,#5/{\ifGR@cri{\s@uvc@ntr@l\et@tfigptorthoprojplane%
    \setc@ntr@l{2}\figvectPTD-1[#3,#4]\vecunit@TD{-2}{#5}%
    \c@lproscalTD\v@leur[-1,-2]\edef\v@lcoef{\repdecn@mb{\v@leur}}%
    \figpttraTD#1:{#2}=#3/\v@lcoef,-2/\resetc@ntr@l\et@tfigptorthoprojplane}\ignorespaces\fi}
\ctr@ld@f\def\figpthom#1:#2=#3/#4,#5/{\ifGR@cri{\s@uvc@ntr@l\et@tfigpthom%
    \setc@ntr@l{2}\figvectP-1[#4,#3]\figpttra#1:{#2}=#4/#5,-1/%
    \resetc@ntr@l\et@tfigpthom}\ignorespaces\fi}
\ctr@ld@f\def\figptinv#1:#2=#3/#4,#5/{\ifGR@cri{\s@uvc@ntr@l\et@tfigptinv%
    \setc@ntr@l{2}\figvectP-1[#4,#3]\Figg@tXY{-1}%
    \getredf@ctB\f@ctech\n@rmeucC{\delt@}{-1}%
    \delt@=\ptT@unit@\delt@\delt@=\ptT@unit@\delt@%
    \invers@{\delt@}{\delt@}\multiply\f@ctech\f@ctech\divide\delt@\f@ctech%
    \delt@=#5\delt@\edef\v@lcoef{\repdecn@mb{\delt@}}\figpttra#1:{#2}=#4/\v@lcoef,-1/%
    \resetc@ntr@l\et@tfigptinv}\ignorespaces\fi}
\ctr@ln@m\figptrot
\ctr@ld@f\def\figptrotDD#1:#2=#3/#4,#5/{\ifGR@cri{\s@uvc@ntr@l\et@tfigptrotDD%
    \c@ssin{\C@}{\S@}{#5}\setc@ntr@l{2}\figvectPDD-1[#4,#3]\Figg@tXY{-1}%
    \v@lXa=\C@\v@lX\advance\v@lXa-\S@\v@lY%
    \v@lYa=\S@\v@lX\advance\v@lYa\C@\v@lY%
    \Figv@ctCreg-1(\v@lXa,\v@lYa)\figpttraDD#1:{#2}=#4/1,-1/%
    \resetc@ntr@l\et@tfigptrotDD}\ignorespaces\fi}
\ctr@ld@f\def\figptrotTD#1:#2=#3/#4,#5,#6/{\ifGR@cri{\s@uvc@ntr@l\et@tfigptrotTD%
    \c@ssin{\C@}{\S@}{#5}%
    \setc@ntr@l{2}\figptorthoprojplaneTD-3:=#4/#3,#6/\figvectPTD-2[-3,#3]%
    \n@rmeucTD\v@leur{-2}\ifdim\v@leur<\Cepsil@n\Figg@tXYa{#3}\else%
    \edef\v@lcoef{\repdecn@mb{\v@leur}}\figvectNVTD-1[#6,-2]%
    \Figg@tXYa{-1}\v@lXa=\v@lcoef\v@lXa\v@lYa=\v@lcoef\v@lYa\v@lZa=\v@lcoef\v@lZa%
    \v@lXa=\S@\v@lXa\v@lYa=\S@\v@lYa\v@lZa=\S@\v@lZa\Figg@tXY{-2}%
    \advance\v@lXa\C@\v@lX\advance\v@lYa\C@\v@lY\advance\v@lZa\C@\v@lZ%
    \Figg@tXY{-3}\advance\v@lXa\v@lX\advance\v@lYa\v@lY\advance\v@lZa\v@lZ\fi%
    \Figp@intregTD#1:{#2}(\v@lXa,\v@lYa,\v@lZa)\resetc@ntr@l\et@tfigptrotTD}\ignorespaces\fi}
\ctr@ln@m\figptsym
\ctr@ld@f\def\figptsymDD#1:#2=#3/#4,#5/{\ifGR@cri{\s@uvc@ntr@l\et@tfigptsymDD%
    \resetc@ntr@l{2}\figptorthoprojlineDD-5:=#3/#4,#5/\figvectPDD-2[#3,-5]%
    \figpttraDD#1:{#2}=#3/2,-2/\resetc@ntr@l\et@tfigptsymDD}\ignorespaces\fi}
\ctr@ld@f\def\figptsymTD#1:#2=#3/#4,#5/{\ifGR@cri{\s@uvc@ntr@l\et@tfigptsymTD%
    \resetc@ntr@l{2}\figptorthoprojplaneTD-3:=#3/#4,#5/\figvectPTD-2[#3,-3]%
    \figpttraTD#1:{#2}=#3/2,-2/\resetc@ntr@l\et@tfigptsymTD}\ignorespaces\fi}
\ctr@ln@m\figpttra
\ctr@ld@f\def\figpttraDD#1:#2=#3/#4,#5/{\ifGR@cri{\Figg@tXYa{#5}\v@lXa=#4\v@lXa\v@lYa=#4\v@lYa%
    \Figg@tXY{#3}\advance\v@lX\v@lXa\advance\v@lY\v@lYa%
    \Figp@intregDD#1:{#2}(\v@lX,\v@lY)}\ignorespaces\fi}
\ctr@ld@f\def\figpttraTD#1:#2=#3/#4,#5/{\ifGR@cri{\Figg@tXYa{#5}\v@lXa=#4\v@lXa\v@lYa=#4\v@lYa%
    \v@lZa=#4\v@lZa\Figg@tXY{#3}\advance\v@lX\v@lXa\advance\v@lY\v@lYa%
    \advance\v@lZ\v@lZa\Figp@intregTD#1:{#2}(\v@lX,\v@lY,\v@lZ)}\ignorespaces\fi}
\ctr@ln@m\figpttraC
\ctr@ld@f\def\figpttraCDD#1:#2=#3/#4,#5/{\ifGR@cri{\v@lXa=#4\unit@\v@lYa=#5\unit@%
    \Figg@tXY{#3}\advance\v@lX\v@lXa\advance\v@lY\v@lYa%
    \Figp@intregDD#1:{#2}(\v@lX,\v@lY)}\ignorespaces\fi}
\ctr@ld@f\def\figpttraCTD#1:#2=#3/#4,#5,#6/{\ifGR@cri{\v@lXa=#4\unit@\v@lYa=#5\unit@\v@lZa=#6\unit@%
    \Figg@tXY{#3}\advance\v@lX\v@lXa\advance\v@lY\v@lYa\advance\v@lZ\v@lZa%
    \Figp@intregTD#1:{#2}(\v@lX,\v@lY,\v@lZ)}\ignorespaces\fi}
\ctr@ld@f\def\figptsaxes#1:#2(#3){\ifGR@cri{\an@lys@xes#3,:\ifx\t@xt@\empty%
    \ifTr@isDim\Figpts@xes#1:#2(0,#3,0,#3,0,#3)\else\Figpts@xes#1:#2(0,#3,0,#3)\fi%
    \else\Figpts@xes#1:#2(#3)\fi}\ignorespaces\fi}
\ctr@ln@m\Figpts@xes
\ctr@ld@f\def\Figpts@xesDD#1:#2(#3,#4,#5,#6){%
    \s@mme=#1\figpttraC\the\s@mme:$x$=#2/#4,0/%
    \advance\s@mme\@ne\figpttraC\the\s@mme:$y$=#2/0,#6/}
\ctr@ld@f\def\Figpts@xesTD#1:#2(#3,#4,#5,#6,#7,#8){%
    \s@mme=#1\figpttraC\the\s@mme:$x$=#2/#4,0,0/%
    \advance\s@mme\@ne\figpttraC\the\s@mme:$y$=#2/0,#6,0/%
    \advance\s@mme\@ne\figpttraC\the\s@mme:$z$=#2/0,0,#8/}
\ctr@ld@f\def\figptsmap#1=#2/#3/#4/{\ifGR@cri{\s@uvc@ntr@l\et@tfigptsmap%
    \setc@ntr@l{2}\def\list@num{#2}\s@mme=#1%
    \@ecfor\p@int:=\list@num\do{\figvectP-1[#3,\p@int]\Figg@tXY{-1}%
    \pr@dMatV/#4/\figpttra\the\s@mme:=#3/1,-1/\advance\s@mme\@ne}%
    \resetc@ntr@l\et@tfigptsmap}\ignorespaces\fi}
\ctr@ln@m\figptscontrol
\ctr@ld@f\def\figptscontrolDD#1[#2,#3,#4,#5]{\ifGR@cri{\s@uvc@ntr@l\et@tfigptscontrolDD\setc@ntr@l{2}%
    \v@lX=\z@\v@lY=\z@\Figtr@nptDD{-5}{#2}\Figtr@nptDD{2}{#5}%
    \divide\v@lX\@vi\divide\v@lY\@vi%
    \Figtr@nptDD{3}{#3}\Figtr@nptDD{-1.5}{#4}\Figp@intregDD-1:(\v@lX,\v@lY)%
    \v@lX=\z@\v@lY=\z@\Figtr@nptDD{2}{#2}\Figtr@nptDD{-5}{#5}%
    \divide\v@lX\@vi\divide\v@lY\@vi\Figtr@nptDD{-1.5}{#3}\Figtr@nptDD{3}{#4}%
    \s@mme=#1\advance\s@mme\@ne\Figp@intregDD\the\s@mme:(\v@lX,\v@lY)%
    \figptcopyDD#1:/-1/\resetc@ntr@l\et@tfigptscontrolDD}\ignorespaces\fi}
\ctr@ld@f\def\figptscontrolTD#1[#2,#3,#4,#5]{\ifGR@cri{\s@uvc@ntr@l\et@tfigptscontrolTD\setc@ntr@l{2}%
    \v@lX=\z@\v@lY=\z@\v@lZ=\z@\Figtr@nptTD{-5}{#2}\Figtr@nptTD{2}{#5}%
    \divide\v@lX\@vi\divide\v@lY\@vi\divide\v@lZ\@vi%
    \Figtr@nptTD{3}{#3}\Figtr@nptTD{-1.5}{#4}\Figp@intregTD-1:(\v@lX,\v@lY,\v@lZ)%
    \v@lX=\z@\v@lY=\z@\v@lZ=\z@\Figtr@nptTD{2}{#2}\Figtr@nptTD{-5}{#5}%
    \divide\v@lX\@vi\divide\v@lY\@vi\divide\v@lZ\@vi\Figtr@nptTD{-1.5}{#3}\Figtr@nptTD{3}{#4}%
    \s@mme=#1\advance\s@mme\@ne\Figp@intregTD\the\s@mme:(\v@lX,\v@lY,\v@lZ)%
    \figptcopyTD#1:/-1/\resetc@ntr@l\et@tfigptscontrolTD}\ignorespaces\fi}
\ctr@ld@f\def\Figtr@nptDD#1#2{\Figg@tXYa{#2}\v@lXa=#1\v@lXa\v@lYa=#1\v@lYa%
    \advance\v@lX\v@lXa\advance\v@lY\v@lYa}
\ctr@ld@f\def\Figtr@nptTD#1#2{\Figg@tXYa{#2}\v@lXa=#1\v@lXa\v@lYa=#1\v@lYa\v@lZa=#1\v@lZa%
    \advance\v@lX\v@lXa\advance\v@lY\v@lYa\advance\v@lZ\v@lZa}
\ctr@ld@f\def\figptscontrolcurve#1,#2[#3]{\ifGR@cri{\s@uvc@ntr@l\et@tfigptscontrolcurve%
    \def\list@num{#3}\extrairelepremi@r\Ak@\de\list@num%
    \extrairelepremi@r\Ai@\de\list@num\extrairelepremi@r\Aj@\de\list@num%
    \s@mme=#1\figptcopy\the\s@mme:/\Ai@/%
    \setc@ntr@l{2}\figvectP -1[\Ak@,\Aj@]%
    \@ecfor\Ak@:=\list@num\do{\advance\s@mme\@ne\figpttra\the\s@mme:=\Ai@/\curv@roundness,-1/%
       \figvectP -1[\Ai@,\Ak@]\advance\s@mme\@ne\figpttra\the\s@mme:=\Aj@/-\curv@roundness,-1/%
       \advance\s@mme\@ne\figptcopy\the\s@mme:/\Aj@/%
       \edef\Ai@{\Aj@}\edef\Aj@{\Ak@}}\advance\s@mme-#1\divide\s@mme\thr@@%
       \xdef#2{\the\s@mme}%
    \resetc@ntr@l\et@tfigptscontrolcurve}\ignorespaces\fi}
\ctr@ln@m\figptsintercirc
\ctr@ld@f\def\figptsintercircDD#1[#2,#3;#4,#5]{\ifGR@cri{\s@uvc@ntr@l\et@tfigptsintercircDD%
    \setc@ntr@l{2}\let\c@lNVintc=\c@lNVintcDD\Figptsintercirc@#1[#2,#3;#4,#5]%
    \resetc@ntr@l\et@tfigptsintercircDD}\ignorespaces\fi}
\ctr@ld@f\def\figptsintercircTD#1[#2,#3;#4,#5;#6]{\ifGR@cri{\s@uvc@ntr@l\et@tfigptsintercircTD%
    \setc@ntr@l{2}\let\c@lNVintc=\c@lNVintcTD\vecunitC@TD[#2,#6]%
    \Figv@ctCreg-3(\v@lX,\v@lY,\v@lZ)\Figptsintercirc@#1[#2,#3;#4,#5]%
    \resetc@ntr@l\et@tfigptsintercircTD}\ignorespaces\fi}
\ctr@ld@f\def\Figptsintercirc@#1[#2,#3;#4,#5]{\figvectP-1[#2,#4]%
    \vecunit@{-1}{-1}\delt@=\result@t\f@ctech=\result@tent%
    \s@mme=#1\advance\s@mme\@ne\figptcopy#1:/#2/\figptcopy\the\s@mme:/#4/%
    \ifdim\delt@=\z@\else%
    \v@lmin=#3\unit@\v@lmax=#5\unit@\v@leur=\v@lmin\advance\v@leur\v@lmax%
    \ifdim\v@leur>\delt@%
    \v@leur=\v@lmin\advance\v@leur-\v@lmax\maxim@m{\v@leur}{\v@leur}{-\v@leur}%
    \ifdim\v@leur<\delt@%
    \divide\v@lmin\f@ctech\divide\v@lmax\f@ctech\divide\delt@\f@ctech%
    \v@lmin=\repdecn@mb{\v@lmin}\v@lmin\v@lmax=\repdecn@mb{\v@lmax}\v@lmax%
    \invers@{\v@leur}{\delt@}\advance\v@lmax-\v@lmin%
    \v@lmax=-\repdecn@mb{\v@leur}\v@lmax\advance\delt@\v@lmax\delt@=.5\delt@%
    \v@lmax=\delt@\multiply\v@lmax\f@ctech%
    \edef\t@ille{\repdecn@mb{\v@lmax}}\figpttra-2:=#2/\t@ille,-1/%
    \delt@=\repdecn@mb{\delt@}\delt@\advance\v@lmin-\delt@%
    \sqrt@{\v@leur}{\v@lmin}\multiply\v@leur\f@ctech\edef\t@ille{\repdecn@mb{\v@leur}}%
    \c@lNVintc\figpttra#1:=-2/-\t@ille,-1/\figpttra\the\s@mme:=-2/\t@ille,-1/\fi\fi\fi}
\ctr@ld@f\def\c@lNVintcDD{\Figg@tXY{-1}\Figv@ctCreg-1(-\v@lY,\v@lX)} 
\ctr@ld@f\def\c@lNVintcTD{{\Figg@tXY{-3}\v@lmin=\v@lX\v@lmax=\v@lY\v@leur=\v@lZ%
    \Figg@tXY{-1}\c@lprovec{-3}\vecunit@{-3}{-3}
    \Figg@tXY{-1}\v@lmin=\v@lX\v@lmax=\v@lY%
    \v@leur=\v@lZ\Figg@tXY{-3}\c@lprovec{-1}}} 
\ctr@ln@m\figptsinterlinell
\ctr@ld@f\def\figptsinterlinellDD#1[#2,#3,#4,#5;#6,#7]{\ifGR@cri{\s@uvc@ntr@l\et@tfigptsinterlinellDD%
    \figptcopy#1:/#6/\s@mme=#1\advance\s@mme\@ne\figptcopy\the\s@mme:/#7/%
    \v@lmin=#3\unit@\v@lmax=#4\unit@
    \setc@ntr@l{2}\figptbaryDD-4:[#6,#7;1,1]\figptsrotDD-3=-4,#7/#2,-#5/
    \Figg@tXY{-3}\Figg@tXYa{#2}\advance\v@lX-\v@lXa\advance\v@lY-\v@lYa
    \figvectP-1[-3,-2]\Figg@tXYa{-1}\figvectP-3[-4,#7]\Figptsint@rLE{#1}
    \resetc@ntr@l\et@tfigptsinterlinellDD}\ignorespaces\fi}
\ctr@ld@f\def\figptsinterlinellP#1[#2,#3,#4;#5,#6]{\ifGR@cri{\s@uvc@ntr@l\et@tfigptsinterlinellP%
    \figptcopy#1:/#5/\s@mme=#1\advance\s@mme\@ne\figptcopy\the\s@mme:/#6/\setc@ntr@l{2}%
    \figvectP-1[#2,#3]\vecunit@{-1}{-1}\v@lmin=\result@t
    \figvectP-2[#2,#4]\vecunit@{-2}{-2}\v@lmax=\result@t
    \figptbary-4:[#5,#6;1,1]
    \figvectP-3[#2,-4]\c@lproscal\v@lX[-3,-1]\c@lproscal\v@lY[-3,-2]
    \figvectP-3[-4,#6]\c@lproscal\v@lXa[-3,-1]\c@lproscal\v@lYa[-3,-2]
    \Figptsint@rLE{#1}\resetc@ntr@l\et@tfigptsinterlinellP}\ignorespaces\fi}
\ctr@ld@f\def\Figptsint@rLE#1{%
    \getredf@ctDD\f@ctech(\v@lmin,\v@lmax)%
    \getredf@ctDD\p@rtent(\v@lX,\v@lY)\ifnum\p@rtent>\f@ctech\f@ctech=\p@rtent\fi%
    \getredf@ctDD\p@rtent(\v@lXa,\v@lYa)\ifnum\p@rtent>\f@ctech\f@ctech=\p@rtent\fi%
    \divide\v@lmin\f@ctech\divide\v@lmax\f@ctech\divide\v@lX\f@ctech\divide\v@lY\f@ctech%
    \divide\v@lXa\f@ctech\divide\v@lYa\f@ctech%
    \c@rre=\repdecn@mb\v@lXa\v@lmax\mili@u=\repdecn@mb\v@lYa\v@lmin%
    \getredf@ctDD\f@ctech(\c@rre,\mili@u)%
    \c@rre=\repdecn@mb\v@lX\v@lmax\mili@u=\repdecn@mb\v@lY\v@lmin%
    \getredf@ctDD\p@rtent(\c@rre,\mili@u)\ifnum\p@rtent>\f@ctech\f@ctech=\p@rtent\fi%
    \divide\v@lmin\f@ctech\divide\v@lmax\f@ctech\divide\v@lX\f@ctech\divide\v@lY\f@ctech%
    \divide\v@lXa\f@ctech\divide\v@lYa\f@ctech%
    \v@lmin=\repdecn@mb{\v@lmin}\v@lmin\v@lmax=\repdecn@mb{\v@lmax}\v@lmax%
    \edef\G@xde{\repdecn@mb\v@lmin}\edef\P@xde{\repdecn@mb\v@lmax}%
    \c@rre=-\v@lmax\v@leur=\repdecn@mb\v@lY\v@lY\advance\c@rre\v@leur\c@rre=\G@xde\c@rre%
    \v@leur=\repdecn@mb\v@lX\v@lX\v@leur=\P@xde\v@leur\advance\c@rre\v@leur
    \v@lmin=\repdecn@mb\v@lYa\v@lmin\v@lmax=\repdecn@mb\v@lXa\v@lmax%
    \mili@u=\repdecn@mb\v@lX\v@lmax\advance\mili@u\repdecn@mb\v@lY\v@lmin
    \v@lmax=\repdecn@mb\v@lXa\v@lmax\advance\v@lmax\repdecn@mb\v@lYa\v@lmin
    \ifdim\v@lmax>\epsil@n%
    \maxim@m{\v@leur}{\c@rre}{-\c@rre}\maxim@m{\v@lmin}{\mili@u}{-\mili@u}%
    \maxim@m{\v@leur}{\v@leur}{\v@lmin}\maxim@m{\v@lmin}{\v@lmax}{-\v@lmax}%
    \maxim@m{\v@leur}{\v@leur}{\v@lmin}\p@rtentiere{\p@rtent}{\v@leur}\advance\p@rtent\@ne%
    \divide\c@rre\p@rtent\divide\mili@u\p@rtent\divide\v@lmax\p@rtent%
    \delt@=\repdecn@mb{\mili@u}\mili@u\v@leur=\repdecn@mb{\v@lmax}\c@rre%
    \advance\delt@-\v@leur\ifdim\delt@<\z@\else\sqrt@\delt@\delt@%
    \invers@\v@lmax\v@lmax\edef\Uns@rAp{\repdecn@mb\v@lmax}%
    \v@leur=-\mili@u\advance\v@leur-\delt@\v@leur=\Uns@rAp\v@leur%
    \edef\t@ille{\repdecn@mb\v@leur}\figpttra#1:=-4/\t@ille,-3/\s@mme=#1\advance\s@mme\@ne%
    \v@leur=-\mili@u\advance\v@leur\delt@\v@leur=\Uns@rAp\v@leur%
    \edef\t@ille{\repdecn@mb\v@leur}\figpttra\the\s@mme:=-4/\t@ille,-3/\fi\fi}
\ctr@ln@m\figptsorthoprojline
\ctr@ld@f\def\figptsorthoprojlineDD#1=#2/#3,#4/{\ifGR@cri{\s@uvc@ntr@l\et@tfigptsorthoprojlineDD%
    \setc@ntr@l{2}\figvectPDD-3[#3,#4]\figvectNVDD-4[-3]\resetc@ntr@l{2}%
    \def\list@num{#2}\s@mme=#1\@ecfor\p@int:=\list@num\do{%
    \inters@cDD\the\s@mme:[\p@int,-4;#3,-3]\advance\s@mme\@ne}%
    \resetc@ntr@l\et@tfigptsorthoprojlineDD}\ignorespaces\fi}
\ctr@ld@f\def\figptsorthoprojlineTD#1=#2/#3,#4/{\ifGR@cri{\s@uvc@ntr@l\et@tfigptsorthoprojlineTD%
    \setc@ntr@l{2}\figvectPTD-2[#3,#4]\vecunit@TD{-2}{-2}%
    \def\list@num{#2}\s@mme=#1\@ecfor\p@int:=\list@num\do{%
    \figvectPTD-1[#3,\p@int]\c@lproscalTD\v@leur[-1,-2]%
    \edef\v@lcoef{\repdecn@mb{\v@leur}}\figpttraTD\the\s@mme:=#3/\v@lcoef,-2/%
    \advance\s@mme\@ne}\resetc@ntr@l\et@tfigptsorthoprojlineTD}\ignorespaces\fi}
\ctr@ln@m\figptsorthoprojplane
\ctr@ld@f\def\figptsorthoprojplaneDD{\un@v@ilable{figptsorthoprojplane}}
\ctr@ld@f\def\figptsorthoprojplaneTD#1=#2/#3,#4/{\ifGR@cri{\s@uvc@ntr@l\et@tfigptsorthoprojplane%
    \setc@ntr@l{2}\vecunit@TD{-2}{#4}%
    \def\list@num{#2}\s@mme=#1\@ecfor\p@int:=\list@num\do{\figvectPTD-1[\p@int,#3]%
    \c@lproscalTD\v@leur[-1,-2]\edef\v@lcoef{\repdecn@mb{\v@leur}}%
    \figpttraTD\the\s@mme:=\p@int/\v@lcoef,-2/\advance\s@mme\@ne}%
    \resetc@ntr@l\et@tfigptsorthoprojplane}\ignorespaces\fi}
\ctr@ld@f\def\figptshom#1=#2/#3,#4/{\ifGR@cri{\s@uvc@ntr@l\et@tfigptshom%
    \setc@ntr@l{2}\def\list@num{#2}\s@mme=#1%
    \@ecfor\p@int:=\list@num\do{\figvectP-1[#3,\p@int]%
    \figpttra\the\s@mme:=#3/#4,-1/\advance\s@mme\@ne}%
    \resetc@ntr@l\et@tfigptshom}\ignorespaces\fi}
\ctr@ld@f\def\figptsinv#1=#2/#3,#4/{\ifGR@cri{\s@uvc@ntr@l\et@tfigptsinv%
    \setc@ntr@l{2}\def\list@num{#2}\s@mme=#1%
    \@ecfor\p@int:=\list@num\do{\figvectP-1[#3,\p@int]\Figg@tXY{-1}%
    \getredf@ctB\f@ctech\n@rmeucC{\delt@}{-1}%
    \delt@=\ptT@unit@\delt@\delt@=\ptT@unit@\delt@%
    \invers@{\delt@}{\delt@}\multiply\f@ctech\f@ctech\divide\delt@\f@ctech%
    \delt@=#4\delt@\edef\v@lcoef{\repdecn@mb{\delt@}}\figpttra\the\s@mme:=#3/\v@lcoef,-1/%
    \advance\s@mme\@ne}\resetc@ntr@l\et@tfigptsinv}\ignorespaces\fi}
\ctr@ln@m\figptsrot
\ctr@ld@f\def\figptsrotDD#1=#2/#3,#4/{\ifGR@cri{\s@uvc@ntr@l\et@tfigptsrotDD%
    \c@ssin{\C@}{\S@}{#4}\setc@ntr@l{2}\def\list@num{#2}\s@mme=#1%
    \@ecfor\p@int:=\list@num\do{\figvectPDD-1[#3,\p@int]\Figg@tXY{-1}%
    \v@lXa=\C@\v@lX\advance\v@lXa-\S@\v@lY%
    \v@lYa=\S@\v@lX\advance\v@lYa\C@\v@lY%
    \Figv@ctCreg-1(\v@lXa,\v@lYa)\figpttraDD\the\s@mme:=#3/1,-1/\advance\s@mme\@ne}%
    \resetc@ntr@l\et@tfigptsrotDD}\ignorespaces\fi}
\ctr@ld@f\def\figptsrotTD#1=#2/#3,#4,#5/{\ifGR@cri{\s@uvc@ntr@l\et@tfigptsrotTD%
    \c@ssin{\C@}{\S@}{#4}%
    \setc@ntr@l{2}\def\list@num{#2}\s@mme=#1%
    \@ecfor\p@int:=\list@num\do{\figptorthoprojplaneTD-3:=#3/\p@int,#5/%
    \figvectPTD-2[-3,\p@int]%
    \figvectNVTD-1[#5,-2]\n@rmeucTD\v@leur{-2}\edef\v@lcoef{\repdecn@mb{\v@leur}}%
    \Figg@tXYa{-1}\v@lXa=\v@lcoef\v@lXa\v@lYa=\v@lcoef\v@lYa\v@lZa=\v@lcoef\v@lZa%
    \v@lXa=\S@\v@lXa\v@lYa=\S@\v@lYa\v@lZa=\S@\v@lZa\Figg@tXY{-2}%
    \advance\v@lXa\C@\v@lX\advance\v@lYa\C@\v@lY\advance\v@lZa\C@\v@lZ%
    \Figg@tXY{-3}\advance\v@lXa\v@lX\advance\v@lYa\v@lY\advance\v@lZa\v@lZ%
    \Figp@intregTD\the\s@mme:(\v@lXa,\v@lYa,\v@lZa)\advance\s@mme\@ne}%
    \resetc@ntr@l\et@tfigptsrotTD}\ignorespaces\fi}
\ctr@ln@m\figptssym
\ctr@ld@f\def\figptssymDD#1=#2/#3,#4/{\ifGR@cri{\s@uvc@ntr@l\et@tfigptssymDD%
    \setc@ntr@l{2}\figvectPDD-3[#3,#4]\Figg@tXY{-3}\Figv@ctCreg-4(-\v@lY,\v@lX)%
    \resetc@ntr@l{2}\def\list@num{#2}\s@mme=#1%
    \@ecfor\p@int:=\list@num\do{\inters@cDD-5:[#3,-3;\p@int,-4]\figvectPDD-2[\p@int,-5]%
    \figpttraDD\the\s@mme:=\p@int/2,-2/\advance\s@mme\@ne}%
    \resetc@ntr@l\et@tfigptssymDD}\ignorespaces\fi}
\ctr@ld@f\def\figptssymTD#1=#2/#3,#4/{\ifGR@cri{\s@uvc@ntr@l\et@tfigptssymTD%
    \setc@ntr@l{2}\vecunit@TD{-2}{#4}\def\list@num{#2}\s@mme=#1%
    \@ecfor\p@int:=\list@num\do{\figvectPTD-1[\p@int,#3]%
    \c@lproscalTD\v@leur[-1,-2]\v@leur=2\v@leur\edef\v@lcoef{\repdecn@mb{\v@leur}}%
    \figpttraTD\the\s@mme:=\p@int/\v@lcoef,-2/\advance\s@mme\@ne}%
    \resetc@ntr@l\et@tfigptssymTD}\ignorespaces\fi}
\ctr@ln@m\figptstra
\ctr@ld@f\def\figptstraDD#1=#2/#3,#4/{\ifGR@cri{\Figg@tXYa{#4}\v@lXa=#3\v@lXa\v@lYa=#3\v@lYa%
    \def\list@num{#2}\s@mme=#1\@ecfor\p@int:=\list@num\do{\Figg@tXY{\p@int}%
    \advance\v@lX\v@lXa\advance\v@lY\v@lYa%
    \Figp@intregDD\the\s@mme:(\v@lX,\v@lY)\advance\s@mme\@ne}}\ignorespaces\fi}
\ctr@ld@f\def\figptstraTD#1=#2/#3,#4/{\ifGR@cri{\Figg@tXYa{#4}\v@lXa=#3\v@lXa\v@lYa=#3\v@lYa%
    \v@lZa=#3\v@lZa\def\list@num{#2}\s@mme=#1\@ecfor\p@int:=\list@num\do{\Figg@tXY{\p@int}%
    \advance\v@lX\v@lXa\advance\v@lY\v@lYa\advance\v@lZ\v@lZa%
    \Figp@intregTD\the\s@mme:(\v@lX,\v@lY,\v@lZ)\advance\s@mme\@ne}}\ignorespaces\fi}
\ctr@ln@m\figptvisilimSL
\ctr@ld@f\def\figptvisilimSLDD{\un@v@ilable{figptvisilimSL}}
\ctr@ld@f\def\figptvisilimSLTD#1:#2[#3,#4;#5,#6]{\ifGR@cri{\s@uvc@ntr@l\et@tfigptvisilimSLTD%
    \setc@ntr@l{2}\figvectP-1[#3,#4]\n@rminf{\delt@}{-1}%
    \ifcase\CUR@proj\v@lX=\cxa@\p@\v@lY=-\p@\v@lZ=\cxb@\p@
    \Figv@ctCreg-2(\v@lX,\v@lY,\v@lZ)\figvectP-3[#5,#6]\figvectNV-1[-2,-3]%
    \or\figvectP-1[#5,#6]\vecunitCV@TD{-1}\v@lmin=\v@lX\v@lmax=\v@lY
    \v@leur=\v@lZ\v@lX=\cza@\p@\v@lY=\czb@\p@\v@lZ=\czc@\p@\c@lprovec{-1}%
    \or\c@ley@pt{-2}\figvectN-1[#5,#6,-2]\fi
    \edef\Ai@{#3}\edef\Aj@{#4}\figvectP-2[#5,\Ai@]\c@lproscal\v@leur[-1,-2]%
    \ifdim\v@leur>\z@\p@rtent=\@ne\else\p@rtent=\m@ne\fi%
    \figvectP-2[#5,\Aj@]\c@lproscal\v@leur[-1,-2]%
    \ifdim\p@rtent\v@leur>\z@\figptcopy#1:#2/#3/%
    \message{*** \BS@ figptvisilimSL: points are on the same side.}\else%
    \figptcopy-3:/#3/\figptcopy-4:/#4/%
    \loop\figptbary-5:[-3,-4;1,1]\figvectP-2[#5,-5]\c@lproscal\v@leur[-1,-2]%
    \ifdim\p@rtent\v@leur>\z@\figptcopy-3:/-5/\else\figptcopy-4:/-5/\fi%
    \divide\delt@\tw@\ifdim\delt@>\epsil@n\repeat%
    \figptbary#1:#2[-3,-4;1,1]\fi\resetc@ntr@l\et@tfigptvisilimSLTD}\ignorespaces\fi}
\ctr@ld@f\def\c@ley@pt#1{\t@stp@r\ifitis@K\v@lX=\cza@\p@\v@lY=\czb@\p@\v@lZ=\czc@\p@%
    \Figv@ctCreg-1(\v@lX,\v@lY,\v@lZ)\Figp@intreg-2:(\wd\Bt@rget,\ht\Bt@rget,\dp\Bt@rget)%
    \figpttra#1:=-2/-\disob@intern,-1/\else\end\fi}
\ctr@ld@f\def\t@stp@r{\itis@Ktrue\ifnewt@rgetpt\else\itis@Kfalse%
    \message{*** \BS@ figptvisilimXX: target point undefined.}\fi\ifnewdis@b\else%
    \itis@Kfalse\message{*** \BS@ figptvisilimXX: observation distance undefined.}\fi%
    \ifitis@K\else\message{*** This macro must be called after \BS@ figdrawbegin or after
    having set the missing parameter(s) with \BS@ figset proj()}\fi}
\ctr@ld@f\def\figscan#1(#2,#3){{\s@uvc@ntr@l\et@tfigscan\@psfgetbb{#1}\if@psfbbfound\else%
    \def\@psfllx{0}\def\@psflly{20}\def\@psfurx{540}\def\@psfury{640}\fi\figscan@{#2}{#3}%
    \resetc@ntr@l\et@tfigscan}\ignorespaces}
\ctr@ld@f\def\figscan@#1#2{%
    \unit@=\@ne bp\setc@ntr@l{2}\figsetmark{}%
    \def\minst@p{20pt}%
    \v@lX=\@psfllx\p@\v@lX=\Sc@leFact\v@lX\r@undint\v@lX\v@lX%
    \v@lY=\@psflly\p@\v@lY=\Sc@leFact\v@lY\ifdim\v@lY>\z@\r@undint\v@lY\v@lY\fi%
    \delt@=\@psfury\p@\delt@=\Sc@leFact\delt@%
    \advance\delt@-\v@lY\v@lXa=\@psfurx\p@\v@lXa=\Sc@leFact\v@lXa\v@leur=\minst@p%
    \edef\valv@lY{\repdecn@mb{\v@lY}}\edef\LgTr@it{\the\delt@}%
    \loop\ifdim\v@lX<\v@lXa\edef\valv@lX{\repdecn@mb{\v@lX}}%
    \figptDD -1:(\valv@lX,\valv@lY)\figwriten -1:\hbox{\vrule height\LgTr@it}(0)%
    \ifdim\v@leur<\minst@p\else\figsetmark{\raise-8bp\hbox{$\scriptscriptstyle\triangle$}}%
    \figwrites -1:\@ffichnb{0}{\valv@lX}(6)\v@leur=\z@\figsetmark{}\fi%
    \advance\v@leur#1pt\advance\v@lX#1pt\repeat%
    \def\minst@p{10pt}%
    \v@lX=\@psfllx\p@\v@lX=\Sc@leFact\v@lX\ifdim\v@lX>\z@\r@undint\v@lX\v@lX\fi%
    \v@lY=\@psflly\p@\v@lY=\Sc@leFact\v@lY\r@undint\v@lY\v@lY%
    \delt@=\@psfurx\p@\delt@=\Sc@leFact\delt@%
    \advance\delt@-\v@lX\v@lYa=\@psfury\p@\v@lYa=\Sc@leFact\v@lYa\v@leur=\minst@p%
    \edef\valv@lX{\repdecn@mb{\v@lX}}\edef\LgTr@it{\the\delt@}%
    \loop\ifdim\v@lY<\v@lYa\edef\valv@lY{\repdecn@mb{\v@lY}}%
    \figptDD -1:(\valv@lX,\valv@lY)\figwritee -1:\vbox{\hrule width\LgTr@it}(0)%
    \ifdim\v@leur<\minst@p\else\figsetmark{$\triangleright$\kern4bp}%
    \figwritew -1:\@ffichnb{0}{\valv@lY}(6)\v@leur=\z@\figsetmark{}\fi%
    \advance\v@leur#2pt\advance\v@lY#2pt\repeat}
\ctr@ld@f
\ctr@ld@f\def\figscan@E#1(#2,#3){{\s@uvc@ntr@l\et@tfigscan@E%
    \Figdisc@rdLTS{#1}{\t@xt@}\pdfximage{\t@xt@}%
    \setbox\Gb@x=\hbox{\pdfrefximage\pdflastximage}%
    \edef\@psfllx{0}\v@lY=-\dp\Gb@x\edef\@psflly{\repdecn@mb{\v@lY}}%
    \edef\@psfurx{\repdecn@mb{\wd\Gb@x}}%
    \v@lY=\dp\Gb@x\advance\v@lY\ht\Gb@x\edef\@psfury{\repdecn@mb{\v@lY}}%
    \figscan@{#2}{#3}\resetc@ntr@l\et@tfigscan@E}\ignorespaces}
\ctr@ld@f\def\figshowpts[#1,#2]{{\figsetmark{$\bullet$}\figsetptname{\bf ##1}%
    \p@rtent=#2\relax\ifnum\p@rtent<\z@\p@rtent=\z@\fi%
    \s@mme=#1\relax\ifnum\s@mme<\z@\s@mme=\z@\fi%
    \loop\ifnum\s@mme<\p@rtent\pt@rvect{\s@mme}%
    \ifitis@K\figwriten{\the\s@mme}:(4pt)\fi\advance\s@mme\@ne\repeat%
    \pt@rvect{\s@mme}\ifitis@K\figwriten{\the\s@mme}:(4pt)\fi}\ignorespaces}
\ctr@ld@f\def\pt@rvect#1{\set@bjc@de{#1}%
    \expandafter\expandafter\expandafter\inqpt@rvec\csname\objc@de\endcsname:}
\ctr@ld@f\def\inqpt@rvec#1#2:{\if#1\C@dCl@spt\itis@Ktrue\else\itis@Kfalse\fi}
\ctr@ld@f\def\figshowsettings{{%
    \immediate\write16{====================================================================}%
    \immediate\write16{ Current settings are (DDV means "with dynamic default value"):}%
    \immediate\write16{ --- GENERAL ---}%
    \immediate\write16{Scale factor and Unit = \unit@util\space (\the\unit@)
     \space -> \BS@ figinit{ScaleFactorUnit}}%
    \immediate\write16{Update mode = \ifGRupdatem@de yes\else no\fi
     \space-> \BS@ figset(update=yes/no) or \BS@ figsetdefault(update=yes/no)}%
    \immediate\write16{ --- WRITING ---}%
    \immediate\write16{Implicit point name = \ptn@me{i} \space-> \BS@ figset write(ptname={Name})}%
    \immediate\write16{Point marker = \the\c@nsymb \space -> \BS@ figset write(mark=Mark)}%
    \immediate\write16{Print rounded coordinates = \ifr@undcoord yes\else no\fi
     \space-> \BS@ figset write(roundcoord=yes/no)}%
    \immediate\write16{ --- GRAPHICAL (general) ---}%
    \immediate\write16{Color = \CUR@color \space-> \BS@ figset(color=ColorDefinition)}%
    \immediate\write16{Filling mode = \iffillm@de yes\else no\fi
     \space-> \BS@ figset(fillmode=yes/no)}%
    \immediate\write16{Line join = \CUR@join \space-> \BS@ figset(join=miter/round/bevel)}%
    \immediate\write16{Line style = \CUR@dash \space-> \BS@ figset(dash=Index/Pattern)}%
    \immediate\write16{Line width = \CUR@width
     \space-> \BS@ figset(width=real in PostScript units)}%
    \immediate\write16{ --- GRAPHICAL (specific) ---}%
    \immediate\write16{Altitude (all the following attributes are DDV):}%
    \immediate\write16{ Base line color =
     \ifx\DDV@blcolor\D@FTref general color\else\DDV@blcolor\fi
     \space-> \BS@ figset altitude(blcolor=ColorDefinition)}%
    \immediate\write16{ Base line style =
     \ifx\DDV@bldash\D@FTref general style\else\DDV@bldash\fi
     \space-> \BS@ figset altitude(bldash=Index/Pattern)}%
    \immediate\write16{ Base line width =
     \ifx\DDV@blwidth\D@FTref general width\else\DDV@blwidth\fi
     \space-> \BS@ figset altitude(blwidth=real in PostScript units)}%
    \immediate\write16{ Square line color =
     \ifx\DDV@sqcolor\D@FTref general color\else\DDV@sqcolor\fi
     \space-> \BS@ figset altitude(sqcolor=ColorDefinition)}%
    \immediate\write16{ Square line style =
     \ifx\DDV@sqdash\D@FTref general style\else\DDV@sqdash\fi
     \space-> \BS@ figset altitude(sqdash=Index/Pattern)}%
    \immediate\write16{ Square line width =
     \ifx\DDV@sqwidth\D@FTref general width\else\DDV@sqwidth\fi
     \space-> \BS@ figset altitude(sqwidth=real in PostScript units)}%
    \immediate\write16{Arrowhead:}%
    \immediate\write16{ (half-)Angle = \@rrowheadangle
     \space-> \BS@ figset arrowhead(angle=real in degrees)}%
    \immediate\write16{ Filling mode = \if@rrowhfill yes\else no\fi
     \space-> \BS@ figset arrowhead(fillmode=yes/no)}%
    \immediate\write16{ "Outside" = \if@rrowhout yes\else no\fi
     \space-> \BS@ figset arrowhead(out=yes/no)}%
    \immediate\write16{ Length = \@rrowheadlength
     \if@rrowratio\space(not active)\else\space(active)\fi
     \space-> \BS@ figset arrowhead(length=real in user coord.)}%
    \immediate\write16{ Ratio = \@rrowheadratio
     \if@rrowratio\space(active)\else\space(not active)\fi
     \space-> \BS@ figset arrowhead(ratio=real in [0,1])}%
    \immediate\write16{Curve:}%
    \immediate\write16{ Roundness = \curv@roundness
     \space-> \BS@ figset curve(roundness=real in [0,0.5])}%
    \immediate\write16{Flow chart:}%
    \immediate\write16{ Arrow position = \@rrowp@s
     \space-> \BS@ figset flowchart(arrowposition=real in [0,1])}%
    \immediate\write16{ Arrow reference point = \ifcase\@rrowr@fpt start\else end\fi
     \space-> \BS@ figset flowchart(arrowrefpt = start/end)}%
    \immediate\write16{ Background color = \fcbgc@lor
     \space-> \BS@ figset flowchart(bgcolor=ColorDefinition)}%
    \immediate\write16{ Line type = \ifcase\fclin@typ@ curve\else polygon\fi
     \space-> \BS@ figset flowchart(line=polygon/curve)}%
    \immediate\write16{ Padding = (\Xp@dd, \Yp@dd)
     \space-> \BS@ figset flowchart(padding = real in user coord.)}%
    \immediate\write16{\space\space\space\space(or
     \BS@ figset flowchart(xpadding=real, ypadding=real) )}%
    \immediate\write16{ Radius = \fclin@r@d
     \space-> \BS@ figset flowchart(radius=positive real in user coord.)}%
    \immediate\write16{ Shape = \fcsh@pe
     \space-> \BS@ figset flowchart(shape = rectangle, ellipse or lozenge)}%
    \immediate\write16{ Thickness color (DDV) = 
     \ifx\DDV@thickcolor\D@FTref general color\else\DDV@thickcolor\fi
     \space-> \BS@ figset flowchart(thickcolor=ColorDefinition)}%
    \immediate\write16{ Thickness = \thickn@ss
     \space-> \BS@ figset flowchart(thickness = real in user coord.)}%
    \immediate\write16{Mesh:}%
    \immediate\write16{ Diagonal = \c@ntrolmesh
     \space-> \BS@ figset mesh(diag=integer in {-1,0,1})}%
    \immediate\write16{ Lines color (DDV) =
     \ifx\DDV@meshcolor\D@FTref general color\else\DDV@meshcolor\fi
     \space-> \BS@ figset mesh(color=ColorDefinition)}%
    \immediate\write16{ Lines style (DDV) =
     \ifx\DDV@meshdash\D@FTref general style\else\DDV@meshdash\fi
     \space-> \BS@ figset mesh(dash=Index/Pattern)}%
    \immediate\write16{ Lines width (DDV) =
     \ifx\DDV@meshwidth\D@FTref general width\else\DDV@meshwidth\fi
     \space-> \BS@ figset mesh(width=real in PostScript units)}%
    \immediate\write16{Trimesh:}%
    \immediate\write16{ Lines color (DDV) =
     \ifx\DDV@tmeshcolor\D@FTref general color\else\DDV@tmeshcolor\fi
     \space-> \BS@ figset trimesh(color=ColorDefinition)}%
    \immediate\write16{ Lines style (DDV) =
     \ifx\DDV@tmeshdash\D@FTref general style\else\DDV@tmeshdash\fi
     \space-> \BS@ figset trimesh(dash=Index/Pattern)}%
    \immediate\write16{ Lines width (DDV) =
     \ifx\DDV@tmeshwidth\D@FTref general width\else\DDV@tmeshwidth\fi
     \space-> \BS@ figset trimesh(width=real in PostScript units)}%
    \ifTr@isDim%
    \immediate\write16{ --- 3D to 2D PROJECTION ---}%
    \immediate\write16{Projection : \typ@proj \space-> \BS@ figinit{ScaleFactorUnit, ProjType}}%
    \immediate\write16{Longitude (psi) = \v@lPsi \space-> \BS@ figset proj(psi=real in degrees)}%
    \ifcase\CUR@proj\immediate\write16{Depth coeff. (Lambda)
     \space = \v@lTheta \space-> \BS@ figset proj(lambda=real in [0,1])}%
    \else\immediate\write16{Latitude (theta)
     \space = \v@lTheta \space-> \BS@ figset proj(theta=real in degrees)}%
    \fi%
    \ifnum\CUR@proj=\tw@%
    \immediate\write16{Observation distance = \disob@unit
     \space-> \BS@ figset proj(dist=real in user coord.)}%
    \immediate\write16{Target point = \t@rgetpt \space-> \BS@ figset proj(targetpt=pt number)}%
     \v@lX=\ptT@unit@\wd\Bt@rget\v@lY=\ptT@unit@\ht\Bt@rget\v@lZ=\ptT@unit@\dp\Bt@rget%
    \immediate\write16{ Its coordinates are
     (\repdecn@mb{\v@lX}, \repdecn@mb{\v@lY}, \repdecn@mb{\v@lZ})}%
    \fi%
    \fi%
    \immediate\write16{====================================================================}%
    \ignorespaces}}
\ctr@ln@w{newif}\ifitis@vect@r
\ctr@ld@f\def\figvectC#1(#2,#3){{\itis@vect@rtrue\figpt#1:(#2,#3)}\ignorespaces}
\ctr@ld@f\def\Figv@ctCreg#1(#2,#3){{\itis@vect@rtrue\Figp@intreg#1:(#2,#3)}\ignorespaces}
\ctr@ln@m\figvectDBezier
\ctr@ld@f\def\figvectDBezierDD#1:#2,#3[#4,#5,#6,#7]{\ifGR@cri{\s@uvc@ntr@l\et@tfigvectDBezierDD%
    \FigvectDBezier@#2,#3[#4,#5,#6,#7]\v@lX=\c@ef\v@lX\v@lY=\c@ef\v@lY%
    \Figv@ctCreg#1(\v@lX,\v@lY)\resetc@ntr@l\et@tfigvectDBezierDD}\ignorespaces\fi}
\ctr@ld@f\def\figvectDBezierTD#1:#2,#3[#4,#5,#6,#7]{\ifGR@cri{\s@uvc@ntr@l\et@tfigvectDBezierTD%
    \FigvectDBezier@#2,#3[#4,#5,#6,#7]\v@lX=\c@ef\v@lX\v@lY=\c@ef\v@lY\v@lZ=\c@ef\v@lZ%
    \Figv@ctCreg#1(\v@lX,\v@lY,\v@lZ)\resetc@ntr@l\et@tfigvectDBezierTD}\ignorespaces\fi}
\ctr@ld@f\def\FigvectDBezier@#1,#2[#3,#4,#5,#6]{\setc@ntr@l{2}%
    \edef\T@{#2}\v@leur=\p@\advance\v@leur-#2pt\edef\UNmT@{\repdecn@mb{\v@leur}}%
    \ifnum#1=\tw@\def\c@ef{6}\else\def\c@ef{3}\fi%
    \figptcopy-4:/#3/\figptcopy-3:/#4/\figptcopy-2:/#5/\figptcopy-1:/#6/%
    \l@mbd@un=-4 \l@mbd@de=-\thr@@\p@rtent=\m@ne\c@lDecast%
    \ifnum#1=\tw@\c@lDCDeux{-4}{-3}\c@lDCDeux{-3}{-2}\c@lDCDeux{-4}{-3}\else%
    \l@mbd@un=-4 \l@mbd@de=-\thr@@\p@rtent=-\tw@\c@lDecast%
    \c@lDCDeux{-4}{-3}\fi\Figg@tXY{-4}}
\ctr@ln@m\c@lDCDeux
\ctr@ld@f\def\c@lDCDeuxDD#1#2{\Figg@tXY{#2}\Figg@tXYa{#1}%
    \advance\v@lX-\v@lXa\advance\v@lY-\v@lYa\Figp@intregDD#1:(\v@lX,\v@lY)}
\ctr@ld@f\def\c@lDCDeuxTD#1#2{\Figg@tXY{#2}\Figg@tXYa{#1}\advance\v@lX-\v@lXa%
    \advance\v@lY-\v@lYa\advance\v@lZ-\v@lZa\Figp@intregTD#1:(\v@lX,\v@lY,\v@lZ)}
\ctr@ln@m\figvectN
\ctr@ld@f\def\figvectNDD#1[#2,#3]{\ifGR@cri{\Figg@tXYa{#2}\Figg@tXY{#3}%
    \advance\v@lX-\v@lXa\advance\v@lY-\v@lYa%
    \Figv@ctCreg#1(-\v@lY,\v@lX)}\ignorespaces\fi}
\ctr@ld@f\def\figvectNTD#1[#2,#3,#4]{\ifGR@cri{\vecunitC@TD[#2,#4]\v@lmin=\v@lX\v@lmax=\v@lY%
    \v@leur=\v@lZ\vecunitC@TD[#2,#3]\c@lprovec{#1}}\ignorespaces\fi}
\ctr@ln@m\figvectNV
\ctr@ld@f\def\figvectNVDD#1[#2]{\ifGR@cri{\Figg@tXY{#2}\Figv@ctCreg#1(-\v@lY,\v@lX)}\ignorespaces\fi}
\ctr@ld@f\def\figvectNVTD#1[#2,#3]{\ifGR@cri{\vecunitCV@TD{#3}\v@lmin=\v@lX\v@lmax=\v@lY%
    \v@leur=\v@lZ\vecunitCV@TD{#2}\c@lprovec{#1}}\ignorespaces\fi}
\ctr@ln@m\figvectP
\ctr@ld@f\def\figvectPDD#1[#2,#3]{\ifGR@cri{\Figg@tXYa{#2}\Figg@tXY{#3}%
    \advance\v@lX-\v@lXa\advance\v@lY-\v@lYa%
    \Figv@ctCreg#1(\v@lX,\v@lY)}\ignorespaces\fi}
\ctr@ld@f\def\figvectPTD#1[#2,#3]{\ifGR@cri{\Figg@tXYa{#2}\Figg@tXY{#3}%
    \advance\v@lX-\v@lXa\advance\v@lY-\v@lYa\advance\v@lZ-\v@lZa%
    \Figv@ctCreg#1(\v@lX,\v@lY,\v@lZ)}\ignorespaces\fi}
\ctr@ln@m\figvectU
\ctr@ld@f\def\figvectUDD#1[#2]{\ifGR@cri{\n@rmeuc\v@leur{#2}\invers@\v@leur\v@leur%
    \delt@=\repdecn@mb{\v@leur}\unit@\edef\v@ldelt@{\repdecn@mb{\delt@}}%
    \Figg@tXY{#2}\v@lX=\v@ldelt@\v@lX\v@lY=\v@ldelt@\v@lY%
    \Figv@ctCreg#1(\v@lX,\v@lY)}\ignorespaces\fi}
\ctr@ld@f\def\figvectUTD#1[#2]{\ifGR@cri{\n@rmeuc\v@leur{#2}\invers@\v@leur\v@leur%
    \delt@=\repdecn@mb{\v@leur}\unit@\edef\v@ldelt@{\repdecn@mb{\delt@}}%
    \Figg@tXY{#2}\v@lX=\v@ldelt@\v@lX\v@lY=\v@ldelt@\v@lY\v@lZ=\v@ldelt@\v@lZ%
    \Figv@ctCreg#1(\v@lX,\v@lY,\v@lZ)}\ignorespaces\fi}
\ctr@ld@f\def\figvisu#1#2#3{\c@ldefproj\initb@undb@x\xdef\figforTeXFigno{\figforTeXnextFigno}%
    \s@mme=\figforTeXnextFigno\advance\s@mme\@ne\xdef\figforTeXnextFigno{\number\s@mme}%
    \setbox\b@xvisu=\hbox{\ifnum\@utoFN>\z@\figinsert{}\gdef\@utoFInDone{0}\fi\ignorespaces#3}%
    \gdef\@utoFInDone{1}\gdef\@utoFN{0}%
    \v@lXa=-\c@@rdYmin\v@lYa=\c@@rdYmax\advance\v@lYa-\c@@rdYmin%
    \v@lX=\c@@rdXmax\advance\v@lX-\c@@rdXmin%
    \setbox#1=\hbox{#2}\v@lY=-\v@lX\maxim@m{\v@lX}{\v@lX}{\wd#1}%
    \advance\v@lY\v@lX\divide\v@lY\tw@\advance\v@lY-\c@@rdXmin%
    \setbox#1=\vbox{\parindent\z@\hsize=\v@lX\vskip\v@lYa%
    \rlap{\hskip\v@lY\smash{\raise\v@lXa\box\b@xvisu}}%
    \def\t@xt@{#2}\ifx\t@xt@\empty\else\medskip\centerline{#2}\fi}\wd#1=\v@lX}
\ctr@ld@f\def\figDecrementFigno{{\xdef\figforTeXnextFigno{\figforTeXFigno}%
    \s@mme=\figforTeXFigno\advance\s@mme\m@ne\xdef\figforTeXFigno{\number\s@mme}}}
\ctr@ln@w{newbox}\Bt@rget\setbox\Bt@rget=\null
\ctr@ln@w{newbox}\BminTD@\setbox\BminTD@=\null
\ctr@ln@w{newbox}\BmaxTD@\setbox\BmaxTD@=\null
\ctr@ln@w{newif}\ifnewt@rgetpt\ctr@ln@w{newif}\ifnewdis@b
\ctr@ld@f\def\b@undb@xTD#1#2#3{%
    \relax\ifdim#1<\wd\BminTD@\global\wd\BminTD@=#1\fi%
    \relax\ifdim#2<\ht\BminTD@\global\ht\BminTD@=#2\fi%
    \relax\ifdim#3<\dp\BminTD@\global\dp\BminTD@=#3\fi%
    \relax\ifdim#1>\wd\BmaxTD@\global\wd\BmaxTD@=#1\fi%
    \relax\ifdim#2>\ht\BmaxTD@\global\ht\BmaxTD@=#2\fi%
    \relax\ifdim#3>\dp\BmaxTD@\global\dp\BmaxTD@=#3\fi}
\ctr@ld@f\def\c@ldefdisob{{\ifdim\wd\BminTD@<\maxdimen\v@leur=\wd\BmaxTD@\advance\v@leur-\wd\BminTD@%
    \delt@=\ht\BmaxTD@\advance\delt@-\ht\BminTD@\maxim@m{\v@leur}{\v@leur}{\delt@}%
    \delt@=\dp\BmaxTD@\advance\delt@-\dp\BminTD@\maxim@m{\v@leur}{\v@leur}{\delt@}%
    \v@leur=5\v@leur\else\v@leur=800pt\fi\c@ldefdisob@{\v@leur}}}
\ctr@ln@m\disob@intern
\ctr@ln@m\disob@
\ctr@ln@m\divf@ctproj
\ctr@ld@f\def\c@ldefdisob@#1{{\v@leur=#1\ifdim\v@leur<\p@\v@leur=800pt\fi%
    \xdef\disob@intern{\repdecn@mb{\v@leur}}%
    \delt@=\ptT@unit@\v@leur\xdef\disob@unit{\repdecn@mb{\delt@}}%
    \f@ctech=\@ne\loop\ifdim\v@leur>\t@n pt\divide\v@leur\t@n\multiply\f@ctech\t@n\repeat%
    \xdef\disob@{\repdecn@mb{\v@leur}}\xdef\divf@ctproj{\the\f@ctech}}%
    \global\newdis@btrue}
\ctr@ln@m\t@rgetpt
\ctr@ld@f\def\c@ldeft@rgetpt{\newt@rgetpttrue\def\t@rgetpt{CenterBoundBox}{%
    \delt@=\wd\BmaxTD@\advance\delt@-\wd\BminTD@\divide\delt@\tw@%
    \v@leur=\wd\BminTD@\advance\v@leur\delt@\global\wd\Bt@rget=\v@leur%
    \delt@=\ht\BmaxTD@\advance\delt@-\ht\BminTD@\divide\delt@\tw@%
    \v@leur=\ht\BminTD@\advance\v@leur\delt@\global\ht\Bt@rget=\v@leur%
    \delt@=\dp\BmaxTD@\advance\delt@-\dp\BminTD@\divide\delt@\tw@%
    \v@leur=\dp\BminTD@\advance\v@leur\delt@\global\dp\Bt@rget=\v@leur}}
\ctr@ln@m\c@ldefproj
\ctr@ld@f\def\c@ldefprojTD{\ifnewt@rgetpt\else\c@ldeft@rgetpt\fi\ifnewdis@b\else\c@ldefdisob\fi}
\ctr@ld@f\def\c@lprojcav{
    \v@lZa=\cxa@\v@lY\advance\v@lX\v@lZa%
    \v@lZa=\cxb@\v@lY\v@lY=\v@lZ\advance\v@lY\v@lZa\ignorespaces}
\ctr@ln@m\v@lcoef
\ctr@ld@f\def\c@lprojrea{
    \advance\v@lX-\wd\Bt@rget\advance\v@lY-\ht\Bt@rget\advance\v@lZ-\dp\Bt@rget%
    \v@lZa=\cza@\v@lX\advance\v@lZa\czb@\v@lY\advance\v@lZa\czc@\v@lZ%
    \divide\v@lZa\divf@ctproj\advance\v@lZa\disob@ pt\invers@{\v@lZa}{\v@lZa}%
    \v@lZa=\disob@\v@lZa\edef\v@lcoef{\repdecn@mb{\v@lZa}}%
    \v@lXa=\cxa@\v@lX\advance\v@lXa\cxb@\v@lY\v@lXa=\v@lcoef\v@lXa%
    \v@lY=\cyb@\v@lY\advance\v@lY\cya@\v@lX\advance\v@lY\cyc@\v@lZ%
    \v@lY=\v@lcoef\v@lY\v@lX=\v@lXa\ignorespaces}
\ctr@ld@f\def\c@lprojort{
    \v@lXa=\cxa@\v@lX\advance\v@lXa\cxb@\v@lY%
    \v@lY=\cyb@\v@lY\advance\v@lY\cya@\v@lX\advance\v@lY\cyc@\v@lZ%
    \v@lX=\v@lXa\ignorespaces}
\ctr@ld@f\def\Figptpr@j#1:#2/#3/{{\Figg@tXY{#3}\superc@lprojSP%
    \Figp@intregDD#1:{#2}(\v@lX,\v@lY)}\ignorespaces}
\ctr@ln@m\figsetobdist
\ctr@ld@f\def\figsetobdistDD{\un@v@ilable{figsetobdist}}
\ctr@ld@f\def\figsetobdistTD(#1){{\ifCUR@PS\W@rnmesIgn{figset proj(dist=...)}%
    \else\v@leur=#1\unit@\c@ldefdisob@{\v@leur}\fi}\ignorespaces}
\ctr@ln@m\c@lprojSP
\ctr@ln@m\CUR@proj
\ctr@ln@m\typ@proj
\ctr@ln@m\superc@lprojSP
\ctr@ld@f\def\Figs@tproj#1{%
    \if#13 \def@ultproj\else\if#1c\def@ultproj%
    \else\if#1o\xdef\CUR@proj{1}\xdef\typ@proj{orthogonal}%
         \figsetviewTD(\def@ultpsi,\def@ulttheta)%
         \global\let\c@lprojSP=\c@lprojort\global\let\superc@lprojSP=\c@lprojort%
    \else\if#1r\xdef\CUR@proj{2}\xdef\typ@proj{realistic}%
         \figsetviewTD(\def@ultpsi,\def@ulttheta)%
         \global\let\c@lprojSP=\c@lprojrea\global\let\superc@lprojSP=\c@lprojrea%
    \else\def@ultproj\message{*** Unknown projection. Cavalier projection assumed.}%
    \fi\fi\fi\fi}
\ctr@ld@f\def\def@ultproj{\xdef\CUR@proj{0}\xdef\typ@proj{cavalier}\figsetviewTD(\def@ultpsi,0.5)%
         \global\let\c@lprojSP=\c@lprojcav\global\let\superc@lprojSP=\c@lprojcav}
\ctr@ln@m\figsettarget
\ctr@ld@f\def\figsettargetDD{\un@v@ilable{figsettarget}}
\ctr@ld@f\def\figsettargetTD[#1]{{\ifCUR@PS\W@rnmesIgn{figset proj(targetpt=...)}%
    \else\global\newt@rgetpttrue\xdef\t@rgetpt{#1}\Figg@tXY{#1}\global\wd\Bt@rget=\v@lX%
    \global\ht\Bt@rget=\v@lY\global\dp\Bt@rget=\v@lZ\fi}\ignorespaces}
\ctr@ln@m\figsetview
\ctr@ld@f\def\figsetviewDD{\un@v@ilable{figsetview}}
\ctr@ld@f\def\figsetviewTD(#1){\ifCUR@PS\W@rnmesIgn{figset proj(Psi|Theta|Lambda=...)}%
     \else\Figsetview@#1,:\fi\ignorespaces}
\ctr@ld@f\def\Figsetview@#1,#2:{{\xdef\v@lPsi{#1}\def\t@xt@{#2}%
    \ifx\t@xt@\empty\def\@rgdeux{\v@lTheta}\else\X@rgdeux@#2\fi%
    \c@ssin{\costhet@}{\sinthet@}{#1}\v@lmin=\costhet@ pt\v@lmax=\sinthet@ pt%
    \ifcase\CUR@proj%
    \v@leur=\@rgdeux\v@lmin\xdef\cxa@{\repdecn@mb{\v@leur}}%
    \v@leur=\@rgdeux\v@lmax\xdef\cxb@{\repdecn@mb{\v@leur}}\v@leur=\@rgdeux pt%
    \relax\ifdim\v@leur>\p@\message{*** Lambda too large ! See \BS@ figset proj() !}\fi%
    \else%
    \v@lmax=-\v@lmax\xdef\cxa@{\repdecn@mb{\v@lmax}}\xdef\cxb@{\costhet@}%
    \ifx\t@xt@\empty\edef\@rgdeux{\def@ulttheta}\fi\c@ssin{\C@}{\S@}{\@rgdeux}%
    \v@lmax=-\S@ pt%
    \v@leur=\v@lmax\v@leur=\costhet@\v@leur\xdef\cya@{\repdecn@mb{\v@leur}}%
    \v@leur=\v@lmax\v@leur=\sinthet@\v@leur\xdef\cyb@{\repdecn@mb{\v@leur}}%
    \xdef\cyc@{\C@}\v@lmin=-\C@ pt%
    \v@leur=\v@lmin\v@leur=\costhet@\v@leur\xdef\cza@{\repdecn@mb{\v@leur}}%
    \v@leur=\v@lmin\v@leur=\sinthet@\v@leur\xdef\czb@{\repdecn@mb{\v@leur}}%
    \xdef\czc@{\repdecn@mb{\v@lmax}}\fi%
    \xdef\v@lTheta{\@rgdeux}}}
\ctr@ld@f\def\def@ultpsi{40}
\ctr@ld@f\def\def@ulttheta{25}
\ctr@ln@m\l@debut
\ctr@ln@m\n@mref
\ctr@ld@f\def\Figsetpr@j#1=#2|{\keln@mtr#1|%
    \def\n@mref{dep}\ifx\l@debut\n@mref\Figsetd@p{#2}\else
    \def\n@mref{dis}\ifx\l@debut\n@mref%
     \ifnum\CUR@proj=\tw@\figsetobdist(#2)\else\Figset@rr\fi\else
    \def\n@mref{lam}\ifx\l@debut\n@mref\Figsetd@p{#2}\else
    \def\n@mref{lat}\ifx\l@debut\n@mref\Figsetth@{#2}\else
    \def\n@mref{lon}\ifx\l@debut\n@mref\figsetview(#2)\else
    \def\n@mref{psi}\ifx\l@debut\n@mref\figsetview(#2)\else
    \def\n@mref{tar}\ifx\l@debut\n@mref%
     \ifnum\CUR@proj=\tw@\figsettarget[#2]\else\Figset@rr\fi\else
    \def\n@mref{the}\ifx\l@debut\n@mref\Figsetth@{#2}\else
    \W@rnmesAttr{figset proj}{#1}\fi\fi\fi\fi\fi\fi\fi\fi}
\ctr@ld@f\def\Figsetd@p#1{\ifnum\CUR@proj=\z@\figsetview(\v@lPsi,#1)\else\Figset@rr\fi}
\ctr@ld@f\def\Figsetth@#1{\ifnum\CUR@proj=\z@\Figset@rr\else\figsetview(\v@lPsi,#1)\fi}
\ctr@ld@f\def\Figset@rr{\message{*** \BS@ figset proj(): Attribute "\n@mref" ignored, incompatible
    with current projection}}
\ctr@ld@f\def\initb@undb@xTD{\wd\BminTD@=\maxdimen\ht\BminTD@=\maxdimen\dp\BminTD@=\maxdimen%
    \wd\BmaxTD@=-\maxdimen\ht\BmaxTD@=-\maxdimen\dp\BmaxTD@=-\maxdimen}
\ctr@ln@w{newbox}\Gb@x      
\ctr@ln@w{newbox}\Gb@xSC    
\ctr@ln@w{newtoks}\c@nsymb  
\ctr@ln@w{newif}\ifr@undcoord\ctr@ln@w{newif}\ifunitpr@sent
\ctr@ld@f\def\unssqrttw@{0.707106 }
\ctr@ld@f\def\figAst{\raise-1.15ex\hbox{$\ast$}}
\ctr@ld@f\def\figBullet{\raise-1.15ex\hbox{$\bullet$}}
\ctr@ld@f\def\figCirc{\raise-1.15ex\hbox{$\circ$}}
\ctr@ld@f\def\figDiamond{\raise-1.15ex\hbox{$\diamond$}}%
\ctr@ld@f\def\boxit#1#2{\leavevmode\hbox{\vrule\vbox{\hrule\vglue#1%
    \vtop{\hbox{\kern#1{#2}\kern#1}\vglue#1\hrule}}\vrule}}
\ctr@ld@f
\ctr@ld@f
\ctr@ld@f\def\c@nterpt{\ignorespaces%
    \kern-.5\wd\Gb@xSC%
    \raise-.5\ht\Gb@xSC\rlap{\hbox{\raise.5\dp\Gb@xSC\hbox{\copy\Gb@xSC}}}%
    \kern .5\wd\Gb@xSC\ignorespaces}
\ctr@ld@f\def\b@undb@xSC#1#2{{\v@lXa=#1\v@lYa=#2%
    \v@leur=\ht\Gb@xSC\advance\v@leur\dp\Gb@xSC%
    \advance\v@lXa-.5\wd\Gb@xSC\advance\v@lYa-.5\v@leur\b@undb@x{\v@lXa}{\v@lYa}%
    \advance\v@lXa\wd\Gb@xSC\advance\v@lYa\v@leur\b@undb@x{\v@lXa}{\v@lYa}}}
\ctr@ln@m\Dist@n
\ctr@ln@m\l@suite
\ctr@ld@f\def\@keldist#1#2{\edef\Dist@n{#2}\y@tiunit{\Dist@n}%
    \ifunitpr@sent#1=\Dist@n\else#1=\Dist@n\unit@\fi}
\ctr@ld@f\def\y@tiunit#1{\unitpr@sentfalse\expandafter\y@tiunit@#1:}
\ctr@ld@f\def\y@tiunit@#1#2:{\ifcat#1a\unitpr@senttrue\else\def\l@suite{#2}%
    \ifx\l@suite\empty\else\y@tiunit@#2:\fi\fi}
\ctr@ln@m\figcoord
\ctr@ld@f\def\figcoordDD#1{{\v@lX=\ptT@unit@\v@lX\v@lY=\ptT@unit@\v@lY%
    \ifr@undcoord\ifcase#1\v@leur=0.5pt\or\v@leur=0.05pt\or\v@leur=0.005pt%
    \or\v@leur=0.0005pt\else\v@leur=\z@\fi%
    \ifdim\v@lX<\z@\advance\v@lX-\v@leur\else\advance\v@lX\v@leur\fi%
    \ifdim\v@lY<\z@\advance\v@lY-\v@leur\else\advance\v@lY\v@leur\fi\fi%
    (\@ffichnb{#1}{\repdecn@mb{\v@lX}},\ifmmode\else\thinspace\fi%
    \@ffichnb{#1}{\repdecn@mb{\v@lY}})}}
\ctr@ld@f\def\@ffichnb#1#2{{\def\@@ffich{\@ffich#1(}\edef\n@mbre{#2}%
    \expandafter\@@ffich\n@mbre)}}
\ctr@ld@f\def\@ffich#1(#2.#3){{#2\ifnum#1>\z@.\fi\def\dig@ts{#3}\s@mme=\z@%
    \loop\ifnum\s@mme<#1\expandafter\@ffichdec\dig@ts:\advance\s@mme\@ne\repeat}}
\ctr@ld@f\def\@ffichdec#1#2:{\relax#1\def\dig@ts{#20}}
\ctr@ld@f\def\figcoordTD#1{{\v@lX=\ptT@unit@\v@lX\v@lY=\ptT@unit@\v@lY\v@lZ=\ptT@unit@\v@lZ%
    \ifr@undcoord\ifcase#1\v@leur=0.5pt\or\v@leur=0.05pt\or\v@leur=0.005pt%
    \or\v@leur=0.0005pt\else\v@leur=\z@\fi%
    \ifdim\v@lX<\z@\advance\v@lX-\v@leur\else\advance\v@lX\v@leur\fi%
    \ifdim\v@lY<\z@\advance\v@lY-\v@leur\else\advance\v@lY\v@leur\fi%
    \ifdim\v@lZ<\z@\advance\v@lZ-\v@leur\else\advance\v@lZ\v@leur\fi\fi%
    (\@ffichnb{#1}{\repdecn@mb{\v@lX}},\ifmmode\else\thinspace\fi%
     \@ffichnb{#1}{\repdecn@mb{\v@lY}},\ifmmode\else\thinspace\fi%
     \@ffichnb{#1}{\repdecn@mb{\v@lZ}})}}
\ctr@ld@f\def\figsetroundcoord#1{\expandafter\Figsetr@undcoord#1:\ignorespaces}
\ctr@ld@f\def\Figsetr@undcoord#1#2:{\if#1n\r@undcoordfalse\else\r@undcoordtrue\fi}
\ctr@ld@f\def\Figsetwr@te#1=#2|{\keln@mun#1|%
    \def\n@mref{m}\ifx\l@debut\n@mref\figsetmark{#2}\else
    \def\n@mref{p}\ifx\l@debut\n@mref\figsetptname{#2}\else
    \def\n@mref{r}\ifx\l@debut\n@mref\figsetroundcoord{#2}\else
    \W@rnmesAttr{figset write}{#1}\fi\fi\fi}
\ctr@ld@f\def\figsetmark#1{\c@nsymb={#1}\setbox\Gb@xSC=\hbox{\the\c@nsymb}\ignorespaces}
\ctr@ln@m\ptn@me
\ctr@ld@f\def\figsetptname#1{\def\ptn@me##1{#1}\ignorespaces}
\ctr@ld@f\def\FigWrit@L#1:#2(#3,#4){\ignorespaces\@keldist\v@leur{#3}\@keldist\delt@{#4}%
    \C@rp@r@m\def\list@num{#1}\@ecfor\p@int:=\list@num\do{\FigWrit@pt{\p@int}{#2}}}
\ctr@ld@f\def\FigWrit@pt#1#2{\FigWp@r@m{#1}{#2}\Vc@rrect\figWp@si%
    \ifdim\wd\Gb@xSC>\z@\b@undb@xSC{\v@lX}{\v@lY}\fi\figWBB@x}
\ctr@ld@f\def\FigWp@r@m#1#2{\Figg@tXY{#1}%
    \setbox\Gb@x=\hbox{\def\t@xt@{#2}\ifx\t@xt@\empty\Figg@tT{#1}\else#2\fi}\c@lprojSP}
\ctr@ld@f\let\Vc@rrect=\relax
\ctr@ld@f\let\C@rp@r@m=\relax
\ctr@ld@f\def\figwrite[#1]#2{{\ignorespaces\def\list@num{#1}\@ecfor\p@int:=\list@num\do{%
    \setbox\Gb@x=\hbox{\def\t@xt@{#2}\ifx\t@xt@\empty\Figg@tT{\p@int}\else#2\fi}%
    \Figwrit@{\p@int}}}\ignorespaces}
\ctr@ld@f\def\Figwrit@#1{\Figg@tXY{#1}\c@lprojSP%
    \rlap{\kern\v@lX\raise\v@lY\hbox{\unhcopy\Gb@x}}\v@leur=\v@lY%
    \advance\v@lY\ht\Gb@x\b@undb@x{\v@lX}{\v@lY}\advance\v@lX\wd\Gb@x%
    \v@lY=\v@leur\advance\v@lY-\dp\Gb@x\b@undb@x{\v@lX}{\v@lY}}
\ctr@ld@f\def\figwritec[#1]#2{{\ignorespaces\def\list@num{#1}%
    \@ecfor\p@int:=\list@num\do{\Figwrit@c{\p@int}{#2}}}\ignorespaces}
\ctr@ld@f\def\Figwrit@c#1#2{\FigWp@r@m{#1}{#2}%
    \rlap{\kern\v@lX\raise\v@lY\hbox{\rlap{\kern-.5\wd\Gb@x%
    \raise-.5\ht\Gb@x\hbox{\raise.5\dp\Gb@x\hbox{\unhcopy\Gb@x}}}}}%
    \v@leur=\ht\Gb@x\advance\v@leur\dp\Gb@x%
    \advance\v@lX-.5\wd\Gb@x\advance\v@lY-.5\v@leur\b@undb@x{\v@lX}{\v@lY}%
    \advance\v@lX\wd\Gb@x\advance\v@lY\v@leur\b@undb@x{\v@lX}{\v@lY}}
\ctr@ld@f\def\figwritep[#1]{{\ignorespaces\def\list@num{#1}\setbox\Gb@x=\hbox{\c@nterpt}%
    \@ecfor\p@int:=\list@num\do{\Figwrit@{\p@int}}}\ignorespaces}
\ctr@ld@f\def\figwritew#1:#2(#3){\figwritegcw#1:{#2}(#3,0pt)}
\ctr@ld@f\def\figwritee#1:#2(#3){\figwritegce#1:{#2}(#3,0pt)}
\ctr@ld@f\def\figwriten#1:#2(#3){{\def\Vc@rrect{\v@lZ=\v@leur\advance\v@lZ\dp\Gb@x}%
    \Figwrit@NS#1:{#2}(#3)}\ignorespaces}
\ctr@ld@f\def\figwrites#1:#2(#3){{\def\Vc@rrect{\v@lZ=-\v@leur\advance\v@lZ-\ht\Gb@x}%
    \Figwrit@NS#1:{#2}(#3)}\ignorespaces}
\ctr@ld@f\def\Figwrit@NS#1:#2(#3){\let\figWp@si=\FigWp@siNS\let\figWBB@x=\FigWBB@xNS%
    \FigWrit@L#1:{#2}(#3,0pt)}
\ctr@ld@f\def\FigWp@siNS{\rlap{\kern\v@lX\raise\v@lY\hbox{\rlap{\kern-.5\wd\Gb@x%
    \raise\v@lZ\hbox{\unhcopy\Gb@x}}\c@nterpt}}}
\ctr@ld@f\def\FigWBB@xNS{\advance\v@lY\v@lZ%
    \advance\v@lY-\dp\Gb@x\advance\v@lX-.5\wd\Gb@x\b@undb@x{\v@lX}{\v@lY}%
    \advance\v@lY\ht\Gb@x\advance\v@lY\dp\Gb@x%
    \advance\v@lX\wd\Gb@x\b@undb@x{\v@lX}{\v@lY}}
\ctr@ld@f\def\figwritenw#1:#2(#3){{\let\figWp@si=\FigWp@sigW\let\figWBB@x=\FigWBB@xgWE%
    \def\C@rp@r@m{\v@leur=\unssqrttw@\v@leur\delt@=\v@leur%
    \ifdim\delt@=\z@\delt@=\epsil@n\fi}\let@xte={-}\FigWrit@L#1:{#2}(#3,0pt)}\ignorespaces}
\ctr@ld@f\def\figwritesw#1:#2(#3){{\let\figWp@si=\FigWp@sigW\let\figWBB@x=\FigWBB@xgWE%
    \def\C@rp@r@m{\v@leur=\unssqrttw@\v@leur\delt@=-\v@leur%
    \ifdim\delt@=\z@\delt@=-\epsil@n\fi}\let@xte={-}\FigWrit@L#1:{#2}(#3,0pt)}\ignorespaces}
\ctr@ld@f\def\figwritene#1:#2(#3){{\let\figWp@si=\FigWp@sigE\let\figWBB@x=\FigWBB@xgWE%
    \def\C@rp@r@m{\v@leur=\unssqrttw@\v@leur\delt@=\v@leur%
    \ifdim\delt@=\z@\delt@=\epsil@n\fi}\let@xte={}\FigWrit@L#1:{#2}(#3,0pt)}\ignorespaces}
\ctr@ld@f\def\figwritese#1:#2(#3){{\let\figWp@si=\FigWp@sigE\let\figWBB@x=\FigWBB@xgWE%
    \def\C@rp@r@m{\v@leur=\unssqrttw@\v@leur\delt@=-\v@leur%
    \ifdim\delt@=\z@\delt@=-\epsil@n\fi}\let@xte={}\FigWrit@L#1:{#2}(#3,0pt)}\ignorespaces}
\ctr@ld@f\def\figwritegw#1:#2(#3,#4){{\let\figWp@si=\FigWp@sigW\let\figWBB@x=\FigWBB@xgWE%
    \let@xte={-}\FigWrit@L#1:{#2}(#3,#4)}\ignorespaces}
\ctr@ld@f\def\figwritege#1:#2(#3,#4){{\let\figWp@si=\FigWp@sigE\let\figWBB@x=\FigWBB@xgWE%
    \let@xte={}\FigWrit@L#1:{#2}(#3,#4)}\ignorespaces}
\ctr@ld@f\def\FigWp@sigW{\v@lXa=\z@\v@lYa=\ht\Gb@x\advance\v@lYa\dp\Gb@x%
    \ifdim\delt@>\z@\relax%
    \rlap{\kern\v@lX\raise\v@lY\hbox{\rlap{\kern-\wd\Gb@x\kern-\v@leur%
          \raise\delt@\hbox{\raise\dp\Gb@x\hbox{\unhcopy\Gb@x}}}\c@nterpt}}%
    \else\ifdim\delt@<\z@\relax\v@lYa=-\v@lYa%
    \rlap{\kern\v@lX\raise\v@lY\hbox{\rlap{\kern-\wd\Gb@x\kern-\v@leur%
          \raise\delt@\hbox{\raise-\ht\Gb@x\hbox{\unhcopy\Gb@x}}}\c@nterpt}}%
    \else\v@lXa=-.5\v@lYa%
    \rlap{\kern\v@lX\raise\v@lY\hbox{\rlap{\kern-\wd\Gb@x\kern-\v@leur%
          \raise-.5\ht\Gb@x\hbox{\raise.5\dp\Gb@x\hbox{\unhcopy\Gb@x}}}\c@nterpt}}%
    \fi\fi}
\ctr@ld@f\def\FigWp@sigE{\v@lXa=\z@\v@lYa=\ht\Gb@x\advance\v@lYa\dp\Gb@x%
    \ifdim\delt@>\z@\relax%
    \rlap{\kern\v@lX\raise\v@lY\hbox{\c@nterpt\kern\v@leur%
          \raise\delt@\hbox{\raise\dp\Gb@x\hbox{\unhcopy\Gb@x}}}}%
    \else\ifdim\delt@<\z@\relax\v@lYa=-\v@lYa%
    \rlap{\kern\v@lX\raise\v@lY\hbox{\c@nterpt\kern\v@leur%
          \raise\delt@\hbox{\raise-\ht\Gb@x\hbox{\unhcopy\Gb@x}}}}%
    \else\v@lXa=-.5\v@lYa%
    \rlap{\kern\v@lX\raise\v@lY\hbox{\c@nterpt\kern\v@leur%
          \raise-.5\ht\Gb@x\hbox{\raise.5\dp\Gb@x\hbox{\unhcopy\Gb@x}}}}%
    \fi\fi}
\ctr@ld@f\def\FigWBB@xgWE{\advance\v@lY\delt@%
    \advance\v@lX\the\let@xte\v@leur\advance\v@lY\v@lXa\b@undb@x{\v@lX}{\v@lY}%
    \advance\v@lX\the\let@xte\wd\Gb@x\advance\v@lY\v@lYa\b@undb@x{\v@lX}{\v@lY}}
\ctr@ld@f\def\figwritegcw#1:#2(#3,#4){{\let\figWp@si=\FigWp@sigcW\let\figWBB@x=\FigWBB@xgcWE%
    \let@xte={-}\FigWrit@L#1:{#2}(#3,#4)}\ignorespaces}
\ctr@ld@f\def\figwritegce#1:#2(#3,#4){{\let\figWp@si=\FigWp@sigcE\let\figWBB@x=\FigWBB@xgcWE%
    \let@xte={}\FigWrit@L#1:{#2}(#3,#4)}\ignorespaces}
\ctr@ld@f\def\FigWp@sigcW{\rlap{\kern\v@lX\raise\v@lY\hbox{\rlap{\kern-\wd\Gb@x\kern-\v@leur%
     \raise-.5\ht\Gb@x\hbox{\raise\delt@\hbox{\raise.5\dp\Gb@x\hbox{\unhcopy\Gb@x}}}}%
     \c@nterpt}}}
\ctr@ld@f\def\FigWp@sigcE{\rlap{\kern\v@lX\raise\v@lY\hbox{\c@nterpt\kern\v@leur%
    \raise-.5\ht\Gb@x\hbox{\raise\delt@\hbox{\raise.5\dp\Gb@x\hbox{\unhcopy\Gb@x}}}}}}
\ctr@ld@f\def\FigWBB@xgcWE{\v@lZ=\ht\Gb@x\advance\v@lZ\dp\Gb@x%
    \advance\v@lX\the\let@xte\v@leur\advance\v@lY\delt@\advance\v@lY.5\v@lZ%
    \b@undb@x{\v@lX}{\v@lY}%
    \advance\v@lX\the\let@xte\wd\Gb@x\advance\v@lY-\v@lZ\b@undb@x{\v@lX}{\v@lY}}
\ctr@ld@f\def\figwritebn#1:#2(#3){{\def\Vc@rrect{\v@lZ=\v@leur}\Figwrit@NS#1:{#2}(#3)}\ignorespaces}
\ctr@ld@f\def\figwritebs#1:#2(#3){{\def\Vc@rrect{\v@lZ=-\v@leur}\Figwrit@NS#1:{#2}(#3)}\ignorespaces}
\ctr@ld@f\def\figwritebw#1:#2(#3){{\let\figWp@si=\FigWp@sibW\let\figWBB@x=\FigWBB@xbWE%
    \let@xte={-}\FigWrit@L#1:{#2}(#3,0pt)}\ignorespaces}
\ctr@ld@f\def\figwritebe#1:#2(#3){{\let\figWp@si=\FigWp@sibE\let\figWBB@x=\FigWBB@xbWE%
    \let@xte={}\FigWrit@L#1:{#2}(#3,0pt)}\ignorespaces}
\ctr@ld@f\def\FigWp@sibW{\rlap{\kern\v@lX\raise\v@lY\hbox{\rlap{\kern-\wd\Gb@x\kern-\v@leur%
          \hbox{\unhcopy\Gb@x}}\c@nterpt}}}
\ctr@ld@f\def\FigWp@sibE{\rlap{\kern\v@lX\raise\v@lY\hbox{\c@nterpt\kern\v@leur%
          \hbox{\unhcopy\Gb@x}}}}
\ctr@ld@f\def\FigWBB@xbWE{\v@lZ=\ht\Gb@x\advance\v@lZ\dp\Gb@x%
    \advance\v@lX\the\let@xte\v@leur\advance\v@lY\ht\Gb@x\b@undb@x{\v@lX}{\v@lY}%
    \advance\v@lX\the\let@xte\wd\Gb@x\advance\v@lY-\v@lZ\b@undb@x{\v@lX}{\v@lY}}
\ctr@ln@w{newread}\frf@g  \ctr@ln@w{newwrite}\fwf@g
\ctr@ln@w{newif}\ifCUR@PS
\ctr@ln@w{newif}\ifGR@cri
\ctr@ln@w{newif}\ifUse@llipse
\ctr@ln@w{newif}\ifGRdebugm@de \GRdebugm@defalse 
\ctr@ln@w{newif}\ifPDFm@ke
\ifx\pdfliteral\undefined\else\ifnum\pdfoutput>\z@\PDFm@ketrue\fi\fi
\ctr@ld@f\def\initPDF@rDVI{%
\ifPDFm@ke
 \let\figscan=\figscan@E
 \let\newGr@FN=\newGr@FNPDF
 \ctr@ld@f\def\c@mcurveto{c}
 \ctr@ld@f\def\c@mfill{f}
 \ctr@ld@f\def\c@mgsave{q}
 \ctr@ld@f\def\c@mgrestore{Q}
 \ctr@ld@f\def\c@mlineto{l}
 \ctr@ld@f\def\c@mmoveto{m}
 \ctr@ld@f\def\c@msetgray{g}     \ctr@ld@f\def\c@msetgrayStroke{G}
 \ctr@ld@f\def\c@msetcmykcolor{k}\ctr@ld@f\def\c@msetcmykcolorStroke{K}
 \ctr@ld@f\def\c@msetrgbcolor{rg}\ctr@ld@f\def\c@msetrgbcolorStroke{RG}
 \ctr@ld@f\def\d@fprimarC@lor{\CUR@color\space\CUR@colorc@md%
               \space\CUR@color\space\CUR@colorc@mdStroke}
 \ctr@ld@f\def\c@msetdash{d}
 \ctr@ld@f\def\c@msetlinejoin{j}
 \ctr@ld@f\def\c@msetlinewidth{w}
 \ctr@ld@f\def\f@gclosestroke{\immediate\write\fwf@g{s}}
 \ctr@ld@f\def\f@gfill{\immediate\write\fwf@g{\fillc@md}}
 \ctr@ld@f\def\f@gnewpath{}
 \ctr@ld@f\def\f@gstroke{\immediate\write\fwf@g{S}}
\else
 \let\figinsertE=\figinsert
 \let\newGr@FN=\newGr@FNDVI
 \ctr@ld@f\def\c@mcurveto{curveto}
 \ctr@ld@f\def\c@mfill{fill}
 \ctr@ld@f\def\c@mgsave{gsave}
 \ctr@ld@f\def\c@mgrestore{grestore}
 \ctr@ld@f\def\c@mlineto{lineto}
 \ctr@ld@f\def\c@mmoveto{moveto}
 \ctr@ld@f\def\c@msetgray{setgray}          \ctr@ld@f\def\c@msetgrayStroke{}
 \ctr@ld@f\def\c@msetcmykcolor{setcmykcolor}\ctr@ld@f\def\c@msetcmykcolorStroke{}
 \ctr@ld@f\def\c@msetrgbcolor{setrgbcolor}  \ctr@ld@f\def\c@msetrgbcolorStroke{}
 \ctr@ld@f\def\d@fprimarC@lor{\CUR@color\space\CUR@colorc@md}
 \ctr@ld@f\def\c@msetdash{setdash}
 \ctr@ld@f\def\c@msetlinejoin{setlinejoin}
 \ctr@ld@f\def\c@msetlinewidth{setlinewidth}
 \ctr@ld@f\def\f@gclosestroke{\immediate\write\fwf@g{closepath\space stroke}}
 \ctr@ld@f\def\f@gfill{\immediate\write\fwf@g{\fillc@md}}
 \ctr@ld@f\def\f@gnewpath{\immediate\write\fwf@g{newpath}}
 \ctr@ld@f\def\f@gstroke{\immediate\write\fwf@g{stroke}}
\fi}
\ctr@ld@f\def\c@pypsfile#1#2{\c@pyfil@{\immediate\write#1}{#2}}
\ctr@ld@f\def\Figinclud@PDF#1#2{\openin\frf@g=#1\pdfliteral{q #2 0 0 #2 0 0 cm}%
    \c@pyfil@{\pdfliteral}{\frf@g}\pdfliteral{Q}\closein\frf@g}
\ctr@ln@w{newif}\ifmored@ta
\ctr@ln@m\bl@nkline
\ctr@ld@f\def\c@pyfil@#1#2{\def\bl@nkline{\par}{\catcode`\%=12
    \loop\ifeof#2\mored@tafalse\else\mored@tatrue\immediate\read#2 to\tr@c
    \ifx\tr@c\bl@nkline\else#1{\tr@c}\fi\fi\ifmored@ta\repeat}}
\ctr@ld@f\def\keln@mun#1#2|{\def\l@debut{#1}\def\l@suite{#2}}
\ctr@ld@f\def\keln@mde#1#2#3|{\def\l@debut{#1#2}\def\l@suite{#3}}
\ctr@ld@f\def\keln@mtr#1#2#3#4|{\def\l@debut{#1#2#3}\def\l@suite{#4}}
\ctr@ld@f\def\keln@mqu#1#2#3#4#5|{\def\l@debut{#1#2#3#4}\def\l@suite{#5}}
\ctr@ld@f\let\@psffilein=\frf@g 
\ctr@ln@w{newif}\if@psffileok    
\ctr@ln@w{newif}\if@psfbbfound   
\ctr@ln@w{newif}\if@psfverbose   
\@psfverbosetrue
\ctr@ln@m\@psfllx \ctr@ln@m\@psflly
\ctr@ln@m\@psfurx \ctr@ln@m\@psfury
\ctr@ln@m\resetcolonc@tcode
\ctr@ld@f\def\@psfgetbb#1{\global\@psfbbfoundfalse%
\global\def\@psfllx{0}\global\def\@psflly{0}%
\global\def\@psfurx{30}\global\def\@psfury{30}%
\openin\@psffilein=#1\relax
\ifeof\@psffilein\errmessage{I couldn't open #1, will ignore it}\else
   \edef\resetcolonc@tcode{\catcode`\noexpand\:\the\catcode`\:\relax}%
   {\@psffileoktrue \chardef\other=12
    \def\do##1{\catcode`##1=\other}\dospecials \catcode`\ =10 \resetcolonc@tcode
    \loop
       \read\@psffilein to \@psffileline
       \ifeof\@psffilein\@psffileokfalse\else
          \expandafter\@psfaux\@psffileline:. \\%
       \fi
   \if@psffileok\repeat
   \if@psfbbfound\else
    \if@psfverbose\message{No bounding box comment in #1; using defaults}\fi\fi
   }\closein\@psffilein\fi}%
\ctr@ln@m\@psfbblit
\ctr@ln@m\@psfpercent
{\catcode`\%=12 \global\let\@psfpercent=
\ctr@ln@m\@psfaux
\long\def\@psfaux#1#2:#3\\{\ifx#1\@psfpercent
   \def\testit{#2}\ifx\testit\@psfbblit
      \@psfgrab #3 . . . \\%
      \@psffileokfalse
      \global\@psfbbfoundtrue
   \fi\else\ifx#1\par\else\@psffileokfalse\fi\fi}%
\ctr@ld@f\def\@psfempty{}%
\ctr@ld@f\def\@psfgrab #1 #2 #3 #4 #5\\{%
\global\def\@psfllx{#1}\ifx\@psfllx\@psfempty
      \@psfgrab #2 #3 #4 #5 .\\\else
   \global\def\@psflly{#2}%
   \global\def\@psfurx{#3}\global\def\@psfury{#4}\fi}%
\ctr@ld@f\def\PSwrit@cmd#1#2#3{{\Figg@tXY{#1}\c@lprojSP\b@undb@x{\v@lX}{\v@lY}%
    \v@lX=\ptT@ptps\v@lX\v@lY=\ptT@ptps\v@lY%
    \immediate\write#3{\repdecn@mb{\v@lX}\space\repdecn@mb{\v@lY}\space#2}}}
\ctr@ld@f\def\PSwrit@cmdS#1#2#3#4#5{{\Figg@tXY{#1}\c@lprojSP\b@undb@x{\v@lX}{\v@lY}%
    \global\result@t=\v@lX\global\result@@t=\v@lY%
    \v@lX=\ptT@ptps\v@lX\v@lY=\ptT@ptps\v@lY%
    \immediate\write#3{\repdecn@mb{\v@lX}\space\repdecn@mb{\v@lY}\space#2}}%
    \edef#4{\the\result@t}\edef#5{\the\result@@t}}
\ctr@ld@f\def\update@ttr#1#2#3{\Figdisc@rdLTS{#3}{\n@mref}%
    \ifx\n@mref\D@FTref#2{#1}\else#2{#3}\fi}
\ctr@ld@f\def\D@FTref{default}
\ctr@ld@f\def\W@rnmesAttr#1#2{%
    \immediate\write16{*** Unknown attribute: \BS@ #1(..., #2=...)}}
\ctr@ld@f\def\W@rnmeskwd#1#2{%
    \immediate\write16{*** Unknown keyword #2 in \BS@ #1}}
\ctr@ld@f\def\W@rnmesIgn#1{\immediate\write16{*** \BS@ #1 is ignored inside a
     \BS@ figdrawbegin-\BS@ figdrawend block.}}
\ctr@ld@f\def\Psset@lti#1=#2|{\keln@mtr#1|%
    \def\n@mref{blc}\ifx\l@debut\n@mref\update@ttr\D@FTref\P@setblcolor{#2}\else
    \def\n@mref{bld}\ifx\l@debut\n@mref\update@ttr\D@FTref\P@setbldash{#2}\else
    \def\n@mref{blw}\ifx\l@debut\n@mref\update@ttr\D@FTref\P@setblwidth{#2}\else
    \def\n@mref{sqc}\ifx\l@debut\n@mref\update@ttr\D@FTref\P@setsqcolor{#2}\else
    \def\n@mref{sqd}\ifx\l@debut\n@mref\update@ttr\D@FTref\P@setsqdash{#2}\else
    \def\n@mref{sqw}\ifx\l@debut\n@mref\update@ttr\D@FTref\P@setsqwidth{#2}\else
    \W@rnmesAttr{figset altitude}{#1}\fi\fi\fi\fi\fi\fi}
\ctr@ln@m\DDV@blcolor
\ctr@ld@f\def\P@setblcolor#1{\edef\DDV@blcolor{#1}}
\ctr@ln@m\DDV@bldash
\ctr@ld@f\def\P@setbldash#1{\edef\DDV@bldash{#1}}
\ctr@ln@m\DDV@blwidth
\ctr@ld@f\def\P@setblwidth#1{\edef\DDV@blwidth{#1}}
\ctr@ln@m\DDV@sqcolor
\ctr@ld@f\def\P@setsqcolor#1{\edef\DDV@sqcolor{#1}}
\ctr@ln@m\DDV@sqdash
\ctr@ld@f\def\P@setsqdash#1{\edef\DDV@sqdash{#1}}
\ctr@ln@m\DDV@sqwidth
\ctr@ld@f\def\P@setsqwidth#1{\edef\DDV@sqwidth{#1}}
\ctr@ld@f\def\figdrawaltitude#1[#2,#3,#4]{{\ifCUR@PS\ifGR@cri%
    \PSc@mment{altitude Square Dim=#1, Triangle=[#2 / #3,#4]}%
    \s@uvc@ntr@l\et@tpsaltitude\resetc@ntr@l{2}\figptorthoprojline-5:=#2/#3,#4/%
    \figvectP -1[#3,#4]\n@rminf{\v@leur}{-1}\vecunit@{-3}{-1}%
    \figvectP -1[-5,#3]\n@rminf{\v@lmin}{-1}\figvectP -2[-5,#4]\n@rminf{\v@lmax}{-2}%
    \ifdim\v@lmin<\v@lmax\s@mme=#3\else\v@lmax=\v@lmin\s@mme=#4\fi%
    \figvectP -4[-5,#2]\vecunit@{-4}{-4}\delt@=#1\unit@%
    \edef\t@ille{\repdecn@mb{\delt@}}\figpttra-1:=-5/\t@ille,-3/%
    \figptstra-3=-5,-1/\t@ille,-4/\figdrawline[#2,-5]%
    \Pss@tspecifSt{color=\DDV@sqcolor,dash=\DDV@sqdash,width=\DDV@sqwidth}%
    \figdrawline[-1,-2,-3]%
    \Psrest@reSt{color=\DDV@sqcolor,dash=\DDV@sqdash,width=\DDV@sqwidth}%
    \ifdim\v@leur<\v@lmax%
    \Pss@tspecifSt{color=\DDV@blcolor,dash=\DDV@bldash,width=\DDV@blwidth}%
    \figdrawline[-5,\the\s@mme]%
    \Psrest@reSt{color=\DDV@blcolor,dash=\DDV@bldash,width=\DDV@blwidth}%
    \fi\PSc@mment{End altitude}\resetc@ntr@l\et@tpsaltitude\fi\fi}}
\ctr@ld@f\def\Ps@rcerc#1;#2(#3,#4){\ellBB@x#1;#2,#2(#3,#4,0)%
    \f@gnewpath{\delt@=#2\unit@\delt@=\ptT@ptps\delt@%
    \BdingB@xfalse%
    \PSwrit@cmd{#1}{\repdecn@mb{\delt@}\space #3\space #4\space arc}{\fwf@g}}}
\ctr@ln@m\figdrawarccirc
\ctr@ld@f\def\Q@arccircDD#1;#2(#3,#4){\ifCUR@PS\ifGR@cri%
    \PSc@mment{arccircDD Center=#1 ; Radius=#2 (Ang1=#3, Ang2=#4)}%
    \iffillm@de\Ps@rcerc#1;#2(#3,#4)%
    \f@gfill%
    \else\Ps@rcerc#1;#2(#3,#4)\f@gstroke\fi%
    \PSc@mment{End arccircDD}\fi\fi}
\ctr@ld@f\def\Q@arccircTD#1,#2,#3;#4(#5,#6){{\ifCUR@PS\ifGR@cri\s@uvc@ntr@l\et@tpsarccircTD%
    \PSc@mment{arccircTD Center=#1,P1=#2,P2=#3 ; Radius=#4 (Ang1=#5, Ang2=#6)}%
    \setc@ntr@l{2}\c@lExtAxes#1,#2,#3(#4)\Q@arcellPATD#1,-4,-5(#5,#6)%
    \PSc@mment{End arccircTD}\resetc@ntr@l\et@tpsarccircTD\fi\fi}}
\ctr@ld@f\def\c@lExtAxes#1,#2,#3(#4){%
    \figvectPTD-5[#1,#2]\vecunit@{-5}{-5}\figvectNTD-4[#1,#2,#3]\vecunit@{-4}{-4}%
    \figvectNVTD-3[-4,-5]\delt@=#4\unit@\edef\r@yon{\repdecn@mb{\delt@}}%
    \figpttra-4:=#1/\r@yon,-5/\figpttra-5:=#1/\r@yon,-3/}
\ctr@ln@m\figdrawarccircP
\ctr@ld@f\def\Q@arccircPDD#1;#2[#3,#4]{{\ifCUR@PS\ifGR@cri\s@uvc@ntr@l\et@tpsarccircPDD%
    \PSc@mment{arccircPDD Center=#1; Radius=#2, [P1=#3, P2=#4]}%
    \Ps@ngleparam#1;#2[#3,#4]\ifdim\v@lmin>\v@lmax\advance\v@lmax\DePI@deg\fi%
    \edef\@ngdeb{\repdecn@mb{\v@lmin}}\edef\@ngfin{\repdecn@mb{\v@lmax}}%
    \figdrawarccirc#1;\r@dius(\@ngdeb,\@ngfin)%
    \PSc@mment{End arccircPDD}\resetc@ntr@l\et@tpsarccircPDD\fi\fi}}
\ctr@ld@f\def\Q@arccircPTD#1;#2[#3,#4,#5]{{\ifCUR@PS\ifGR@cri\s@uvc@ntr@l\et@tpsarccircPTD%
    \PSc@mment{arccircPTD Center=#1; Radius=#2, [P1=#3, P2=#4, P3=#5]}%
    \setc@ntr@l{2}\c@lExtAxes#1,#3,#5(#2)\figdrawarcellPP#1,-4,-5[#3,#4]%
    \PSc@mment{End arccircPTD}\resetc@ntr@l\et@tpsarccircPTD\fi\fi}}
\ctr@ld@f\def\Ps@ngleparam#1;#2[#3,#4]{\setc@ntr@l{2}%
    \figvectPDD-1[#1,#3]\vecunit@{-1}{-1}\Figg@tXY{-1}\arct@n\v@lmin(\v@lX,\v@lY)%
    \figvectPDD-2[#1,#4]\vecunit@{-2}{-2}\Figg@tXY{-2}\arct@n\v@lmax(\v@lX,\v@lY)%
    \v@lmin=\rdT@deg\v@lmin\v@lmax=\rdT@deg\v@lmax%
    \v@leur=#2pt\maxim@m{\mili@u}{-\v@leur}{\v@leur}%
    \edef\r@dius{\repdecn@mb{\mili@u}}}
\ctr@ld@f\def\Ps@rcercBz#1;#2(#3,#4){\Ps@rellBz#1;#2,#2(#3,#4,0)}
\ctr@ld@f\def\Ps@rellBz#1;#2,#3(#4,#5,#6){%
    \ellBB@x#1;#2,#3(#4,#5,#6)\BdingB@xfalse%
    \c@lNbarcs{#4}{#5}\v@leur=#4pt\setc@ntr@l{2}\figptell-13::#1;#2,#3(#4,#6)%
    \f@gnewpath\PSwrit@cmd{-13}{\c@mmoveto}{\fwf@g}%
    \s@mme=\z@\bcl@rellBz#1;#2,#3(#6)\BdingB@xtrue}
\ctr@ld@f\def\bcl@rellBz#1;#2,#3(#4){\relax%
    \ifnum\s@mme<\p@rtent\advance\s@mme\@ne%
    \advance\v@leur\delt@\edef\@ngle{\repdecn@mb\v@leur}\figptell-14::#1;#2,#3(\@ngle,#4)%
    \advance\v@leur\delt@\edef\@ngle{\repdecn@mb\v@leur}\figptell-15::#1;#2,#3(\@ngle,#4)%
    \advance\v@leur\delt@\edef\@ngle{\repdecn@mb\v@leur}\figptell-16::#1;#2,#3(\@ngle,#4)%
    \figptscontrolDD-18[-13,-14,-15,-16]%
    \PSwrit@cmd{-18}{}{\fwf@g}\PSwrit@cmd{-17}{}{\fwf@g}%
    \PSwrit@cmd{-16}{\c@mcurveto}{\fwf@g}%
    \figptcopyDD-13:/-16/\bcl@rellBz#1;#2,#3(#4)\fi}
\ctr@ld@f\def\Ps@rell#1;#2,#3(#4,#5,#6){\ellBB@x#1;#2,#3(#4,#5,#6)%
    \f@gnewpath{\v@lmin=#2\unit@\v@lmin=\ptT@ptps\v@lmin%
    \v@lmax=#3\unit@\v@lmax=\ptT@ptps\v@lmax\BdingB@xfalse%
    \PSwrit@cmd{#1}%
    {#6\space\repdecn@mb{\v@lmin}\space\repdecn@mb{\v@lmax}\space #4\space #5\space ellipse}{\fwf@g}}%
    \global\Use@llipsetrue}
\ctr@ln@m\figdrawarcell
\ctr@ld@f\def\Q@arcellDD#1;#2,#3(#4,#5,#6){{\ifCUR@PS\ifGR@cri%
    \PSc@mment{arcellDD Center=#1 ; XRad=#2, YRad=#3 (Ang1=#4, Ang2=#5, Inclination=#6)}%
    \iffillm@de\Ps@rell#1;#2,#3(#4,#5,#6)%
    \f@gfill%
    \else\Ps@rell#1;#2,#3(#4,#5,#6)\f@gstroke\fi%
    \PSc@mment{End arcellDD}\fi\fi}}
\ctr@ld@f\def\Q@arcellTD#1;#2,#3(#4,#5,#6){{\ifCUR@PS\ifGR@cri\s@uvc@ntr@l\et@tpsarcellTD%
    \PSc@mment{arcellTD Center=#1 ; XRad=#2, YRad=#3 (Ang1=#4, Ang2=#5, Inclination=#6)}%
    \setc@ntr@l{2}\figpttraC -8:=#1/#2,0,0/\figpttraC -7:=#1/0,#3,0/%
    \figvectC -4(0,0,1)\figptsrot -8=-8,-7/#1,#6,-4/\Q@arcellPATD#1,-8,-7(#4,#5)%
    \PSc@mment{End arcellTD}\resetc@ntr@l\et@tpsarcellTD\fi\fi}}
\ctr@ln@m\figdrawarcellPA
\ctr@ld@f\def\Q@arcellPADD#1,#2,#3(#4,#5){{\ifCUR@PS\ifGR@cri\s@uvc@ntr@l\et@tpsarcellPADD%
    \PSc@mment{arcellPADD Center=#1,PtAxis1=#2,PtAxis2=#3 (Ang1=#4, Ang2=#5)}%
    \setc@ntr@l{2}\figvectPDD-1[#1,#2]\vecunit@DD{-1}{-1}\v@lX=\ptT@unit@\result@t%
    \edef\XR@d{\repdecn@mb{\v@lX}}\Figg@tXY{-1}\arct@n\v@lmin(\v@lX,\v@lY)%
    \v@lmin=\rdT@deg\v@lmin\edef\Inclin@{\repdecn@mb{\v@lmin}}%
    \figgetdist\YR@d[#1,#3]\Q@arcellDD#1;\XR@d,\YR@d(#4,#5,\Inclin@)%
    \PSc@mment{End arcellPADD}\resetc@ntr@l\et@tpsarcellPADD\fi\fi}}
\ctr@ld@f\def\Q@arcellPATD#1,#2,#3(#4,#5){{\ifCUR@PS\ifGR@cri\s@uvc@ntr@l\et@tpsarcellPATD%
    \PSc@mment{arcellPATD Center=#1,PtAxis1=#2,PtAxis2=#3 (Ang1=#4, Ang2=#5)}%
    \iffillm@de\Ps@rellPATD#1,#2,#3(#4,#5)%
    \f@gfill%
    \else\Ps@rellPATD#1,#2,#3(#4,#5)\f@gstroke\fi%
    \PSc@mment{End arcellPATD}\resetc@ntr@l\et@tpsarcellPATD\fi\fi}}
\ctr@ld@f\def\Ps@rellPATD#1,#2,#3(#4,#5){\let\c@lprojSP=\relax%
    \setc@ntr@l{2}\figvectPTD-1[#1,#2]\figvectPTD-2[#1,#3]\c@lNbarcs{#4}{#5}%
    \v@leur=#4pt\c@lptellP{#1}{-1}{-2}\Figptpr@j-5:/-3/%
    \f@gnewpath\PSwrit@cmdS{-5}{\c@mmoveto}{\fwf@g}{\X@un}{\Y@un}%
    \edef\C@nt@r{#1}\s@mme=\z@\bcl@rellPATD}
\ctr@ld@f\def\bcl@rellPATD{\relax%
    \ifnum\s@mme<\p@rtent\advance\s@mme\@ne%
    \advance\v@leur\delt@\c@lptellP{\C@nt@r}{-1}{-2}\Figptpr@j-4:/-3/%
    \advance\v@leur\delt@\c@lptellP{\C@nt@r}{-1}{-2}\Figptpr@j-6:/-3/%
    \advance\v@leur\delt@\c@lptellP{\C@nt@r}{-1}{-2}\Figptpr@j-3:/-3/%
    \v@lX=\z@\v@lY=\z@\Figtr@nptDD{-5}{-5}\Figtr@nptDD{2}{-3}%
    \divide\v@lX\@vi\divide\v@lY\@vi%
    \Figtr@nptDD{3}{-4}\Figtr@nptDD{-1.5}{-6}\v@lmin=\v@lX\v@lmax=\v@lY%
    \v@lX=\z@\v@lY=\z@\Figtr@nptDD{2}{-5}\Figtr@nptDD{-5}{-3}%
    \divide\v@lX\@vi\divide\v@lY\@vi\Figtr@nptDD{-1.5}{-4}\Figtr@nptDD{3}{-6}%
    \BdingB@xfalse%
    \Figp@intregDD-4:(\v@lmin,\v@lmax)\PSwrit@cmdS{-4}{}{\fwf@g}{\X@de}{\Y@de}%
    \Figp@intregDD-4:(\v@lX,\v@lY)\PSwrit@cmdS{-4}{}{\fwf@g}{\X@tr}{\Y@tr}%
    \BdingB@xtrue\PSwrit@cmdS{-3}{\c@mcurveto}{\fwf@g}{\X@qu}{\Y@qu}%
    \B@zierBB@x{1}{\Y@un}(\X@un,\X@de,\X@tr,\X@qu)%
    \B@zierBB@x{2}{\X@un}(\Y@un,\Y@de,\Y@tr,\Y@qu)%
    \edef\X@un{\X@qu}\edef\Y@un{\Y@qu}\figptcopyDD-5:/-3/\bcl@rellPATD\fi}
\ctr@ld@f\def\c@lNbarcs#1#2{%
    \delt@=#2pt\advance\delt@-#1pt\maxim@m{\v@lmax}{\delt@}{-\delt@}%
    \v@leur=\v@lmax\divide\v@leur45 \p@rtentiere{\p@rtent}{\v@leur}\advance\p@rtent\@ne%
    \s@mme=\p@rtent\multiply\s@mme\thr@@\divide\delt@\s@mme}
\ctr@ld@f\def\figdrawarcellPP#1,#2,#3[#4,#5]{{\ifCUR@PS\ifGR@cri\s@uvc@ntr@l\et@tpsarcellPP%
    \PSc@mment{arcellPP Center=#1,PtAxis1=#2,PtAxis2=#3 [Point1=#4, Point2=#5]}%
    \setc@ntr@l{2}\figvectP-2[#1,#3]\vecunit@{-2}{-2}\v@lmin=\result@t%
    \invers@{\v@lmax}{\v@lmin}%
    \figvectP-1[#1,#2]\vecunit@{-1}{-1}\v@leur=\result@t%
    \v@leur=\repdecn@mb{\v@lmax}\v@leur\edef\AsB@{\repdecn@mb{\v@leur}}
    \c@lAngle{#1}{#4}{\v@lmin}\edef\@ngdeb{\repdecn@mb{\v@lmin}}%
    \c@lAngle{#1}{#5}{\v@lmax}\ifdim\v@lmin>\v@lmax\advance\v@lmax\DePI@deg\fi%
    \edef\@ngfin{\repdecn@mb{\v@lmax}}\figdrawarcellPA#1,#2,#3(\@ngdeb,\@ngfin)%
    \PSc@mment{End arcellPP}\resetc@ntr@l\et@tpsarcellPP\fi\fi}}
\ctr@ld@f\def\c@lAngle#1#2#3{\figvectP-3[#1,#2]%
    \c@lproscal\delt@[-3,-1]\c@lproscal\v@leur[-3,-2]%
    \v@leur=\AsB@\v@leur\arct@n#3(\delt@,\v@leur)#3=\rdT@deg#3}
\ctr@ln@w{newif}\if@rrowratio\@rrowratiotrue
\ctr@ln@w{newif}\if@rrowhfill
\ctr@ln@w{newif}\if@rrowhout
\ctr@ld@f\def\Psset@rrowhe@d#1=#2|{\keln@mun#1|%
    \def\n@mref{a}\ifx\l@debut\n@mref\update@ttr\D@FTarrowheadangle\Q@s@tarrowheadangle{#2}\else
    \def\n@mref{f}\ifx\l@debut\n@mref\update@ttr\D@FTarrowheadfill\Q@s@tarrowheadfill{#2}\else
    \def\n@mref{l}\ifx\l@debut\n@mref\update@ttr\D@FTarrowheadlength\Q@s@tarrowheadlength{#2}\else
    \def\n@mref{o}\ifx\l@debut\n@mref\update@ttr\D@FTarrowheadout\Q@s@tarrowheadout{#2}\else
    \def\n@mref{r}\ifx\l@debut\n@mref\update@ttr\D@FTarrowheadratio\Q@s@tarrowheadratio{#2}\else
    \W@rnmesAttr{figset arrowhead}{#1}\fi\fi\fi\fi\fi}
\ctr@ln@m\@rrowheadangle
\ctr@ln@m\C@AHANG \ctr@ln@m\S@AHANG \ctr@ln@m\UNSS@N
\ctr@ld@f\def\Q@s@tarrowheadangle#1{\edef\@rrowheadangle{#1}{\c@ssin{\C@}{\S@}{#1}%
    \xdef\C@AHANG{\C@}\xdef\S@AHANG{\S@}\v@lmax=\S@ pt%
    \invers@{\v@leur}{\v@lmax}\maxim@m{\v@leur}{\v@leur}{-\v@leur}%
    \xdef\UNSS@N{\the\v@leur}}}
\ctr@ld@f\def\Q@s@tarrowheadfill#1{\expandafter\set@rrowhfill#1:}
\ctr@ld@f\def\set@rrowhfill#1#2:{\if#1n\@rrowhfillfalse\else\@rrowhfilltrue\fi}
\ctr@ld@f\def\Q@s@tarrowheadout#1{\expandafter\set@rrowhout#1:}
\ctr@ld@f\def\set@rrowhout#1#2:{\if#1n\@rrowhoutfalse\else\@rrowhouttrue\fi}
\ctr@ln@m\@rrowheadlength
\ctr@ld@f\def\Q@s@tarrowheadlength#1{\edef\@rrowheadlength{#1}\@rrowratiofalse}
\ctr@ln@m\@rrowheadratio
\ctr@ld@f\def\Q@s@tarrowheadratio#1{\edef\@rrowheadratio{#1}\@rrowratiotrue}
\ctr@ln@m\D@FTarrowheadlength
\ctr@ld@f\def\figresetarrowhead{%
    \Q@s@tarrowheadangle{\D@FTarrowheadangle}%
    \Q@s@tarrowheadfill{\D@FTarrowheadfill}%
    \Q@s@tarrowheadout{\D@FTarrowheadout}%
    \Q@s@tarrowheadratio{\D@FTarrowheadratio}%
    \d@fm@cdim\D@FTarrowheadlength{\D@FTh@rdahlength}
    \Q@s@tarrowheadlength{\D@FTarrowheadlength}}
\ctr@ld@f\def\D@FTarrowheadratio{0.1}
\ctr@ld@f\def\D@FTarrowheadangle{20}
\ctr@ld@f\def\D@FTarrowheadfill{no}
\ctr@ld@f\def\D@FTarrowheadout{no}
\ctr@ld@f\def\D@FTh@rdahlength{8pt}
\ctr@ln@m\figdrawarrow
\ctr@ld@f\def\Q@arrowDD[#1,#2]{{\ifCUR@PS\ifGR@cri\s@uvc@ntr@l\et@tpsarrow%
    \PSc@mment{arrowDD [Pt1,Pt2]=[#1,#2]}\Q@s@tfillmode{no}%
    \Q@arrowheadDD[#1,#2]\setc@ntr@l{2}\figdrawline[#1,-3]%
    \PSc@mment{End arrowDD}\resetc@ntr@l\et@tpsarrow\fi\fi}}
\ctr@ld@f\def\Q@arrowTD[#1,#2]{{\ifCUR@PS\ifGR@cri\s@uvc@ntr@l\et@tpsarrowTD%
    \PSc@mment{arrowTD [Pt1,Pt2]=[#1,#2]}\resetc@ntr@l{2}%
    \Figptpr@j-5:/#1/\Figptpr@j-6:/#2/\let\c@lprojSP=\relax\Q@arrowDD[-5,-6]%
    \PSc@mment{End arrowTD}\resetc@ntr@l\et@tpsarrowTD\fi\fi}}
\ctr@ln@m\figdrawarrowhead
\ctr@ld@f\def\Q@arrowheadDD[#1,#2]{{\ifCUR@PS\ifGR@cri\s@uvc@ntr@l\et@tpsarrowheadDD%
    \if@rrowhfill\def\@hangle{-\@rrowheadangle}\else\def\@hangle{\@rrowheadangle}\fi%
    \if@rrowratio%
    \if@rrowhout\def\@hratio{-\@rrowheadratio}\else\def\@hratio{\@rrowheadratio}\fi%
    \PSc@mment{arrowheadDD Ratio=\@hratio, Angle=\@hangle, [Pt1,Pt2]=[#1,#2]}%
    \Ps@rrowhead\@hratio,\@hangle[#1,#2]%
    \else%
    \if@rrowhout\def\@hlength{-\@rrowheadlength}\else\def\@hlength{\@rrowheadlength}\fi%
    \PSc@mment{arrowheadDD Length=\@hlength, Angle=\@hangle, [Pt1,Pt2]=[#1,#2]}%
    \Ps@rrowheadfd\@hlength,\@hangle[#1,#2]%
    \fi%
    \PSc@mment{End arrowheadDD}\resetc@ntr@l\et@tpsarrowheadDD\fi\fi}}
\ctr@ld@f\def\Q@arrowheadTD[#1,#2]{{\ifCUR@PS\ifGR@cri\s@uvc@ntr@l\et@tpsarrowheadTD%
    \PSc@mment{arrowheadTD [Pt1,Pt2]=[#1,#2]}\resetc@ntr@l{2}%
    \Figptpr@j-5:/#1/\Figptpr@j-6:/#2/\let\c@lprojSP=\relax\Q@arrowheadDD[-5,-6]%
    \PSc@mment{End arrowheadTD}\resetc@ntr@l\et@tpsarrowheadTD\fi\fi}}
\ctr@ld@f\def\Ps@rrowhead#1,#2[#3,#4]{\v@leur=#1\p@\maxim@m{\v@leur}{\v@leur}{-\v@leur}%
    \ifdim\v@leur>\Cepsil@n{
    \PSc@mment{@rrowhead Ratio=#1, Angle=#2, [Pt1,Pt2]=[#3,#4]}\v@leur=\UNSS@N%
    \v@leur=\CUR@width\v@leur\v@leur=\ptpsT@pt\v@leur\delt@=.5\v@leur
    \setc@ntr@l{2}\figvectPDD-3[#4,#3]%
    \Figg@tXY{-3}\v@lX=#1\v@lX\v@lY=#1\v@lY\Figv@ctCreg-3(\v@lX,\v@lY)%
    \vecunit@{-4}{-3}\mili@u=\result@t%
    \ifdim#2pt>\z@\v@lXa=-\C@AHANG\delt@%
     \edef\c@ef{\repdecn@mb{\v@lXa}}\figpttraDD-3:=-3/\c@ef,-4/\fi%
    \edef\c@ef{\repdecn@mb{\delt@}}%
    \v@lXa=\mili@u\v@lXa=\C@AHANG\v@lXa%
    \v@lYa=\ptpsT@pt\p@\v@lYa=\CUR@width\v@lYa\v@lYa=\sDcc@ngle\v@lYa%
    \advance\v@lXa-\v@lYa\gdef\sDcc@ngle{0}%
    \ifdim\v@lXa>\v@leur\edef\c@efendpt{\repdecn@mb{\v@leur}}%
    \else\edef\c@efendpt{\repdecn@mb{\v@lXa}}\fi%
    \Figg@tXY{-3}\v@lmin=\v@lX\v@lmax=\v@lY%
    \v@lXa=\C@AHANG\v@lmin\v@lYa=\S@AHANG\v@lmax\advance\v@lXa\v@lYa%
    \v@lYa=-\S@AHANG\v@lmin\v@lX=\C@AHANG\v@lmax\advance\v@lYa\v@lX%
    \setc@ntr@l{1}\Figg@tXY{#4}\advance\v@lX\v@lXa\advance\v@lY\v@lYa%
    \setc@ntr@l{2}\Figp@intregDD-2:(\v@lX,\v@lY)%
    \v@lXa=\C@AHANG\v@lmin\v@lYa=-\S@AHANG\v@lmax\advance\v@lXa\v@lYa%
    \v@lYa=\S@AHANG\v@lmin\v@lX=\C@AHANG\v@lmax\advance\v@lYa\v@lX%
    \setc@ntr@l{1}\Figg@tXY{#4}\advance\v@lX\v@lXa\advance\v@lY\v@lYa%
    \setc@ntr@l{2}\Figp@intregDD-1:(\v@lX,\v@lY)%
    \ifdim#2pt<\z@\fillm@detrue\figdrawline[-2,#4,-1]
    \else\figptstraDD-3=#4,-2,-1/\c@ef,-4/\s@uvdash{\typ@dash}\Q@s@tdash{\D@FTdash}%
    \figdrawline[-2,-3,-1]\Q@s@tdash{\typ@dash}\fi
    \ifdim#1pt>\z@\figpttraDD-3:=#4/\c@efendpt,-4/\else\figptcopyDD-3:/#4/\fi%
    \PSc@mment{End @rrowhead}}\fi}
\ctr@ld@f\def\sDcc@ngle{0}
\ctr@ld@f\def\Ps@rrowheadfd#1,#2[#3,#4]{{%
    \PSc@mment{@rrowheadfd Length=#1, Angle=#2, [Pt1,Pt2]=[#3,#4]}%
    \setc@ntr@l{2}\figvectPDD-1[#3,#4]\n@rmeucDD{\v@leur}{-1}\v@leur=\ptT@unit@\v@leur%
    \invers@{\v@leur}{\v@leur}\v@leur=#1\v@leur\edef\R@tio{\repdecn@mb{\v@leur}}%
    \Ps@rrowhead\R@tio,#2[#3,#4]\PSc@mment{End @rrowheadfd}}}
\ctr@ln@m\figdrawarrowBezier
\ctr@ld@f\def\Q@arrowBezierDD[#1,#2,#3,#4]{{\ifCUR@PS\ifGR@cri\s@uvc@ntr@l\et@tpsarrowBezierDD%
    \PSc@mment{arrowBezierDD Control points=#1,#2,#3,#4}\setc@ntr@l{2}%
    \if@rrowratio\c@larclengthDD\v@leur,10[#1,#2,#3,#4]\else\v@leur=\z@\fi%
    \Ps@rrowB@zDD\v@leur[#1,#2,#3,#4]%
    \PSc@mment{End arrowBezierDD}\resetc@ntr@l\et@tpsarrowBezierDD\fi\fi}}
\ctr@ld@f\def\Q@arrowBezierTD[#1,#2,#3,#4]{{\ifCUR@PS\ifGR@cri\s@uvc@ntr@l\et@tpsarrowBezierTD%
    \PSc@mment{arrowBezierTD Control points=#1,#2,#3,#4}\resetc@ntr@l{2}%
    \Figptpr@j-7:/#1/\Figptpr@j-8:/#2/\Figptpr@j-9:/#3/\Figptpr@j-10:/#4/%
    \let\c@lprojSP=\relax\ifnum\CUR@proj<\tw@\Q@arrowBezierDD[-7,-8,-9,-10]%
    \else\f@gnewpath\PSwrit@cmd{-7}{\c@mmoveto}{\fwf@g}%
    \if@rrowratio\c@larclengthDD\mili@u,10[-7,-8,-9,-10]\else\mili@u=\z@\fi%
    \p@rtent=\NBz@rcs\advance\p@rtent\m@ne\subB@zierTD\p@rtent[#1,#2,#3,#4]%
    \f@gstroke%
    \advance\v@lmin\p@rtent\delt@
    \v@leur=\v@lmin\advance\v@leur0.33333 \delt@\edef\unti@rs{\repdecn@mb{\v@leur}}%
    \v@leur=\v@lmin\advance\v@leur0.66666 \delt@\edef\deti@rs{\repdecn@mb{\v@leur}}%
    \figptcopyDD-8:/-10/\c@lsubBzarc\unti@rs,\deti@rs[#1,#2,#3,#4]%
    \figptcopyDD-8:/-4/\figptcopyDD-9:/-3/\Ps@rrowB@zDD\mili@u[-7,-8,-9,-10]\fi%
    \PSc@mment{End arrowBezierTD}\resetc@ntr@l\et@tpsarrowBezierTD\fi\fi}}
\ctr@ld@f\def\c@larclengthDD#1,#2[#3,#4,#5,#6]{{\p@rtent=#2\figptcopyDD-5:/#3/%
    \delt@=\p@\divide\delt@\p@rtent\c@rre=\z@\v@leur=\z@\s@mme=\z@%
    \loop\ifnum\s@mme<\p@rtent\advance\s@mme\@ne\advance\v@leur\delt@%
    \edef\T@{\repdecn@mb{\v@leur}}\figptBezierDD-6::\T@[#3,#4,#5,#6]%
    \figvectPDD-1[-5,-6]\n@rmeucDD{\mili@u}{-1}\advance\c@rre\mili@u%
    \figptcopyDD-5:/-6/\repeat\global\result@t=\ptT@unit@\c@rre}#1=\result@t}
\ctr@ld@f\def\Ps@rrowB@zDD#1[#2,#3,#4,#5]{{\Q@s@tfillmode{no}%
    \if@rrowratio\delt@=\@rrowheadratio#1\else\delt@=\@rrowheadlength pt\fi%
    \v@leur=\C@AHANG\delt@\edef\R@dius{\repdecn@mb{\v@leur}}%
    \FigptintercircB@zDD-5::0,\R@dius[#5,#4,#3,#2]%
    \Q@s@tarrowheadlength{\repdecn@mb{\delt@}}\Q@arrowheadDD[-5,#5]%
    \let\n@rmeuc=\n@rmeucDD\figgetdist\R@dius[#5,-3]%
    \FigptintercircB@zDD-6::0,\R@dius[#5,#4,#3,#2]%
    \figptBezierDD-5::0.33333[#5,#4,#3,#2]\figptBezierDD-3::0.66666[#5,#4,#3,#2]%
    \figptscontrolDD-5[-6,-5,-3,#2]\Q@BezierDD1[-6,-5,-4,#2]}}
\ctr@ln@m\figdrawarrowcirc
\ctr@ld@f\def\Q@arrowcircDD#1;#2(#3,#4){{\ifCUR@PS\ifGR@cri\s@uvc@ntr@l\et@tpsarrowcircDD%
    \PSc@mment{arrowcircDD Center=#1 ; Radius=#2 (Ang1=#3,Ang2=#4)}%
    \Q@s@tfillmode{no}\Pscirc@rrowhead#1;#2(#3,#4)%
    \setc@ntr@l{2}\figvectPDD -4[#1,-3]\vecunit@{-4}{-4}%
    \Figg@tXY{-4}\arct@n\v@lmin(\v@lX,\v@lY)%
    \v@lmin=\rdT@deg\v@lmin\v@leur=#4pt\advance\v@leur-\v@lmin%
    \maxim@m{\v@leur}{\v@leur}{-\v@leur}%
    \ifdim\v@leur>\DemiPI@deg\relax\ifdim\v@lmin<#4pt\advance\v@lmin\DePI@deg%
    \else\advance\v@lmin-\DePI@deg\fi\fi\edef\ar@ngle{\repdecn@mb{\v@lmin}}%
    \ifdim#3pt<#4pt\figdrawarccirc#1;#2(#3,\ar@ngle)\else\figdrawarccirc#1;#2(\ar@ngle,#3)\fi%
    \PSc@mment{End arrowcircDD}\resetc@ntr@l\et@tpsarrowcircDD\fi\fi}}
\ctr@ld@f\def\Q@arrowcircTD#1,#2,#3;#4(#5,#6){{\ifCUR@PS\ifGR@cri\s@uvc@ntr@l\et@tpsarrowcircTD%
    \PSc@mment{arrowcircTD Center=#1,P1=#2,P2=#3 ; Radius=#4 (Ang1=#5, Ang2=#6)}%
    \resetc@ntr@l{2}\c@lExtAxes#1,#2,#3(#4)\let\c@lprojSP=\relax%
    \figvectPTD-11[#1,-4]\figvectPTD-12[#1,-5]\c@lNbarcs{#5}{#6}%
    \if@rrowratio\v@lmax=\degT@rd\v@lmax\edef\D@lpha{\repdecn@mb{\v@lmax}}\fi%
    \advance\p@rtent\m@ne\mili@u=\z@%
    \v@leur=#5pt\c@lptellP{#1}{-11}{-12}\Figptpr@j-9:/-3/%
    \f@gnewpath\PSwrit@cmdS{-9}{\c@mmoveto}{\fwf@g}{\X@un}{\Y@un}%
    \edef\C@nt@r{#1}\s@mme=\z@\bcl@rcircTD\f@gstroke%
    \advance\v@leur\delt@\c@lptellP{#1}{-11}{-12}\Figptpr@j-5:/-3/%
    \advance\v@leur\delt@\c@lptellP{#1}{-11}{-12}\Figptpr@j-6:/-3/%
    \advance\v@leur\delt@\c@lptellP{#1}{-11}{-12}\Figptpr@j-10:/-3/%
    \figptscontrolDD-8[-9,-5,-6,-10]%
    \if@rrowratio\c@lcurvradDD0.5[-9,-8,-7,-10]\advance\mili@u\result@t%
    \maxim@m{\mili@u}{\mili@u}{-\mili@u}\mili@u=\ptT@unit@\mili@u%
    \mili@u=\D@lpha\mili@u\advance\p@rtent\@ne\divide\mili@u\p@rtent\fi%
    \Ps@rrowB@zDD\mili@u[-9,-8,-7,-10]%
    \PSc@mment{End arrowcircTD}\resetc@ntr@l\et@tpsarrowcircTD\fi\fi}}
\ctr@ld@f\def\bcl@rcircTD{\relax%
    \ifnum\s@mme<\p@rtent\advance\s@mme\@ne%
    \advance\v@leur\delt@\c@lptellP{\C@nt@r}{-11}{-12}\Figptpr@j-5:/-3/%
    \advance\v@leur\delt@\c@lptellP{\C@nt@r}{-11}{-12}\Figptpr@j-6:/-3/%
    \advance\v@leur\delt@\c@lptellP{\C@nt@r}{-11}{-12}\Figptpr@j-10:/-3/%
    \figptscontrolDD-8[-9,-5,-6,-10]\BdingB@xfalse%
    \PSwrit@cmdS{-8}{}{\fwf@g}{\X@de}{\Y@de}\PSwrit@cmdS{-7}{}{\fwf@g}{\X@tr}{\Y@tr}%
    \BdingB@xtrue\PSwrit@cmdS{-10}{\c@mcurveto}{\fwf@g}{\X@qu}{\Y@qu}%
    \if@rrowratio\c@lcurvradDD0.5[-9,-8,-7,-10]\advance\mili@u\result@t\fi%
    \B@zierBB@x{1}{\Y@un}(\X@un,\X@de,\X@tr,\X@qu)%
    \B@zierBB@x{2}{\X@un}(\Y@un,\Y@de,\Y@tr,\Y@qu)%
    \edef\X@un{\X@qu}\edef\Y@un{\Y@qu}\figptcopyDD-9:/-10/\bcl@rcircTD\fi}
\ctr@ld@f\def\Pscirc@rrowhead#1;#2(#3,#4){{%
    \PSc@mment{circ@rrowhead Center=#1 ; Radius=#2 (Ang1=#3,Ang2=#4)}%
    \v@leur=#2\unit@\edef\s@glen{\repdecn@mb{\v@leur}}\v@lY=\z@\v@lX=\v@leur%
    \resetc@ntr@l{2}\Figv@ctCreg-3(\v@lX,\v@lY)\figpttraDD-5:=#1/1,-3/%
    \figptrotDD-5:=-5/#1,#4/%
    \figvectPDD-3[#1,-5]\Figg@tXY{-3}\v@leur=\v@lX%
    \ifdim#3pt<#4pt\v@lX=\v@lY\v@lY=-\v@leur\else\v@lX=-\v@lY\v@lY=\v@leur\fi%
    \Figv@ctCreg-3(\v@lX,\v@lY)\vecunit@{-3}{-3}%
    \if@rrowratio\v@leur=#4pt\advance\v@leur-#3pt\maxim@m{\mili@u}{-\v@leur}{\v@leur}%
    \mili@u=\degT@rd\mili@u\v@leur=\s@glen\mili@u\edef\s@glen{\repdecn@mb{\v@leur}}%
    \mili@u=#2\mili@u\mili@u=\@rrowheadratio\mili@u\else\mili@u=\@rrowheadlength pt\fi%
    \figpttraDD-6:=-5/\s@glen,-3/\v@leur=#2pt\v@leur=2\v@leur%
    \invers@{\v@leur}{\v@leur}\c@rre=\repdecn@mb{\v@leur}\mili@u
    \mili@u=\c@rre\mili@u=\repdecn@mb{\c@rre}\mili@u%
    \v@leur=\p@\advance\v@leur-\mili@u
    \invers@{\mili@u}{2\v@leur}\delt@=\c@rre\delt@=\repdecn@mb{\mili@u}\delt@%
    \xdef\sDcc@ngle{\repdecn@mb{\delt@}}
    \sqrt@{\mili@u}{\v@leur}\arct@n\v@leur(\mili@u,\c@rre)%
    \v@leur=\rdT@deg\v@leur
    \ifdim#3pt<#4pt\v@leur=-\v@leur\fi%
    \if@rrowhout\v@leur=-\v@leur\fi\edef\cor@ngle{\repdecn@mb{\v@leur}}%
    \figptrotDD-6:=-6/-5,\cor@ngle/\Q@arrowheadDD[-6,-5]%
    \PSc@mment{End circ@rrowhead}}}
\ctr@ln@m\figdrawarrowcircP
\ctr@ld@f\def\Q@arrowcircPDD#1;#2[#3,#4]{{\ifCUR@PS\ifGR@cri%
    \PSc@mment{arrowcircPDD Center=#1; Radius=#2, [P1=#3,P2=#4]}%
    \s@uvc@ntr@l\et@tpsarrowcircPDD\Ps@ngleparam#1;#2[#3,#4]%
    \ifdim\v@leur>\z@\ifdim\v@lmin>\v@lmax\advance\v@lmax\DePI@deg\fi%
    \else\ifdim\v@lmin<\v@lmax\advance\v@lmin\DePI@deg\fi\fi%
    \edef\@ngdeb{\repdecn@mb{\v@lmin}}\edef\@ngfin{\repdecn@mb{\v@lmax}}%
    \figdrawarrowcirc#1;\r@dius(\@ngdeb,\@ngfin)%
    \PSc@mment{End arrowcircPDD}\resetc@ntr@l\et@tpsarrowcircPDD\fi\fi}}
\ctr@ld@f\def\Q@arrowcircPTD#1;#2[#3,#4,#5]{{\ifCUR@PS\ifGR@cri\s@uvc@ntr@l\et@tpsarrowcircPTD%
    \PSc@mment{arrowcircPTD Center=#1; Radius=#2, [P1=#3,P2=#4,P3=#5]}%
    \figgetangleTD\@ngfin[#1,#3,#4,#5]\v@leur=#2pt%
    \maxim@m{\mili@u}{-\v@leur}{\v@leur}\edef\r@dius{\repdecn@mb{\mili@u}}%
    \ifdim\v@leur<\z@\v@lmax=\@ngfin pt\advance\v@lmax-\DePI@deg%
    \edef\@ngfin{\repdecn@mb{\v@lmax}}\fi\Q@arrowcircTD#1,#3,#5;\r@dius(0,\@ngfin)%
    \PSc@mment{End arrowcircPTD}\resetc@ntr@l\et@tpsarrowcircPTD\fi\fi}}
\ctr@ld@f\def\figdrawaxes#1(#2){{\ifCUR@PS\ifGR@cri\s@uvc@ntr@l\et@tpsaxes%
    \PSc@mment{axes Origin=#1 Range=(#2)}\an@lys@xes#2,:\resetc@ntr@l{2}%
    \ifx\t@xt@\empty\ifTr@isDim\Q@@xes#1(0,#2,0,#2,0,#2)\else\Q@@xes#1(0,#2,0,#2)\fi%
    \else\Q@@xes#1(#2)\fi\PSc@mment{End axes}\resetc@ntr@l\et@tpsaxes\fi\fi}}
\ctr@ld@f\def\an@lys@xes#1,#2:{\def\t@xt@{#2}}
\ctr@ln@m\Q@@xes
\ctr@ld@f\def\Q@@xesDD#1(#2,#3,#4,#5){%
    \figpttraC-5:=#1/#2,0/\figpttraC-6:=#1/#3,0/\Q@arrowDD[-5,-6]%
    \figpttraC-5:=#1/0,#4/\figpttraC-6:=#1/0,#5/\Q@arrowDD[-5,-6]}
\ctr@ld@f\def\Q@@xesTD#1(#2,#3,#4,#5,#6,#7){%
    \figpttraC-7:=#1/#2,0,0/\figpttraC-8:=#1/#3,0,0/\Q@arrowTD[-7,-8]%
    \figpttraC-7:=#1/0,#4,0/\figpttraC-8:=#1/0,#5,0/\Q@arrowTD[-7,-8]%
    \figpttraC-7:=#1/0,0,#6/\figpttraC-8:=#1/0,0,#7/\Q@arrowTD[-7,-8]}
\ctr@ln@m\newGr@FN
\ctr@ld@f\def\newGr@FNPDF#1{\s@mme=\Gr@FNb\advance\s@mme\@ne\xdef\Gr@FNb{\number\s@mme}}
\ctr@ld@f\def\newGr@FNDVI#1{\newGr@FNPDF{}\xdef#1{\jobname GI\Gr@FNb.anx}}
\ctr@ld@f\def\figdrawbegin#1{\newGr@FN\DefGIfilen@me\gdef\@utoFN{0}%
    \def\t@xt@{#1}\relax\ifx\t@xt@\empty\GRupdatem@detrue%
    \gdef\@utoFN{1}\Psb@ginfig\DefGIfilen@me\else\expandafter\Psb@ginfigNu@#1 :\fi}
\ctr@ld@f\def\Psb@ginfigNu@#1 #2:{\def\t@xt@{#1}\relax\ifx\t@xt@\empty\def\t@xt@{#2}%
    \ifx\t@xt@\empty\GRupdatem@detrue\gdef\@utoFN{1}\Psb@ginfig\DefGIfilen@me%
    \else\Psb@ginfigNu@#2:\fi\else\Psb@ginfig{#1}\fi}
\ctr@ln@m\PSfilen@me \ctr@ln@m\auxfilen@me
\ctr@ld@f\def\Psb@ginfig#1{\ifCUR@PS\else%
    \edef\PSfilen@me{#1}\edef\auxfilen@me{\jobname.anx}%
    \ifGRupdatem@de\GR@critrue\else\openin\frf@g=\PSfilen@me\relax%
    \ifeof\frf@g\GR@critrue\else\GR@crifalse\fi\closein\frf@g\fi%
    \CUR@PStrue\c@ldefproj\expandafter\setupd@te\D@FTupdate:%
    \ifGR@cri\initb@undb@x%
    \immediate\openout\fwf@g=\auxfilen@me\initpss@ttings\fi%
    \fi}
\ctr@ld@f\def\Gr@FNb{0}
\ctr@ld@f\def\figforTeXFileno{\Gr@FNb}
\ctr@ld@f\def\figforTeXFigno{0 }
\ctr@ld@f\def\figforTeXnextFigno{1 }
\ctr@ld@f\edef\DefGIfilen@me{\jobname GI.anx}
\ctr@ld@f\def\initpss@ttings{\figreset{altitude,arrowhead,curve,general,flowchart,mesh,trimesh}%
    \Use@llipsefalse}
\ctr@ld@f\def\B@zierBB@x#1#2(#3,#4,#5,#6){{\c@rre=\t@n\epsil@n
    \v@lmax=#4\advance\v@lmax-#5\v@lmax=\thr@@\v@lmax\advance\v@lmax#6\advance\v@lmax-#3%
    \mili@u=#4\mili@u=-\tw@\mili@u\advance\mili@u#3\advance\mili@u#5%
    \v@lmin=#4\advance\v@lmin-#3\maxim@m{\v@leur}{-\v@lmax}{\v@lmax}%
    \maxim@m{\delt@}{-\mili@u}{\mili@u}\maxim@m{\v@leur}{\v@leur}{\delt@}%
    \maxim@m{\delt@}{-\v@lmin}{\v@lmin}\maxim@m{\v@leur}{\v@leur}{\delt@}%
    \ifdim\v@leur>\c@rre\invers@{\v@leur}{\v@leur}\edef\Uns@rM@x{\repdecn@mb{\v@leur}}%
    \v@lmax=\Uns@rM@x\v@lmax\mili@u=\Uns@rM@x\mili@u\v@lmin=\Uns@rM@x\v@lmin%
    \maxim@m{\v@leur}{-\v@lmax}{\v@lmax}\ifdim\v@leur<\c@rre%
    \maxim@m{\v@leur}{-\mili@u}{\mili@u}\ifdim\v@leur<\c@rre\else%
    \invers@{\mili@u}{\mili@u}\v@leur=-0.5\v@lmin%
    \v@leur=\repdecn@mb{\mili@u}\v@leur\m@jBBB@x{\v@leur}{#1}{#2}(#3,#4,#5,#6)\fi%
    \else\delt@=\repdecn@mb{\mili@u}\mili@u\v@leur=\repdecn@mb{\v@lmax}\v@lmin%
    \advance\delt@-\v@leur\ifdim\delt@<\z@\else\invers@{\v@lmax}{\v@lmax}%
    \edef\Uns@rAp{\repdecn@mb{\v@lmax}}\sqrt@{\delt@}{\delt@}%
    \v@leur=-\mili@u\advance\v@leur\delt@\v@leur=\Uns@rAp\v@leur%
    \m@jBBB@x{\v@leur}{#1}{#2}(#3,#4,#5,#6)%
    \v@leur=-\mili@u\advance\v@leur-\delt@\v@leur=\Uns@rAp\v@leur%
    \m@jBBB@x{\v@leur}{#1}{#2}(#3,#4,#5,#6)\fi\fi\fi}}
\ctr@ld@f\def\m@jBBB@x#1#2#3(#4,#5,#6,#7){{\relax\ifdim#1>\z@\ifdim#1<\p@%
    \edef\T@{\repdecn@mb{#1}}\v@lX=\p@\advance\v@lX-#1\edef\UNmT@{\repdecn@mb{\v@lX}}%
    \v@lX=#4\v@lY=#5\v@lZ=#6\v@lXa=#7\v@lX=\UNmT@\v@lX\advance\v@lX\T@\v@lY%
    \v@lY=\UNmT@\v@lY\advance\v@lY\T@\v@lZ\v@lZ=\UNmT@\v@lZ\advance\v@lZ\T@\v@lXa%
    \v@lX=\UNmT@\v@lX\advance\v@lX\T@\v@lY\v@lY=\UNmT@\v@lY\advance\v@lY\T@\v@lZ%
    \v@lX=\UNmT@\v@lX\advance\v@lX\T@\v@lY%
    \ifcase#2\or\v@lY=#3\or\v@lY=\v@lX\v@lX=#3\fi\b@undb@x{\v@lX}{\v@lY}\fi\fi}}
\ctr@ld@f\def\PsB@zier#1[#2]{{\f@gnewpath%
    \s@mme=\z@\def\list@num{#2,0}\extrairelepremi@r\p@int\de\list@num%
    \PSwrit@cmdS{\p@int}{\c@mmoveto}{\fwf@g}{\X@un}{\Y@un}\p@rtent=#1\bclB@zier}}
\ctr@ld@f\def\bclB@zier{\relax%
    \ifnum\s@mme<\p@rtent\advance\s@mme\@ne\BdingB@xfalse%
    \extrairelepremi@r\p@int\de\list@num\PSwrit@cmdS{\p@int}{}{\fwf@g}{\X@de}{\Y@de}%
    \extrairelepremi@r\p@int\de\list@num\PSwrit@cmdS{\p@int}{}{\fwf@g}{\X@tr}{\Y@tr}%
    \BdingB@xtrue%
    \extrairelepremi@r\p@int\de\list@num\PSwrit@cmdS{\p@int}{\c@mcurveto}{\fwf@g}{\X@qu}{\Y@qu}%
    \B@zierBB@x{1}{\Y@un}(\X@un,\X@de,\X@tr,\X@qu)%
    \B@zierBB@x{2}{\X@un}(\Y@un,\Y@de,\Y@tr,\Y@qu)%
    \edef\X@un{\X@qu}\edef\Y@un{\Y@qu}\bclB@zier\fi}
\ctr@ln@m\figdrawBezier
\ctr@ld@f\def\Q@BezierDD#1[#2]{\ifCUR@PS\ifGR@cri%
    \PSc@mment{BezierDD N arcs=#1, Control points=#2}%
    \iffillm@de\PsB@zier#1[#2]%
    \f@gfill%
    \else\PsB@zier#1[#2]\f@gstroke\fi%
    \PSc@mment{End BezierDD}\fi\fi}
\ctr@ln@m\et@tpsBezierTD
\ctr@ld@f\def\Q@BezierTD#1[#2]{\ifCUR@PS\ifGR@cri\s@uvc@ntr@l\et@tpsBezierTD%
    \PSc@mment{BezierTD N arcs=#1, Control points=#2}%
    \iffillm@de\PsB@zierTD#1[#2]%
    \f@gfill%
    \else\PsB@zierTD#1[#2]\f@gstroke\fi%
    \PSc@mment{End BezierTD}\resetc@ntr@l\et@tpsBezierTD\fi\fi}
\ctr@ld@f\def\PsB@zierTD#1[#2]{\ifnum\CUR@proj<\tw@\PsB@zier#1[#2]\else\PsB@zier@TD#1[#2]\fi}
\ctr@ld@f\def\PsB@zier@TD#1[#2]{{\f@gnewpath%
    \s@mme=\z@\def\list@num{#2,0}\extrairelepremi@r\p@int\de\list@num%
    \let\c@lprojSP=\relax\setc@ntr@l{2}\Figptpr@j-7:/\p@int/%
    \PSwrit@cmd{-7}{\c@mmoveto}{\fwf@g}%
    \loop\ifnum\s@mme<#1\advance\s@mme\@ne\extrairelepremi@r\p@intun\de\list@num%
    \extrairelepremi@r\p@intde\de\list@num\extrairelepremi@r\p@inttr\de\list@num%
    \subB@zierTD\NBz@rcs[\p@int,\p@intun,\p@intde,\p@inttr]\edef\p@int{\p@inttr}\repeat}}
\ctr@ld@f\def\subB@zierTD#1[#2,#3,#4,#5]{\delt@=\p@\divide\delt@\NBz@rcs\v@lmin=\z@%
    {\Figg@tXY{-7}\edef\X@un{\the\v@lX}\edef\Y@un{\the\v@lY}%
    \s@mme=\z@\loop\ifnum\s@mme<#1\advance\s@mme\@ne%
    \v@leur=\v@lmin\advance\v@leur0.33333 \delt@\edef\unti@rs{\repdecn@mb{\v@leur}}%
    \v@leur=\v@lmin\advance\v@leur0.66666 \delt@\edef\deti@rs{\repdecn@mb{\v@leur}}%
    \advance\v@lmin\delt@\edef\trti@rs{\repdecn@mb{\v@lmin}}%
    \figptBezierTD-8::\trti@rs[#2,#3,#4,#5]\Figptpr@j-8:/-8/%
    \c@lsubBzarc\unti@rs,\deti@rs[#2,#3,#4,#5]\BdingB@xfalse%
    \PSwrit@cmdS{-4}{}{\fwf@g}{\X@de}{\Y@de}\PSwrit@cmdS{-3}{}{\fwf@g}{\X@tr}{\Y@tr}%
    \BdingB@xtrue\PSwrit@cmdS{-8}{\c@mcurveto}{\fwf@g}{\X@qu}{\Y@qu}%
    \B@zierBB@x{1}{\Y@un}(\X@un,\X@de,\X@tr,\X@qu)%
    \B@zierBB@x{2}{\X@un}(\Y@un,\Y@de,\Y@tr,\Y@qu)%
    \edef\X@un{\X@qu}\edef\Y@un{\Y@qu}\figptcopyDD-7:/-8/\repeat}}
\ctr@ld@f\def\NBz@rcs{2}
\ctr@ld@f\def\c@lsubBzarc#1,#2[#3,#4,#5,#6]{\figptBezierTD-5::#1[#3,#4,#5,#6]%
    \figptBezierTD-6::#2[#3,#4,#5,#6]\Figptpr@j-4:/-5/\Figptpr@j-5:/-6/%
    \figptscontrolDD-4[-7,-4,-5,-8]}
\ctr@ln@m\figdrawcirc
\ctr@ld@f\def\Q@circDD#1(#2){\ifCUR@PS\ifGR@cri\PSc@mment{circDD Center=#1 (Radius=#2)}%
    \Q@arccircDD#1;#2(0,360)\PSc@mment{End circDD}\fi\fi}
\ctr@ld@f\def\Q@circTD#1,#2,#3(#4){\ifCUR@PS\ifGR@cri%
    \PSc@mment{circTD Center=#1,P1=#2,P2=#3 (Radius=#4)}%
    \Q@arccircTD#1,#2,#3;#4(0,360)\PSc@mment{End circTD}\fi\fi}
\ctr@ln@m\p@urcent
{\catcode`\%=12\gdef\p@urcent{
\ctr@ld@f\def\PSc@mment#1{\ifGRdebugm@de\immediate\write\fwf@g{\p@urcent\space#1}\fi}
\ctr@ln@m\acc@louv \ctr@ln@m\acc@lfer
{\catcode`\[=1\catcode`\{=12\gdef\acc@louv[{}}
{\catcode`\]=2\catcode`\}=12\gdef\acc@lfer{}]]
\ctr@ld@f\def\PSdict@{\ifUse@llipse%
    \immediate\write\fwf@g{/ellipsedict 9 dict def ellipsedict /mtrx matrix put}%
    \immediate\write\fwf@g{/ellipse \acc@louv ellipsedict begin}%
    \immediate\write\fwf@g{ /endangle exch def /startangle exch def}%
    \immediate\write\fwf@g{ /yrad exch def /xrad exch def}%
    \immediate\write\fwf@g{ /rotangle exch def /y exch def /x exch def}%
    \immediate\write\fwf@g{ /savematrix mtrx currentmatrix def}%
    \immediate\write\fwf@g{ x y translate rotangle rotate xrad yrad scale}%
    \immediate\write\fwf@g{ 0 0 1 startangle endangle arc}%
    \immediate\write\fwf@g{ savematrix setmatrix end\acc@lfer def}%
    \fi\PShe@der{EndProlog}}
\ctr@ld@f\def\Pssetc@rve#1=#2|{\keln@mun#1|%
    \def\n@mref{r}\ifx\l@debut\n@mref\update@ttr\D@FTroundness\Q@s@troundness{#2}\else
    \W@rnmesAttr{figset curve}{#1}\fi}
\ctr@ln@m\curv@roundness
\ctr@ld@f\def\Q@s@troundness#1{\edef\curv@roundness{#1}}
\ctr@ld@f\def\D@FTroundness{0.2} 
\ctr@ln@m\figdrawcurve
\ctr@ld@f\def\Q@curveDD[#1]{{\ifCUR@PS\ifGR@cri\PSc@mment{curveDD Points=#1}%
    \s@uvc@ntr@l\et@tpscurveDD%
    \iffillm@de\Psc@rveDD\curv@roundness[#1]%
    \f@gfill%
    \else\Psc@rveDD\curv@roundness[#1]\f@gstroke\fi%
    \PSc@mment{End curveDD}\resetc@ntr@l\et@tpscurveDD\fi\fi}}
\ctr@ld@f\def\Q@curveTD[#1]{{\ifCUR@PS\ifGR@cri%
    \PSc@mment{curveTD Points=#1}\s@uvc@ntr@l\et@tpscurveTD\let\c@lprojSP=\relax%
    \iffillm@de\Psc@rveTD\curv@roundness[#1]%
    \f@gfill%
    \else\Psc@rveTD\curv@roundness[#1]\f@gstroke\fi%
    \PSc@mment{End curveTD}\resetc@ntr@l\et@tpscurveTD\fi\fi}}
\ctr@ld@f\def\Psc@rveDD#1[#2]{%
    \def\list@num{#2}\extrairelepremi@r\Ak@\de\list@num%
    \extrairelepremi@r\Ai@\de\list@num\extrairelepremi@r\Aj@\de\list@num%
    \f@gnewpath\PSwrit@cmdS{\Ai@}{\c@mmoveto}{\fwf@g}{\X@un}{\Y@un}%
    \setc@ntr@l{2}\figvectPDD -1[\Ak@,\Aj@]%
    \@ecfor\Ak@:=\list@num\do{\figpttraDD-2:=\Ai@/#1,-1/\BdingB@xfalse%
       \PSwrit@cmdS{-2}{}{\fwf@g}{\X@de}{\Y@de}%
       \figvectPDD -1[\Ai@,\Ak@]\figpttraDD-2:=\Aj@/-#1,-1/%
       \PSwrit@cmdS{-2}{}{\fwf@g}{\X@tr}{\Y@tr}\BdingB@xtrue%
       \PSwrit@cmdS{\Aj@}{\c@mcurveto}{\fwf@g}{\X@qu}{\Y@qu}%
       \B@zierBB@x{1}{\Y@un}(\X@un,\X@de,\X@tr,\X@qu)%
       \B@zierBB@x{2}{\X@un}(\Y@un,\Y@de,\Y@tr,\Y@qu)%
       \edef\X@un{\X@qu}\edef\Y@un{\Y@qu}\edef\Ai@{\Aj@}\edef\Aj@{\Ak@}}}
\ctr@ld@f\def\Psc@rveTD#1[#2]{\ifnum\CUR@proj<\tw@\Psc@rvePPTD#1[#2]\else\Psc@rveCPTD#1[#2]\fi}
\ctr@ld@f\def\Psc@rvePPTD#1[#2]{\setc@ntr@l{2}%
    \def\list@num{#2}\extrairelepremi@r\Ak@\de\list@num\Figptpr@j-5:/\Ak@/%
    \extrairelepremi@r\Ai@\de\list@num\Figptpr@j-3:/\Ai@/%
    \extrairelepremi@r\Aj@\de\list@num\Figptpr@j-4:/\Aj@/%
    \f@gnewpath\PSwrit@cmdS{-3}{\c@mmoveto}{\fwf@g}{\X@un}{\Y@un}%
    \figvectPDD -1[-5,-4]%
    \@ecfor\Ak@:=\list@num\do{\Figptpr@j-5:/\Ak@/\figpttraDD-2:=-3/#1,-1/%
       \BdingB@xfalse\PSwrit@cmdS{-2}{}{\fwf@g}{\X@de}{\Y@de}%
       \figvectPDD -1[-3,-5]\figpttraDD-2:=-4/-#1,-1/%
       \PSwrit@cmdS{-2}{}{\fwf@g}{\X@tr}{\Y@tr}\BdingB@xtrue%
       \PSwrit@cmdS{-4}{\c@mcurveto}{\fwf@g}{\X@qu}{\Y@qu}%
       \B@zierBB@x{1}{\Y@un}(\X@un,\X@de,\X@tr,\X@qu)%
       \B@zierBB@x{2}{\X@un}(\Y@un,\Y@de,\Y@tr,\Y@qu)%
       \edef\X@un{\X@qu}\edef\Y@un{\Y@qu}\figptcopyDD-3:/-4/\figptcopyDD-4:/-5/}}
\ctr@ld@f\def\Psc@rveCPTD#1[#2]{\setc@ntr@l{2}%
    \def\list@num{#2}\extrairelepremi@r\Ak@\de\list@num%
    \extrairelepremi@r\Ai@\de\list@num\extrairelepremi@r\Aj@\de\list@num%
    \Figptpr@j-7:/\Ai@/%
    \f@gnewpath\PSwrit@cmd{-7}{\c@mmoveto}{\fwf@g}%
    \figvectPTD -9[\Ak@,\Aj@]%
    \@ecfor\Ak@:=\list@num\do{\figpttraTD-10:=\Ai@/#1,-9/%
       \figvectPTD -9[\Ai@,\Ak@]\figpttraTD-11:=\Aj@/-#1,-9/%
       \subB@zierTD\NBz@rcs[\Ai@,-10,-11,\Aj@]\edef\Ai@{\Aj@}\edef\Aj@{\Ak@}}}
\ctr@ld@f\def\figdrawend{\ifCUR@PS\ifGR@cri\immediate\closeout\fwf@g%
    \immediate\openout\fwf@g=\PSfilen@me\relax%
    \ifPDFm@ke\PSBdingB@x\else%
    \immediate\write\fwf@g{\p@urcent\string!PS-Adobe-2.0 EPSF-2.0}%
    \PShe@der{Creator\string: TeX (fig4tex.tex)}%
    \PShe@der{Title\string: \PSfilen@me}%
    \PShe@der{CreationDate\string: \the\day/\the\month/\the\year}%
    \PSBdingB@x%
    \PShe@der{EndComments}\PSdict@\fi%
    \immediate\write\fwf@g{\c@mgsave}%
    \openin\frf@g=\auxfilen@me\c@pypsfile\fwf@g\frf@g\closein\frf@g%
    \immediate\write\fwf@g{\c@mgrestore}%
    \PSc@mment{End of file.}\immediate\closeout\fwf@g%
    \immediate\openout\fwf@g=\auxfilen@me\immediate\closeout\fwf@g%
    \immediate\write16{File \PSfilen@me\space created.}\fi\fi\CUR@PSfalse\GR@critrue}
\ctr@ld@f\def\PShe@der#1{\immediate\write\fwf@g{\p@urcent\p@urcent#1}}
\ctr@ld@f\def\PSBdingB@x{{\v@lX=\ptT@ptps\c@@rdXmin\v@lY=\ptT@ptps\c@@rdYmin%
     \v@lXa=\ptT@ptps\c@@rdXmax\v@lYa=\ptT@ptps\c@@rdYmax%
     \PShe@der{BoundingBox\string: \repdecn@mb{\v@lX}\space\repdecn@mb{\v@lY}%
     \space\repdecn@mb{\v@lXa}\space\repdecn@mb{\v@lYa}}}}
\ctr@ld@f\def\figdrawfcconnect[#1]{{\ifCUR@PS\ifGR@cri\PSc@mment{fcconnect Points=#1}%
    \Q@s@tfillmode{no}\s@uvc@ntr@l\et@tpsfcconnect\resetc@ntr@l{2}%
    \fcc@nnect@[#1]\resetc@ntr@l\et@tpsfcconnect\PSc@mment{End fcconnect}\fi\fi}}
\ctr@ld@f\def\fcc@nnect@[#1]{\let\N@rm=\n@rmeucDD\def\list@num{#1}%
    \extrairelepremi@r\Ai@\de\list@num\edef\pr@m{\Ai@}\v@leur=\z@\p@rtent=\@ne\c@llgtot%
    \ifcase\fclin@typ@\edef\list@num{[\pr@m,#1,\Ai@}\expandafter\figdrawcurve\list@num]%
    \else\ifdim\fclin@r@d\p@>\z@\Pslin@conge[#1]\else\figdrawline[#1]\fi\fi%
    \v@leur=\@rrowp@s\v@leur\edef\list@num{#1,\Ai@,0}%
    \extrairelepremi@r\Ai@\de\list@num\mili@u=\epsil@n\c@llgpart%
    \advance\mili@u-\epsil@n\advance\mili@u-\delt@\advance\v@leur-\mili@u%
    \ifcase\fclin@typ@\invers@\mili@u\delt@%
    \ifnum\@rrowr@fpt>\z@\advance\delt@-\v@leur\v@leur=\delt@\fi%
    \v@leur=\repdecn@mb\v@leur\mili@u\edef\v@lt{\repdecn@mb\v@leur}%
    \extrairelepremi@r\Ak@\de\list@num%
    \figvectPDD-1[\pr@m,\Aj@]\figpttraDD-6:=\Ai@/\curv@roundness,-1/%
    \figvectPDD-1[\Ak@,\Ai@]\figpttraDD-7:=\Aj@/\curv@roundness,-1/%
    \delt@=\@rrowheadlength\p@\delt@=\C@AHANG\delt@\edef\R@dius{\repdecn@mb{\delt@}}%
    \ifcase\@rrowr@fpt%
    \FigptintercircB@zDD-8::\v@lt,\R@dius[\Ai@,-6,-7,\Aj@]\Q@arrowheadDD[-5,-8]\else%
    \FigptintercircB@zDD-8::\v@lt,\R@dius[\Aj@,-7,-6,\Ai@]\Q@arrowheadDD[-8,-5]\fi%
    \else\advance\delt@-\v@leur%
    \p@rtentiere{\p@rtent}{\delt@}\edef\C@efun{\the\p@rtent}%
    \p@rtentiere{\p@rtent}{\v@leur}\edef\C@efde{\the\p@rtent}%
    \figptbaryDD-5:[\Ai@,\Aj@;\C@efun,\C@efde]\ifcase\@rrowr@fpt%
    \delt@=\@rrowheadlength\unit@\delt@=\C@AHANG\delt@\edef\t@ille{\repdecn@mb{\delt@}}%
    \figvectPDD-2[\Ai@,\Aj@]\vecunit@{-2}{-2}\figpttraDD-5:=-5/\t@ille,-2/\fi%
    \Q@arrowheadDD[\Ai@,-5]\fi}
\ctr@ld@f\def\c@llgtot{\@ecfor\Aj@:=\list@num\do{\figvectP-1[\Ai@,\Aj@]\N@rm\delt@{-1}%
    \advance\v@leur\delt@\advance\p@rtent\@ne\edef\Ai@{\Aj@}}}
\ctr@ld@f\def\c@llgpart{\extrairelepremi@r\Aj@\de\list@num\figvectP-1[\Ai@,\Aj@]\N@rm\delt@{-1}%
    \advance\mili@u\delt@\ifdim\mili@u<\v@leur\edef\pr@m{\Ai@}\edef\Ai@{\Aj@}\c@llgpart\fi}
\ctr@ld@f\def\Pslin@conge[#1]{\ifnum\p@rtent>\tw@{\def\list@num{#1}%
    \extrairelepremi@r\Ai@\de\list@num\extrairelepremi@r\Aj@\de\list@num%
    \figptcopy-6:/\Ai@/\figvectP-3[\Ai@,\Aj@]\vecunit@{-3}{-3}\v@lmax=\result@t%
    \@ecfor\Ak@:=\list@num\do{\figvectP-4[\Aj@,\Ak@]\vecunit@{-4}{-4}%
    \minim@m\v@lmin\v@lmax\result@t\v@lmax=\result@t%
    \det@rm\delt@[-3,-4]\maxim@m\mili@u{\delt@}{-\delt@}\ifdim\mili@u>\Cepsil@n%
    \ifdim\delt@>\z@\figgetangleDD\Angl@[\Aj@,\Ak@,\Ai@]\else%
    \figgetangleDD\Angl@[\Aj@,\Ai@,\Ak@]\fi%
    \v@leur=\PI@deg\advance\v@leur-\Angl@\p@\divide\v@leur\tw@%
    \edef\Angl@{\repdecn@mb\v@leur}\c@ssin{\C@}{\S@}{\Angl@}\v@leur=\fclin@r@d\unit@%
    \v@leur=\S@\v@leur\mili@u=\C@\p@\invers@\mili@u\mili@u%
    \v@leur=\repdecn@mb{\mili@u}\v@leur%
    \minim@m\v@leur\v@leur\v@lmin\edef\t@ille{\repdecn@mb{\v@leur}}%
    \figpttra-5:=\Aj@/-\t@ille,-3/\figdrawline[-6,-5]\figpttra-6:=\Aj@/\t@ille,-4/%
    \figvectNVDD-3[-3]\figvectNVDD-8[-4]\inters@cDD-7:[-5,-3;-6,-8]%
    \ifdim\delt@>\z@\figdrawarccircP-7;\fclin@r@d[-5,-6]\else\figdrawarccircP-7;\fclin@r@d[-6,-5]\fi%
    \else\figdrawline[-6,\Aj@]\figptcopy-6:/\Aj@/\fi
    \edef\Ai@{\Aj@}\edef\Aj@{\Ak@}\figptcopy-3:/-4/}\figdrawline[-6,\Aj@]}\else\figdrawline[#1]\fi}
\ctr@ld@f\def\figdrawfcnode[#1]#2{{\ifCUR@PS\ifGR@cri\PSc@mment{fcnode Points=#1}%
    \s@uvc@ntr@l\et@tpsfcnode\resetc@ntr@l{2}%
    \def\t@xt@{#2}\ifx\t@xt@\empty\def\g@tt@xt{\setbox\Gb@x=\hbox{\Figg@tT{\p@int}}}%
    \else\def\g@tt@xt{\setbox\Gb@x=\hbox{#2}}\fi%
    \v@lmin=\h@rdfcXp@dd\advance\v@lmin\Xp@dd\unit@\multiply\v@lmin\tw@%
    \v@lmax=\h@rdfcYp@dd\advance\v@lmax\Yp@dd\unit@\multiply\v@lmax\tw@%
    \Figv@ctCreg-8(\unit@,-\unit@)\def\list@num{#1}%
    \delt@=\CUR@width bp\divide\delt@\tw@%
    \fcn@de\PSc@mment{End fcnode}\resetc@ntr@l\et@tpsfcnode\fi\fi}}
\ctr@ld@f\def\d@butn@de{\g@tt@xt\v@lX=\wd\Gb@x%
    \v@lY=\ht\Gb@x\advance\v@lY\dp\Gb@x\advance\v@lX\v@lmin\advance\v@lY\v@lmax}
\ctr@ld@f\def\fcn@deE{%
    \@ecfor\p@int:=\list@num\do{\d@butn@de\v@lX=\unssqrttw@\v@lX\v@lY=\unssqrttw@\v@lY%
    \ifdim\thickn@ss\p@>\z@
    \v@lXa=\v@lX\advance\v@lXa\delt@\v@lXa=\ptT@unit@\v@lXa\edef\XR@d{\repdecn@mb\v@lXa}%
    \v@lYa=\v@lY\advance\v@lYa\delt@\v@lYa=\ptT@unit@\v@lYa\edef\YR@d{\repdecn@mb\v@lYa}%
    \arct@n\v@leur(\v@lXa,\v@lYa)\v@leur=\rdT@deg\v@leur\edef\@nglde{\repdecn@mb\v@leur}%
    {\c@lptellDD-2::\p@int;\XR@d,\YR@d(\@nglde)}
    \advance\v@leur-\PI@deg\edef\@nglun{\repdecn@mb\v@leur}%
    {\c@lptellDD-3::\p@int;\XR@d,\YR@d(\@nglun)}%
    \figptstra-6=-3,-2,\p@int/\thickn@ss,-8/\Q@s@tfillmode{yes}%
    \Pss@tspecifSt{color=\DDV@thickcolor}%
    \figdrawline[-2,-3,-6,-5]\figdrawarcell-4;\XR@d,\YR@d(\@nglun,\@nglde,0)%
    \Psrest@reSt{color=\DDV@thickcolor}\fi
    \v@lX=\ptT@unit@\v@lX\v@lY=\ptT@unit@\v@lY%
    \edef\XR@d{\repdecn@mb\v@lX}\edef\YR@d{\repdecn@mb\v@lY}%
    \Q@s@tfillmode{yes}\Pss@tspecifSt{color=\fcbgc@lor}%
    \figdrawarcell\p@int;\XR@d,\YR@d(0,360,0)%
    \Q@s@tfillmode{no}\Psrest@reSt{color=\fcbgc@lor}\figdrawarcell\p@int;\XR@d,\YR@d(0,360,0)}}
\ctr@ld@f\def\fcn@deL{\delt@=\ptT@unit@\delt@\edef\t@ille{\repdecn@mb\delt@}%
    \@ecfor\p@int:=\list@num\do{\Figg@tXYa{\p@int}\d@butn@de%
    \ifdim\v@lX>\v@lY\itis@Ktrue\else\itis@Kfalse\fi%
    \advance\v@lXa-\v@lX\Figp@intreg-1:(\v@lXa,\v@lYa)%
    \advance\v@lXa\v@lX\advance\v@lYa-\v@lY\Figp@intreg-2:(\v@lXa,\v@lYa)%
    \advance\v@lXa\v@lX\advance\v@lYa\v@lY\Figp@intreg-3:(\v@lXa,\v@lYa)%
    \advance\v@lXa-\v@lX\advance\v@lYa\v@lY\Figp@intreg-4:(\v@lXa,\v@lYa)%
    \ifdim\thickn@ss\p@>\z@
    \Figg@tXYa{\p@int}\Q@s@tfillmode{yes}\Pss@tspecifSt{color=\DDV@thickcolor}%
    \c@lpt@xt{-1}{-4}\c@lpt@xt@\v@lXa\v@lYa\v@lX\v@lY\c@rre\delt@%
    \Figp@intregDD-9:(\v@lZ,\v@lYa)\Figp@intregDD-11:(\v@lZa,\v@lYa)%
    \c@lpt@xt{-4}{-3}\c@lpt@xt@\v@lYa\v@lXa\v@lY\v@lX\delt@\c@rre%
    \Figp@intregDD-12:(\v@lXa,\v@lZ)\Figp@intregDD-10:(\v@lXa,\v@lZa)%
    \ifitis@K\figptstra-7=-9,-10,-11/\thickn@ss,-8/\figdrawline[-9,-11,-5,-6,-7]\else%
    \figptstra-7=-10,-11,-12/\thickn@ss,-8/\figdrawline[-10,-12,-5,-6,-7]\fi%
    \Psrest@reSt{color=\DDV@thickcolor}\fi
    \Q@s@tfillmode{yes}\Pss@tspecifSt{color=\fcbgc@lor}\figdrawline[-1,-2,-3,-4]%
    \Q@s@tfillmode{no}\Psrest@reSt{color=\fcbgc@lor}\figdrawline[-1,-2,-3,-4,-1]}}
\ctr@ld@f\def\c@lpt@xt#1#2{\figvectN-7[#1,#2]\vecunit@{-7}{-7}\figpttra-5:=#1/\t@ille,-7/%
    \figvectP-7[#1,#2]\Figg@tXY{-7}\c@rre=\v@lX\delt@=\v@lY\Figg@tXY{-5}}
\ctr@ld@f\def\c@lpt@xt@#1#2#3#4#5#6{\v@lZ=#6\invers@{\v@lZ}{\v@lZ}\v@leur=\repdecn@mb{#5}\v@lZ%
    \v@lZ=#2\advance\v@lZ-#4\mili@u=\repdecn@mb{\v@leur}\v@lZ%
    \v@lZ=#3\advance\v@lZ\mili@u\v@lZa=-\v@lZ\advance\v@lZa\tw@#1}
\ctr@ld@f\def\fcn@deR{\@ecfor\p@int:=\list@num\do{\Figg@tXYa{\p@int}\d@butn@de%
    \advance\v@lXa-0.5\v@lX\advance\v@lYa-0.5\v@lY\Figp@intreg-1:(\v@lXa,\v@lYa)%
    \advance\v@lXa\v@lX\Figp@intreg-2:(\v@lXa,\v@lYa)%
    \advance\v@lYa\v@lY\Figp@intreg-3:(\v@lXa,\v@lYa)%
    \advance\v@lXa-\v@lX\Figp@intreg-4:(\v@lXa,\v@lYa)%
    \ifdim\thickn@ss\p@>\z@
    \Q@s@tfillmode{yes}\Pss@tspecifSt{color=\DDV@thickcolor}%
    \Figv@ctCreg-5(-\delt@,-\delt@)\figpttra-9:=-1/1,-5/%
    \Figv@ctCreg-5(\delt@,-\delt@)\figpttra-10:=-2/1,-5/%
    \Figv@ctCreg-5(\delt@,\delt@)\figpttra-11:=-3/1,-5/%
    \figptstra-7=-9,-10,-11/\thickn@ss,-8/\figdrawline[-9,-11,-5,-6,-7]%
    \Psrest@reSt{color=\DDV@thickcolor}\fi
    \Q@s@tfillmode{yes}\Pss@tspecifSt{color=\fcbgc@lor}\figdrawline[-1,-2,-3,-4]%
    \Q@s@tfillmode{no}\Psrest@reSt{color=\fcbgc@lor}\figdrawline[-1,-2,-3,-4,-1]}}
\ctr@ld@f\def\Pssetfl@wchart#1=#2|{\keln@mtr#1|%
    \def\n@mref{arr}\ifx\l@debut\n@mref\expandafter\keln@mtr\l@suite|%
     \def\n@mref{owp}\ifx\l@debut\n@mref\update@ttr\D@FTfcarrowposition\P@setfcarrowposition{#2}\else
     \def\n@mref{owr}\ifx\l@debut\n@mref\update@ttr\D@FTfcarrowrefpt\P@setfcarrowrefpt{#2}\else
     \W@rnmesAttr{figset flowchart}{#1}\fi\fi\else%
    \def\n@mref{bgc}\ifx\l@debut\n@mref\update@ttr\D@FTfcbgcolor\P@setfcbgcolor{#2}\else
    \def\n@mref{lin}\ifx\l@debut\n@mref\update@ttr\D@FTfcline\P@setfcline{#2}\else
    \def\n@mref{pad}\ifx\l@debut\n@mref\update@ttr\D@FTfcxpadding\P@setfcxpadding{#2}%
                                       \update@ttr\D@FTfcypadding\P@setfcypadding{#2}\else
    \def\n@mref{rad}\ifx\l@debut\n@mref\update@ttr\D@FTfcradius\P@setfcradius{#2}\else
    \def\n@mref{sha}\ifx\l@debut\n@mref\update@ttr\D@FTfcshape\P@setfcshape{#2}\else
    \def\n@mref{thi}\ifx\l@debut\n@mref\expandafter\keln@mtr\l@suite|%
     \def\n@mref{ckc}\ifx\l@debut\n@mref\update@ttr\D@FTref\P@setfcthickcolor{#2}\else
     \def\n@mref{ckn}\ifx\l@debut\n@mref\update@ttr\D@FTfcthickness\P@setfcthickness{#2}\else
     \W@rnmesAttr{figset flowchart}{#1}\fi\fi\else%
    \def\n@mref{xpa}\ifx\l@debut\n@mref\update@ttr\D@FTfcxpadding\P@setfcxpadding{#2}\else
    \def\n@mref{ypa}\ifx\l@debut\n@mref\update@ttr\D@FTfcypadding\P@setfcypadding{#2}\else
    \W@rnmesAttr{figset flowchart}{#1}\fi\fi\fi\fi\fi\fi\fi\fi\fi}
\ctr@ln@m\@rrowp@s
\ctr@ld@f\def\P@setfcarrowposition#1{\edef\@rrowp@s{#1}}
\ctr@ln@m\@rrowr@fpt
\ctr@ld@f\def\P@setfcarrowrefpt#1{\setfcr@fpt#1|}
\ctr@ld@f\def\setfcr@fpt#1#2|{\if#1e\def\@rrowr@fpt{1}\else\def\@rrowr@fpt{0}\fi}
\ctr@ln@m\fcbgc@lor
\ctr@ld@f\def\P@setfcbgcolor#1{\edef\fcbgc@lor{#1}}
\ctr@ln@m\fclin@typ@
\ctr@ld@f\def\P@setfcline#1{\setfccurv@#1|}
\ctr@ld@f\def\setfccurv@#1#2|{\if#1c\def\fclin@typ@{0}\else\def\fclin@typ@{1}\fi}
\ctr@ln@m\fclin@r@d
\ctr@ld@f\def\P@setfcradius#1{\edef\fclin@r@d{#1}}
\ctr@ln@m\fcn@de \ctr@ln@m\fcsh@pe
\ctr@ln@m\h@rdfcXp@dd \ctr@ln@m\h@rdfcYp@dd
\ctr@ld@f\def\P@setfcshape#1{\setfcshap@#1|}
\ctr@ld@f\def\setfcshap@#1#2|{%
    \if#1e\let\fcn@de=\fcn@deE\def\h@rdfcXp@dd{4pt}\def\h@rdfcYp@dd{4pt}%
     \edef\fcsh@pe{ellipse}\else%
    \if#1l\let\fcn@de=\fcn@deL\def\h@rdfcXp@dd{4pt}\def\h@rdfcYp@dd{4pt}%
     \edef\fcsh@pe{lozenge}\else%
          \let\fcn@de=\fcn@deR\def\h@rdfcXp@dd{6pt}\def\h@rdfcYp@dd{6pt}%
     \edef\fcsh@pe{rectangle}\fi\fi}
\ctr@ln@m\DDV@thickcolor
\ctr@ld@f\def\P@setfcthickcolor#1{\edef\DDV@thickcolor{#1}}
\ctr@ln@m\thickn@ss
\ctr@ld@f\def\P@setfcthickness#1{\edef\thickn@ss{#1}}
\ctr@ln@m\Xp@dd
\ctr@ld@f\def\P@setfcxpadding#1{\edef\Xp@dd{#1}}
\ctr@ln@m\Yp@dd
\ctr@ld@f\def\P@setfcypadding#1{\edef\Yp@dd{#1}}
\ctr@ld@f\def\figdrawline[#1]{{\ifCUR@PS\ifGR@cri\PSc@mment{line Points=#1}%
    \let\figdrawlign@=\Pslign@P\Pslin@{#1}\PSc@mment{End line}\fi\fi}}
\ctr@ld@f\def\figdrawlineF#1{{\ifCUR@PS\ifGR@cri\PSc@mment{lineF Filename=#1}%
    \let\figdrawlign@=\Pslign@F\Pslin@{#1}\PSc@mment{End lineF}\fi\fi}}
\ctr@ld@f\def\figdrawlineC(#1){{\ifCUR@PS\ifGR@cri\PSc@mment{lineC}%
    \let\figdrawlign@=\Pslign@C\Pslin@{#1}\PSc@mment{End lineC}\fi\fi}}
\ctr@ld@f\def\Pslin@#1{\iffillm@de\figdrawlign@{#1}%
    \f@gfill%
    \else\figdrawlign@{#1}\ifx\derp@int\premp@int%
    \f@gclosestroke%
    \else\f@gstroke\fi\fi}
\ctr@ld@f\def\Pslign@P#1{\def\list@num{#1}\extrairelepremi@r\p@int\de\list@num%
    \edef\premp@int{\p@int}\f@gnewpath%
    \PSwrit@cmd{\p@int}{\c@mmoveto}{\fwf@g}%
    \@ecfor\p@int:=\list@num\do{\PSwrit@cmd{\p@int}{\c@mlineto}{\fwf@g}%
    \edef\derp@int{\p@int}}}
\ctr@ld@f\def\Pslign@F#1{\s@uvc@ntr@l\et@tPslign@F\setc@ntr@l{2}\openin\frf@g=#1\relax%
    \ifeof\frf@g\message{*** File #1 not found !}\end\else%
    \read\frf@g to\tr@c\edef\premp@int{\tr@c}\expandafter\extr@ctCF\tr@c:%
    \f@gnewpath\PSwrit@cmd{-1}{\c@mmoveto}{\fwf@g}%
    \loop\read\frf@g to\tr@c\ifeof\frf@g\mored@tafalse\else\mored@tatrue\fi%
    \ifmored@ta\expandafter\extr@ctCF\tr@c:\PSwrit@cmd{-1}{\c@mlineto}{\fwf@g}%
    \edef\derp@int{\tr@c}\repeat\fi\closein\frf@g\resetc@ntr@l\et@tPslign@F}
\ctr@ln@m\extr@ctCF
\ctr@ld@f\def\extr@ctCFDD#1 #2:{\v@lX=#1\unit@\v@lY=#2\unit@\Figp@intregDD-1:(\v@lX,\v@lY)}
\ctr@ld@f\def\extr@ctCFTD#1 #2 #3:{\v@lX=#1\unit@\v@lY=#2\unit@\v@lZ=#3\unit@%
    \Figp@intregTD-1:(\v@lX,\v@lY,\v@lZ)}
\ctr@ld@f\def\Pslign@C#1{\s@uvc@ntr@l\et@tPslign@C\setc@ntr@l{2}%
    \def\list@num{#1}\extrairelepremi@r\p@int\de\list@num%
    \edef\premp@int{\p@int}\f@gnewpath%
    \expandafter\Pslign@C@\p@int:\PSwrit@cmd{-1}{\c@mmoveto}{\fwf@g}%
    \@ecfor\p@int:=\list@num\do{\expandafter\Pslign@C@\p@int:%
    \PSwrit@cmd{-1}{\c@mlineto}{\fwf@g}\edef\derp@int{\p@int}}%
    \resetc@ntr@l\et@tPslign@C}
\ctr@ld@f\def\Pslign@C@#1 #2:{{\def\t@xt@{#1}\ifx\t@xt@\empty\Pslign@C@#2:
    \else\extr@ctCF#1 #2:\fi}}
\ctr@ld@f\def\Pssetm@sh#1=#2|{\keln@mde#1|%
    \def\n@mref{co}\ifx\l@debut\n@mref\update@ttr\D@FTref\P@setmeshcolor{#2}\else
    \def\n@mref{da}\ifx\l@debut\n@mref\update@ttr\D@FTref\P@setmeshdash{#2}\else
    \def\n@mref{di}\ifx\l@debut\n@mref\update@ttr\D@FTmeshdiag\Q@s@tmeshdiag{#2}\else
    \def\n@mref{wi}\ifx\l@debut\n@mref\update@ttr\D@FTref\P@setmeshwidth{#2}\else
    \W@rnmesAttr{figset mesh}{#1}\fi\fi\fi\fi}
\ctr@ln@m\c@ntrolmesh
\ctr@ld@f\def\Q@s@tmeshdiag#1{\edef\c@ntrolmesh{#1}}
\ctr@ld@f\def\D@FTmeshdiag{0}    
\ctr@ln@m\DDV@meshcolor
\ctr@ld@f\def\P@setmeshcolor#1{\edef\DDV@meshcolor{#1}}
\ctr@ln@m\DDV@meshdash
\ctr@ld@f\def\P@setmeshdash#1{\edef\DDV@meshdash{#1}}
\ctr@ln@m\DDV@meshwidth
\ctr@ld@f\def\P@setmeshwidth#1{\edef\DDV@meshwidth{#1}}
\ctr@ld@f\def\figdrawmesh#1,#2[#3,#4,#5,#6]{{\ifCUR@PS\ifGR@cri%
    \PSc@mment{mesh N1=#1, N2=#2, Quadrangle=[#3,#4,#5,#6]}\s@uvc@ntr@l\et@tpsmesh%
    \Pss@tspecifSt{color=\DDV@meshcolor,dash=\DDV@meshdash,width=\DDV@meshwidth}%
    \setc@ntr@l{2}%
    \ifnum#1>\@ne\Psmeshp@rt#1[#3,#4,#5,#6]\fi%
    \ifnum#2>\@ne\Psmeshp@rt#2[#4,#5,#6,#3]\fi%
    \ifnum\c@ntrolmesh>\z@\Psmeshdi@g#1,#2[#3,#4,#5,#6]\fi%
    \ifnum\c@ntrolmesh<\z@\Psmeshdi@g#2,#1[#4,#5,#6,#3]\fi%
    \Psrest@reSt{color=\DDV@meshcolor,dash=\DDV@meshdash,width=\DDV@meshwidth}%
    \figdrawline[#3,#4,#5,#6,#3]\PSc@mment{End mesh}\resetc@ntr@l\et@tpsmesh\fi\fi}}
\ctr@ld@f\def\Psmeshp@rt#1[#2,#3,#4,#5]{{\l@mbd@un=\@ne\l@mbd@de=#1\loop%
    \ifnum\l@mbd@un<#1\advance\l@mbd@de\m@ne\figptbary-1:[#2,#3;\l@mbd@de,\l@mbd@un]%
    \figptbary-2:[#5,#4;\l@mbd@de,\l@mbd@un]\figdrawline[-1,-2]\advance\l@mbd@un\@ne\repeat}}
\ctr@ld@f\def\Psmeshdi@g#1,#2[#3,#4,#5,#6]{\figptcopy-2:/#3/\figptcopy-3:/#6/%
    \l@mbd@un=\z@\l@mbd@de=#1\loop\ifnum\l@mbd@un<#1%
    \advance\l@mbd@un\@ne\advance\l@mbd@de\m@ne\figptcopy-1:/-2/\figptcopy-4:/-3/%
    \figptbary-2:[#3,#4;\l@mbd@de,\l@mbd@un]%
    \figptbary-3:[#6,#5;\l@mbd@de,\l@mbd@un]\Psmeshdi@gp@rt#2[-1,-2,-3,-4]\repeat}
\ctr@ld@f\def\Psmeshdi@gp@rt#1[#2,#3,#4,#5]{{\l@mbd@un=\z@\l@mbd@de=#1\loop%
    \ifnum\l@mbd@un<#1\figptbary-5:[#2,#5;\l@mbd@de,\l@mbd@un]%
    \advance\l@mbd@de\m@ne\advance\l@mbd@un\@ne%
    \figptbary-6:[#3,#4;\l@mbd@de,\l@mbd@un]\figdrawline[-5,-6]\repeat}}
\ctr@ln@m\figdrawnormal
\ctr@ld@f\def\Q@normalDD#1,#2[#3,#4]{{\ifCUR@PS\ifGR@cri%
    \PSc@mment{normal Length=#1, Lambda=#2 [Pt1,Pt2]=[#3,#4]}%
    \s@uvc@ntr@l\et@tpsnormal\resetc@ntr@l{2}\figptendnormal-6::#1,#2[#3,#4]%
    \figptcopyDD-5:/-1/\figdrawarrow[-5,-6]%
    \PSc@mment{End normal}\resetc@ntr@l\et@tpsnormal\fi\fi}}
\ctr@ld@f\def\figreset#1{\trtlis@rg{#1}{\Psreset@}}
\ctr@ld@f\def\Psreset@#1|{\def\t@xt@{#1}\ifx\t@xt@\empty\P@resetg@n
    \else\keln@mde#1|%
    \def\n@mref{al}\ifx\l@debut\n@mref%
        \figset altitude(blcolor=default,bldash=default,blwidth=default,%
        sqcolor=default,sqdash=default,sqwidth=default)\else
    \def\n@mref{ar}\ifx\l@debut\n@mref\figresetarrowhead\else
    \def\n@mref{cu}\ifx\l@debut\n@mref\figset curve(roundness=\D@FTroundness)\else
    \def\n@mref{ge}\ifx\l@debut\n@mref\P@resetg@n\else
    \def\n@mref{fl}\ifx\l@debut\n@mref%
        \figset flowchart(arrowp=\D@FTfcarrowposition,arrowr=\D@FTfcarrowrefpt,%
	bgcolor=\D@FTfcbgcolor,line=\D@FTfcline,radius=\D@FTfcradius,%
	shape=\D@FTfcshape,thickcolor=default,thickness=\D@FTfcthickness,%
	xpadd=\D@FTfcxpadding,ypadd=\D@FTfcypadding)\else
    \def\n@mref{me}\ifx\l@debut\n@mref\figset mesh(diag=\D@FTmeshdiag,%
        color=default,dash=default,width=default)\else
    \def\n@mref{tr}\ifx\l@debut\n@mref%
        \figset trimesh(color=default,dash=default,width=default)\else
    \W@rnmeskwd{figreset}{#1}\fi\fi\fi\fi\fi\fi\fi\fi}
\ctr@ld@f\def\P@resetg@n{\figset (color=\D@FTcolor,dash=\D@FTdash,fill=\D@FTfill,%
    join=\D@FTjoin,width=\D@FTwidth)}
\ctr@ld@f\def\figset#1(#2){\def\t@xt@{#1}\ifx\t@xt@\empty\trtlis@rg{#2}{\Pssetg@n}
    \else\keln@mde#1|%
    \def\n@mref{al}\ifx\l@debut\n@mref\trtlis@rg{#2}{\Psset@lti}\else
    \def\n@mref{ar}\ifx\l@debut\n@mref\trtlis@rg{#2}{\Psset@rrowhe@d}\else
    \def\n@mref{cu}\ifx\l@debut\n@mref\trtlis@rg{#2}{\Pssetc@rve}\else
    \def\n@mref{fl}\ifx\l@debut\n@mref\trtlis@rg{#2}{\Pssetfl@wchart}\else
    \def\n@mref{ge}\ifx\l@debut\n@mref\trtlis@rg{#2}{\Pssetg@n}\else
    \def\n@mref{me}\ifx\l@debut\n@mref\trtlis@rg{#2}{\Pssetm@sh}\else
    \def\n@mref{pr}\ifx\l@debut\n@mref\ifCUR@PS\W@rnmesIgn{figset proj(...)}%
     \else\trtlis@rg{#2}{\Figsetpr@j}\fi\else
    \def\n@mref{tr}\ifx\l@debut\n@mref\trtlis@rg{#2}{\Pssettrim@sh}\else
    \def\n@mref{wr}\ifx\l@debut\n@mref\let\M@cro=\Figsetwr@te\trtlis@rgtok{#2,|}\else
    \W@rnmeskwd{figset}{#1}\fi\fi\fi\fi\fi\fi\fi\fi\fi\fi\ignorespaces}
\ctr@ld@f\def\figsetdefault#1(#2){\ifCUR@PS\W@rnmesIgn{figsetdefault}\else%
    \def\t@xt@{#1}\ifx\t@xt@\empty\trtlis@rg{#2}{\Pssd@g@n}\else\keln@mun#1|
    \def\n@mref{a}\ifx\l@debut\n@mref\trtlis@rg{#2}{\Pssd@@rrowhe@d}\else
    \def\n@mref{c}\ifx\l@debut\n@mref\trtlis@rg{#2}{\Pssd@c@rve}\else
    \def\n@mref{g}\ifx\l@debut\n@mref\trtlis@rg{#2}{\Pssd@g@n}\else
    \def\n@mref{f}\ifx\l@debut\n@mref\trtlis@rg{#2}{\Pssd@fl@wchart}\else
    \def\n@mref{m}\ifx\l@debut\n@mref\trtlis@rg{#2}{\Pssd@m@sh}\else
    \W@rnmeskwd{figsetdefault}{#1}\fi\fi\fi\fi\fi\fi\initpss@ttings\fi}
\ctr@ld@f\def\Pssd@g@n#1=#2|{\keln@mun#1|%
    \def\n@mref{c}\ifx\l@debut\n@mref\edef\D@FTcolor{#2}\else
    \def\n@mref{d}\ifx\l@debut\n@mref\edef\D@FTdash{#2}\else
    \def\n@mref{f}\ifx\l@debut\n@mref\edef\D@FTfill{#2}\else
    \def\n@mref{j}\ifx\l@debut\n@mref\edef\D@FTjoin{#2}\else
    \def\n@mref{u}\ifx\l@debut\n@mref\edef\D@FTupdate{#2}\Q@s@tupdate{#2}\else
    \def\n@mref{w}\ifx\l@debut\n@mref\edef\D@FTwidth{#2}\else
    \W@rnmesAttr{figsetdefault}{#1}\fi\fi\fi\fi\fi\fi}
\ctr@ld@f\def\Pssd@@rrowhe@d#1=#2|{\keln@mun#1|%
    \def\n@mref{a}\ifx\l@debut\n@mref\edef\D@FTarrowheadangle{#2}\else
    \def\n@mref{f}\ifx\l@debut\n@mref\edef\D@FTarrowheadfill{#2}\else
    \def\n@mref{l}\ifx\l@debut\n@mref\y@tiunit{#2}\ifunitpr@sent%
     \edef\D@FTh@rdahlength{#2}\else\edef\D@FTh@rdahlength{#2pt}%
     \message{*** \BS@ figsetdefault (..., #1=#2, ...) : unit is missing, pt is assumed.}%
     \fi\else
    \def\n@mref{o}\ifx\l@debut\n@mref\edef\D@FTarrowheadout{#2}\else
    \def\n@mref{r}\ifx\l@debut\n@mref\edef\D@FTarrowheadratio{#2}\else
    \W@rnmesAttr{figsetdefault arrowhead}{#1}\fi\fi\fi\fi\fi}
\ctr@ld@f\def\Pssd@c@rve#1=#2|{\keln@mun#1|%
    \def\n@mref{r}\ifx\l@debut\n@mref\edef\D@FTroundness{#2}\else%
    \W@rnmesAttr{figsetdefault curve}{#1}\fi}
\ctr@ld@f\def\Pssd@fl@wchart#1=#2|{\keln@mtr#1|%
    \def\n@mref{arr}\ifx\l@debut\n@mref\expandafter\keln@mtr\l@suite|%
     \def\n@mref{owp}\ifx\l@debut\n@mref\edef\D@FTfcarrowposition{#2}\else
     \def\n@mref{owr}\ifx\l@debut\n@mref\edef\D@FTfcarrowrefpt{#2}\else
                     \W@rnmesAttr{figsetdefault flowchart}{#1}\fi\fi\else%
    \def\n@mref{bgc}\ifx\l@debut\n@mref\edef\D@FTfcbgcolor{#2}\else
    \def\n@mref{lin}\ifx\l@debut\n@mref\edef\D@FTfcline{#2}\else
    \def\n@mref{pad}\ifx\l@debut\n@mref\edef\D@FTfcxpadding{#2}%
                    \edef\D@FTfcypadding{#2}\else
    \def\n@mref{rad}\ifx\l@debut\n@mref\edef\D@FTfcradius{#2}\else
    \def\n@mref{sha}\ifx\l@debut\n@mref\edef\D@FTfcshape{#2}\else
    \def\n@mref{thi}\ifx\l@debut\n@mref\expandafter\keln@mtr\l@suite|%
     \def\n@mref{ckn}\ifx\l@debut\n@mref\edef\D@FTfcthickness{#2}\else
                     \W@rnmesAttr{figsetdefault flowchart}{#1}\fi\else%
    \def\n@mref{xpa}\ifx\l@debut\n@mref\edef\D@FTfcxpadding{#2}\else
    \def\n@mref{ypa}\ifx\l@debut\n@mref\edef\D@FTfcypadding{#2}\else
    \W@rnmesAttr{figsetdefault flowchart}{#1}\fi\fi\fi\fi\fi\fi\fi\fi\fi}
\ctr@ld@f\def\D@FTfcarrowposition{0.5}
\ctr@ld@f\def\D@FTfcarrowrefpt{start}
\ctr@ld@f\def\D@FTfcbgcolor{1}
\ctr@ld@f\def\D@FTfcline{polygon}
\ctr@ld@f\def\D@FTfcradius{0}
\ctr@ld@f\def\D@FTfcshape{rectangle}
\ctr@ld@f\def\D@FTfcthickness{0}
\ctr@ld@f\def\D@FTfcxpadding{0}
\ctr@ld@f\def\D@FTfcypadding{0}
\ctr@ld@f\def\Pssd@m@sh#1=#2|{\keln@mun#1|%
    \def\n@mref{d}\ifx\l@debut\n@mref\edef\D@FTmeshdiag{#2}\else%
    \W@rnmesAttr{figsetdefault mesh}{#1}\fi}
\ctr@ln@w{newif}\iffillm@de
\ctr@ld@f\def\Q@s@tfillmode#1{\expandafter\setfillm@de#1:}
\ctr@ld@f\def\setfillm@de#1#2:{\if#1n\fillm@defalse\else\fillm@detrue\fi}
\ctr@ld@f\def\D@FTfill{no}     
\ctr@ln@w{newif}\ifGRupdatem@de
\ctr@ld@f\def\Q@s@tupdate#1{\ifCUR@PS\W@rnmesIgn{figset (update=...)}%
    \else\expandafter\setupd@te#1:\fi}
\ctr@ld@f\def\setupd@te#1#2:{\if#1n\GRupdatem@defalse\else\GRupdatem@detrue\fi}
\ctr@ld@f\def\D@FTupdate{no}     
\ctr@ln@m\CUR@color \ctr@ln@m\CUR@colorc@md
\ctr@ld@f\def\s@uvcolor#1{\edef#1{\CUR@color}}
\ctr@ld@f\def\D@FTcolor{0}       
\ctr@ld@f\def\Pssetc@lor#1{\ifGR@cri\result@tent=\@ne\expandafter\c@lnbV@l#1 :%
    \def\CUR@color{}\def\CUR@colorc@md{}%
    \ifcase\result@tent\or\Q@s@tgray{#1}\or\or\Q@s@trgb{#1}\or\Q@s@tcmyk{#1}\fi\fi}
\ctr@ln@m\CUR@colorc@mdStroke
\ctr@ld@f\def\Q@s@tcmyk#1{\ifGR@cri\def\CUR@color{#1}\def\CUR@colorc@md{\c@msetcmykcolor}%
    \def\CUR@colorc@mdStroke{\c@msetcmykcolorStroke}%
    \ifCUR@PS\PSc@mment{setcmyk Color=#1}\us@primarC@lor\fi\fi}
\ctr@ld@f\def\Q@s@trgb#1{\ifGR@cri\def\CUR@color{#1}\def\CUR@colorc@md{\c@msetrgbcolor}%
    \def\CUR@colorc@mdStroke{\c@msetrgbcolorStroke}%
    \ifCUR@PS\PSc@mment{setrgb Color=#1}\us@primarC@lor\fi\fi}
\ctr@ld@f\def\Q@s@tgray#1{\ifGR@cri\def\CUR@color{#1}\def\CUR@colorc@md{\c@msetgray}%
    \def\CUR@colorc@mdStroke{\c@msetgrayStroke}%
    \ifCUR@PS\PSc@mment{setgray Gray level=#1}\us@primarC@lor\fi\fi}
\ctr@ln@m\fillc@md
\ctr@ld@f\def\us@primarC@lor{\immediate\write\fwf@g{\d@fprimarC@lor}%
    \let\fillc@md=\prfillc@md}
\ctr@ld@f\def\prfillc@md{\d@fprimarC@lor\space\c@mfill}
\ctr@ld@f\def\c@lnbV@l#1 #2:{\def\t@xt@{#1}\relax\ifx\t@xt@\empty\c@lnbV@l#2:
    \else\c@lnbV@l@#1 #2:\fi}
\ctr@ld@f\def\c@lnbV@l@#1 #2:{\def\t@xt@{#2}\ifx\t@xt@\empty%
    \def\t@xt@{#1}\ifx\t@xt@\empty\advance\result@tent\m@ne\fi
    \else\advance\result@tent\@ne\c@lnbV@l@#2:\fi}
\ctr@ld@f\def\Blackcmyk{0 0 0 1}
\ctr@ld@f\def\Whitecmyk{0 0 0 0}
\ctr@ld@f\def\Cyancmyk{1 0 0 0}
\ctr@ld@f\def\Magentacmyk{0 1 0 0}
\ctr@ld@f\def\Yellowcmyk{0 0 1 0}
\ctr@ld@f\def\Redcmyk{0 1 1 0}
\ctr@ld@f\def\Greencmyk{1 0 1 0}
\ctr@ld@f\def\Bluecmyk{1 1 0 0}
\ctr@ld@f\def\Graycmyk{0 0 0 0.50}
\ctr@ld@f\def\BrickRedcmyk{0 0.89 0.94 0.28} 
\ctr@ld@f\def\Browncmyk{0 0.81 1 0.60} 
\ctr@ld@f\def\ForestGreencmyk{0.91 0 0.88 0.12} 
\ctr@ld@f\def\Goldenrodcmyk{ 0 0.10 0.84 0} 
\ctr@ld@f\def\Marooncmyk{0 0.87 0.68 0.32} 
\ctr@ld@f\def\Orangecmyk{0 0.61 0.87 0} 
\ctr@ld@f\def\Purplecmyk{0.45 0.86 0 0} 
\ctr@ld@f\def\RoyalBluecmyk{1. 0.50 0 0} 
\ctr@ld@f\def\Violetcmyk{0.79 0.88 0 0} 
\ctr@ld@f\def\Blackrgb{0 0 0}
\ctr@ld@f\def\Whitergb{1 1 1}
\ctr@ld@f\def\Redrgb{1 0 0}
\ctr@ld@f\def\Greenrgb{0 1 0}
\ctr@ld@f\def\Bluergb{0 0 1}
\ctr@ld@f\def\Cyanrgb{0 1 1}
\ctr@ld@f\def\Magentargb{1 0 1}
\ctr@ld@f\def\Yellowrgb{1 1 0}
\ctr@ld@f\def\Grayrgb{0.5 0.5 0.5}
\ctr@ld@f\def\Chocolatergb{0.824 0.412 0.118}
\ctr@ld@f\def\DarkGoldenrodrgb{0.722 0.525 0.043}
\ctr@ld@f\def\DarkOrangergb{1 0.549 0}
\ctr@ld@f\def\Firebrickrgb{0.698 0.133 0.133}
\ctr@ld@f\def\ForestGreenrgb{0.133 0.545 0.133}
\ctr@ld@f\def\Goldrgb{1 0.843 0}
\ctr@ld@f\def\HotPinkrgb{1 0.412 0.706}
\ctr@ld@f\def\Maroonrgb{0.690 0.188 0.376}
\ctr@ld@f\def\Pinkrgb{1 0.753 0.796}
\ctr@ld@f\def\RoyalBluergb{0.255 0.412 0.882}
\ctr@ld@f\def\Pssetg@n#1=#2|{\keln@mun#1|%
    \def\n@mref{c}\ifx\l@debut\n@mref\update@ttr\D@FTcolor\Pssetc@lor{#2}\else
    \def\n@mref{d}\ifx\l@debut\n@mref\update@ttr\D@FTdash\Q@s@tdash{#2}\else
    \def\n@mref{f}\ifx\l@debut\n@mref\update@ttr\D@FTfill\Q@s@tfillmode{#2}\else
    \def\n@mref{j}\ifx\l@debut\n@mref\update@ttr\D@FTjoin\Q@s@tjoin{#2}\else
    \def\n@mref{u}\ifx\l@debut\n@mref\update@ttr\D@FTupdate\Q@s@tupdate{#2}\else
    \def\n@mref{w}\ifx\l@debut\n@mref\update@ttr\D@FTwidth\Q@s@twidth{#2}\else
    \W@rnmesAttr{figset}{#1}\fi\fi\fi\fi\fi\fi}
\ctr@ln@m\CUR@dash
\ctr@ld@f\def\s@uvdash#1{\edef#1{\CUR@dash}}
\ctr@ld@f\def\D@FTdash{1}        
\ctr@ld@f\def\Q@s@tdash#1{\ifGR@cri\edef\CUR@dash{#1}\ifCUR@PS\expandafter\Pssetd@sh#1 :\fi\fi}
\ctr@ld@f\def\Pssetd@shI#1{\PSc@mment{setdash Index=#1}\ifcase#1%
    \or\immediate\write\fwf@g{[] 0 \c@msetdash}
    \or\immediate\write\fwf@g{[6 2] 0 \c@msetdash}
    \or\immediate\write\fwf@g{[4 2] 0 \c@msetdash}
    \or\immediate\write\fwf@g{[2 2] 0 \c@msetdash}
    \or\immediate\write\fwf@g{[1 2] 0 \c@msetdash}
    \or\immediate\write\fwf@g{[2 4] 0 \c@msetdash}
    \or\immediate\write\fwf@g{[3 5] 0 \c@msetdash}
    \or\immediate\write\fwf@g{[3 3] 0 \c@msetdash}
    \or\immediate\write\fwf@g{[3 5 1 5] 0 \c@msetdash}
    \or\immediate\write\fwf@g{[6 4 2 4] 0 \c@msetdash}
    \fi}
\ctr@ld@f\def\Pssetd@sh#1 #2:{{\def\t@xt@{#1}\ifx\t@xt@\empty\Pssetd@sh#2:
    \else\def\t@xt@{#2}\ifx\t@xt@\empty\Pssetd@shI{#1}\else\s@mme=\@ne\def\debutp@t{#1}%
    \an@lysd@sh#2:\ifodd\s@mme\edef\debutp@t{\debutp@t\space\finp@t}\def\finp@t{0}\fi%
    \PSc@mment{setdash Pattern=#1 #2}%
    \immediate\write\fwf@g{[\debutp@t] \finp@t\space\c@msetdash}\fi\fi}}
\ctr@ld@f\def\an@lysd@sh#1 #2:{\def\t@xt@{#2}\ifx\t@xt@\empty\def\finp@t{#1}\else%
    \edef\debutp@t{\debutp@t\space#1}\advance\s@mme\@ne\an@lysd@sh#2:\fi}
\ctr@ln@m\CUR@width
\ctr@ld@f\def\s@uvwidth#1{\edef#1{\CUR@width}}
\ctr@ld@f\def\D@FTwidth{0.4}     
\ctr@ld@f\def\Q@s@twidth#1{\ifGR@cri\edef\CUR@width{#1}\ifCUR@PS%
    \PSc@mment{setwidth Width=#1}\immediate\write\fwf@g{#1 \c@msetlinewidth}\fi\fi}
\ctr@ln@m\CUR@join
\ctr@ld@f\def\s@uvjoin#1{\edef#1{\CUR@join}}
\ctr@ld@f\def\D@FTjoin{miter}   
\ctr@ld@f\def\Q@s@tjoin#1{\ifGR@cri\edef\CUR@join{#1}\ifCUR@PS\expandafter\Pssetj@in#1:\fi\fi}
\ctr@ld@f\def\Pssetj@in#1#2:{\PSc@mment{setjoin join=#1}%
    \if#1r\def\t@xt@{1}\else\if#1b\def\t@xt@{2}\else\def\t@xt@{0}\fi\fi%
    \immediate\write\fwf@g{\t@xt@\space\c@msetlinejoin}}
\ctr@ld@f\def\Pss@tspecifSt#1{\trtlis@rg{#1}{\Pss@tspecifSt@}}
\ctr@ld@f\def\Pss@tspecifSt@#1=#2|{\keln@mun#1|%
    \def\n@mref{c}\ifx\l@debut\n@mref\def\n@mref{#2}\ifx\n@mref\D@FTref\else%
     \s@uvcolor{\typ@color}\Pssetc@lor{#2}\fi\else
    \def\n@mref{d}\ifx\l@debut\n@mref\def\n@mref{#2}\ifx\n@mref\D@FTref\else%
     \s@uvdash{\typ@dash}\Q@s@tdash{#2}\fi\else
    \def\n@mref{j}\ifx\l@debut\n@mref\def\n@mref{#2}\ifx\n@mref\D@FTref\else%
     \s@uvjoin{\typ@join}\Q@s@tjoin{#2}\fi\else
    \def\n@mref{w}\ifx\l@debut\n@mref\def\n@mref{#2}\ifx\n@mref\D@FTref\else%
     \s@uvwidth{\typ@width}\Q@s@twidth{#2}\fi\else
    \W@rnmeskwd{Pss@tspecifSt}{#1}\fi\fi\fi\fi}
\ctr@ld@f\def\Psrest@reSt#1{\trtlis@rg{#1}{\Psrest@reSt@}}
\ctr@ld@f\def\Psrest@reSt@#1=#2|{\keln@mun#1|%
    \def\n@mref{c}\ifx\l@debut\n@mref\def\n@mref{#2}\ifx\n@mref\D@FTref\else%
     \Pssetc@lor{\typ@color}\fi\else
    \def\n@mref{d}\ifx\l@debut\n@mref\def\n@mref{#2}\ifx\n@mref\D@FTref\else%
     \Q@s@tdash{\typ@dash}\fi\else
    \def\n@mref{j}\ifx\l@debut\n@mref\def\n@mref{#2}\ifx\n@mref\D@FTref\else%
     \Q@s@tjoin{\typ@join}\fi\else
    \def\n@mref{w}\ifx\l@debut\n@mref\def\n@mref{#2}\ifx\n@mref\D@FTref\else%
     \Q@s@twidth{\typ@width}\fi\else
    \W@rnmeskwd{Psrest@reSt}{#1}\fi\fi\fi\fi}
\ctr@ld@f\def\Pssettrim@sh#1=#2|{\keln@mde#1|%
    \def\n@mref{co}\ifx\l@debut\n@mref\update@ttr\D@FTref\P@settmeshcolor{#2}\else
    \def\n@mref{da}\ifx\l@debut\n@mref\update@ttr\D@FTref\P@settmeshdash{#2}\else
    \def\n@mref{wi}\ifx\l@debut\n@mref\update@ttr\D@FTref\P@settmeshwidth{#2}\else
    \W@rnmesAttr{figset trimesh}{#1}\fi\fi\fi}
\ctr@ln@m\DDV@tmeshcolor
\ctr@ld@f\def\P@settmeshcolor#1{\edef\DDV@tmeshcolor{#1}}
\ctr@ln@m\DDV@tmeshdash
\ctr@ld@f\def\P@settmeshdash#1{\edef\DDV@tmeshdash{#1}}
\ctr@ln@m\DDV@tmeshwidth
\ctr@ld@f\def\P@settmeshwidth#1{\edef\DDV@tmeshwidth{#1}}
\ctr@ld@f\def\figdrawtrimesh#1[#2,#3,#4]{{\ifCUR@PS\ifGR@cri%
    \PSc@mment{trimesh Type=#1, Triangle=[#2,#3,#4]}%
    \s@uvc@ntr@l\et@tpstrimesh\ifnum#1>\@ne%
    \Pss@tspecifSt{color=\DDV@tmeshcolor,dash=\DDV@tmeshdash,width=\DDV@tmeshwidth}%
    \setc@ntr@l{2}%
    \Pstrimeshp@rt#1[#2,#3,#4]\Pstrimeshp@rt#1[#3,#4,#2]\Pstrimeshp@rt#1[#4,#2,#3]%
    \Psrest@reSt{color=\DDV@tmeshcolor,dash=\DDV@tmeshdash,width=\DDV@tmeshwidth}%
    \fi\figdrawline[#2,#3,#4,#2]%
    \PSc@mment{End trimesh}\resetc@ntr@l\et@tpstrimesh\fi\fi}}
\ctr@ld@f\def\Pstrimeshp@rt#1[#2,#3,#4]{{\l@mbd@un=\@ne\l@mbd@de=#1\loop\ifnum\l@mbd@de>\@ne%
    \advance\l@mbd@de\m@ne\figptbary-1:[#2,#3;\l@mbd@de,\l@mbd@un]%
    \figptbary-2:[#2,#4;\l@mbd@de,\l@mbd@un]\figdrawline[-1,-2]%
    \advance\l@mbd@un\@ne\repeat}}
\initpr@lim\initpss@ttings\initPDF@rDVI
\ctr@ln@w{newbox}\figBoxA
\ctr@ln@w{newbox}\figBoxB
\ctr@ln@w{newbox}\figBoxC
\catcode`\@=12

\makeatletter
\renewcommand\theequation{\thesection.\arabic{equation}}
\@addtoreset{equation}{section}

\def\@listI{
    \leftmargin\leftmargini
    \parsep 1.5pt plus 1pt minus 1pt
    \topsep -1.5pt plus 1pt minus 1pt
    \itemsep \parsep}
\let\@listi\@listI
\makeatother

\bibliographystyle{abbrv}
\usepackage{color}
\newcommand{\Bk}{\color{black}}
\newcommand{\Rd}{\color{red}}
\newcommand{\Gn}{\color{green}}
\newcommand{\Bl}{\color{blue}}

\newcommand{\N}{\mathbb{N}}
\newcommand{\R}{\mathbb{R}}
\newcommand{\Z}{\mathbb{Z}}

\newcommand{\cB}{\mathcal{B}}
\newcommand{\cH}{\mathcal{H}}
\newcommand{\cL}{\mathcal{L}}
\newcommand{\cO}{\mathcal{O}}
\newcommand{\cQ}{\mathcal{Q}}

\newcommand{\fq}{\mathfrak{q}}

\newcommand{\sC}{\mathscr{C}}

\newcommand{\dd}{\mathsf{d}}
\newcommand{\G}{\mathsf{G}}
\newcommand{\Gui}{\mathsf{Gui}}
\newcommand{\Lay}{\mathsf{Lay}}
\newcommand{\Dom}{\mathsf{Dom}}
\newcommand{\Tri}{\mathsf{Tri}}

\newcommand{\far}{\mathsf{far}}
\newcommand{\near}{\mathsf{near}}
\newcommand{\trans}{\mathsf{trans}}

\newcommand{\ps}[2]{\left\langle#1,#2\right\rangle}
\newcommand{\vabs}[1]{\left|#1\right|}

\newtheorem{theorem}{Theorem}[section]
\newtheorem{prop}[theorem]{Proposition}

\theoremstyle{definition}
\newtheorem{defin}[theorem]{Definition}
\newtheorem{nota}[theorem]{Notation}

\theoremstyle{remark}

\title[Spectral asymptotics in a conical layer]{Spectral asymptotics of the Dirichlet Laplacian\\ in a conical layer}
\author{Monique Dauge} 
\address{IRMAR, Univ. Rennes 1, CNRS, Campus de Beaulieu, F-35042 Rennes cedex, France}
\email{monique.dauge@univ-rennes1.fr}
\author{Thomas Ourmi\`eres-Bonafos}
\address{BCAM - Basque Center for Applied Mathematics, Alameda de Mazarredo, 14 E48009 Bilbao, Basque Country -  Spain}
\email{tourmieres@bcamath.org}
\author{Nicolas Raymond}
\address{IRMAR, Univ. Rennes 1, CNRS, Campus de Beaulieu, F-35042 Rennes cedex, France}
\email{nicolas.raymond@univ-rennes1.fr}
\date{}

\subjclass[2010]{35J05, 35P15, 47A75}

\begin{document}
\maketitle
\let\thefootnote\relax\footnote{\textit{2010 Mathematics Subject Classification.} 35J05, 35P15, 47A75}

\begin{abstract}
The spectrum of the Dirichlet Laplacian on conical layers is analysed through two aspects: the infiniteness of the discrete eigenvalues and their expansions in the small aperture limit.

On the one hand, we prove that, for any aperture, the eigenvalues accumulate below the threshold of the essential spectrum: For a small distance from the essential spectrum, the number of eigenvalues farther from the threshold than this distance behaves like the logarithm of the distance.

On the other hand, in the small aperture regime, we provide a two-term asymptotics of the first eigenvalues thanks to \textit{a priori} localization estimates for the associated eigenfunctions. We prove that these eigenfunctions are localized in the conical cap at a scale of order the cubic root of the aperture angle and that they get into the other part of the layer at a scale involving the logarithm of the aperture angle.
\end{abstract}

\section{Introduction}

\subsection{Motivations}
In mesoscopic physics semiconductors can be modelled by Schr\"odinger operators on tubes (also called waveguides) or layers carrying the Dirichlet condition on their boundaries.
The existence of bound states for the Dirichlet Laplacian in such structures is an important issue since the discrete spectrum can be seen as a defect in the wave propagation.

First, the question of the discrete spectrum was studied for waveguides \textit{i.e.} infinite tubes (in dimension two or three) which are asymptotically straight. For these tubes, one can prove that the essential spectrum of the Dirichlet Laplacian has the form $[a,+\infty)$ (with $a\in\R_+$). For smooth waveguides, the effect of bending is studied by Exner and \v{S}eba in \cite{ES89} and Goldstone and Jaffe in \cite{GJ92}. In \cite{DE95}, Duclos and Exner prove that bending induces bound states. The same question was addressed in dimension two when there is a corner (which can be seen as an infinite curvature). In \cite{ESS89}, Exner, \v{S}eba and {\v{S}}{\v{t}}ov{\'{\i}}{\v{c}}ek prove that for the $L$-shaped waveguide there is a unique bound state below the threshold of the essential spectrum. Two of the authors, with Lafranche, prove in \cite{DLR12} that for a two dimensional waveguide with corner of arbitrary angle, the bound states are in finite number below the threshold of the essential spectrum. In fact, in \cite{NS14}, Nazarov and Shanin prove that for a large enough angle there is a unique bound state.

Second, the question was addressed for layers \textit{i.e.} infinite regions in $\R^3$ limited by two identical surfaces. For smooth enough layers, Duclos, Exner and Krej{\v{c}}i{\v{r}}{\'{\i}}k in \cite{DEK01} and then Carron, Exner and Krej{\v{c}}i{\v{r}}{\'{\i}}k in \cite{CEK04} prove that bending could induce bound states. Now, as for two dimensional waveguides, one can ask what happens when there is an infinite curvature. In \cite{ET11}, Exner and Tater deal with this question in the particular case of a conical layer, i.e., a layer limited by two infinite coaxial conical surfaces with same openings. 
Again, the essential spectrum of the Dirichlet Laplacian has the form $[a,+\infty)$ (with $a\in\R_+$) and one of the main results of their paper is the infiniteness of the discrete spectrum below the threshold $a$.


The first question we tackle in this paper is to quantify this infinity: As the two dimensional waveguide with corner in \cite{DLR12} is the meridian domain of the conical layer in \cite{ET11}, we wanted to understand how one can pass from a finite number of bound states to an infinite number adding one dimension. In the same spirit it is worth mentioning the paper \cite{BEL14} of Behrndt, Exner and Lotoreichik where the spectrum of the Dirichlet Laplacian with a $\delta$-interaction on a cone is investigated. Roughly, one can claim that this operator is another modelling of the physical phenomenon we are interested in. In this problem, the essential spectrum has the form $[-a,+\infty)$ (with $a\in\R_+$) and there is an infinite number of bound states. For $n\geq1$, they also provide an upper bound for the $n$-th eigenvalue of the problem.

The second question dealt with in the present paper is the study of the eigenvalues and associated eigenfunctions of conical layers in the small aperture regime. This question is reminiscent of the article \cite{DR12} where the small angle regime for waveguides with corner is studied. In the latter paper asymptotic expansions at any order for the eigenvalues and eigenfunctions are provided and we would like to determine if similar expansions can be obtained for the conical layer. In fact, in the spirit of the Born-Oppenheimer approximation, one can at least formally reduce the two dimensional waveguide and the conical layer to electric Schr\"odinger operators in one dimension. The minima of the effective electric potentials determine the behavior of the eigenfunctions. This is a well known fact, for a smooth minimum, that it leads to the study of the so called harmonic approximation (see \cite{CFKS87, DJ99, S83}). However, in \cite{DR12} and in the present paper, the minima of the effective potentials are not smooth. For the conical layer it involves a logarithmic singularity. That is why we will only be able to provide a finite term in the expansions of the eigenvalues. To do so, we will prove \textit{a priori} localization estimates for the eigenfunctions, the so called Agmon estimates (see Agmon \cite{A82,A85} and, in the semiclassical context, Helffer \cite{H88} and Helffer and Sj{\"o}strand \cite{HJ84,HJ85}). They will imply that the eigenfunctions of the conical layers are localized in the conical cap and, unlike the eigenfunctions of the two dimensional waveguides with corner, get into the rest of the layer at a scale involving the logarithm of the aperture.

\subsection{The Dirichlet Laplacian on conical layers}
\label{subsec:dirlapl}
Let us denote by $(x_{1},x_{2},x_{3})$ the Cartesian coordinates of the space $\R^3$ and by $\mathbf{0} = (0,0,0)$ the origin. The positive Laplace operator is given by $-\Delta=-\partial_{1}^2 - \partial_{2}^2 -\partial_{3}^2$. We denote by $(r,\phi,z)\in\R_+\times[0,2\pi)\times\R$ the cylindrical coordinates with axis $x_3$, i.e., such that:
\begin{equation}
   x_1 = r\cos\phi\,,\quad
   x_2=r\sin\phi,\quad\mbox{and}\quad 
   z = x_{3}\in\R.
\label{eqn:cylcoor}
\end{equation}
A conical layer being an axisymmetric domain, it is easier to define it through its {\em meridian domain}. We denote our conical layer by $\Lay(\theta)$ from its half-opening angle $\theta\in(0,\frac\pi2)$. The meridian domain of $\Lay(\theta)$ is denoted by $\Gui(\theta)$ and defined as, see also Figure \ref{F:1},
\begin{equation}
   \Gui(\theta) = \Big\{(r,z)\in\R_+\times\R : \quad
   -\frac\pi{\sin\theta}< z,\quad 
   \max(0,z\tan\theta) < r < z\tan\theta + \frac{\pi}{\cos\theta}\Big\}.
\label{eqn:defmertheta}
\end{equation}
The domains are normalized so that the distance between the two connected components of the boundary of $\Lay(\theta)$ is $\pi$ for any value of $\theta$.

The aim of this paper is to investigate some spectral properties of the Dirichlet Laplacian $-\Delta_{\Lay(\theta)}$ whose domain is denoted by $\Dom(-\Delta_{\Lay(\theta)})$. Note that, according to the value of $\theta$, this domain is a subspace of $H^2(\Lay(\theta))$ (if $\theta>\theta_0$) or not (if $\theta\le\theta_0$), where $\theta_0$ is such that $P_{1/2}(\cos(\pi-\theta_0))=0$, see  \cite[\S\,II.4.c]{BDM99} -- here $P_\nu$ is the Legendre function with degree $\nu$, and $\theta_0\simeq0.2738\pi$.

\begin{figure}[h!]
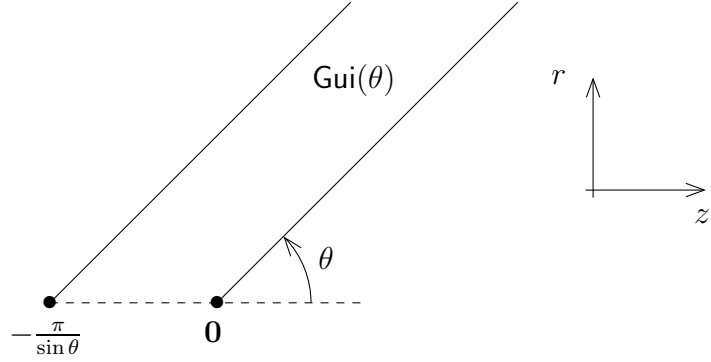

\figinit{0.5cm}
\figpt 0:(0,0)
\figpt 1:(-4.44288294,0)

\figpt 2:(-5,0)
\figpt 3:(4,0) 

\figpt 4:(8,8)
\figpt 5:(3.55711706,8)

\figpt 6:(2,0)
\figpt 7:(2,2)

\figpt 8:(2.2,1.2)
\figpt 9:(3.75,5)

\figpt 10:(10,2.8)
\figpt 11:(10,6)

\figpt 12:(9.8,3)
\figpt 13:(13,3)
\figdrawbegin{}
\figdrawline [0,4]
\figdrawline [1,5]
\figdrawarrow[10,11]
\figdrawarrow[12,13]
\figdrawarrowcircP 0;2.5[6,7]
\figset (dash=8)
\figdrawline [1,3]
\figdrawend

\figvisu{\figBoxA}{}{
\figwrites 13:{$z$} (0.5)
\figwritew 11:{$r$} (0.5)
\figwritee 8:{$\theta$} (0.5)
\figwriten 9:{$\Gui(\theta)$} (0.5)
\figset write(mark=$\bullet$)
\figwrites 1:{$-\frac{\pi}{\sin\theta}$} (0.5)
\figwrites 0:{$\mathbf{0}$} (0.5)
}

\centerline{\box\figBoxA}
\caption{The meridian guide $\Gui(\theta)$.}
\label{F:1}
\end{figure}

Through the change of variables \eqref{eqn:cylcoor} the Cartesian domain $\Lay(\theta)$ becomes $\Gui(\theta)\times\mathbb{S}^1$ and the Dirichlet Laplacian becomes the unbounded selfadjoint operator on $L^2(\Gui(\theta)\times\mathbb{S}^1,r\dd r\dd\phi\dd z)$, denoted by $\cH_{\Gui(\theta)\times\mathbb{S}^1}$:
\begin{equation*}
\cH_{\Gui(\theta)\times\mathbb{S}^1} = 
-\frac1r(r\partial_{r})^2 - \frac1{r^2}\partial_{\phi}^2 - \partial_{z}^2.
\end{equation*}
By Fourier series, according to the terminology of \cite[\S\,XIII.16]{RS78}, we have the constant fiber sum:
\begin{equation*}
L^2(\Gui(\theta)\times\mathbb{S}^1,r \dd r \dd \phi \dd z) = L^2((\Gui(\theta),r\dd r \dd z) \overline{\otimes} L^2(\mathbb{S}^1) = \bigoplus_{m\in\Z}L^2(\Gui(\theta),r \dd r \dd z),
\end{equation*}
where $L^2(\mathbb{S}^1)$ refers to functions on the unit circle with orthonormal basis $\{e^{2i\pi m\phi} : m \in\Z\}$. The operator $\cH_{\Gui(\theta)\times\mathbb{S}^1}$ decomposes as:
\begin{equation}
   \cH_{\Gui(\theta)\times\mathbb{S}^1} = \bigoplus_{m\in\Z}\cH_{\Gui(\theta)}^{[m]},
   \quad \text{with}\quad
   \cH_{\Gui(\theta)}^{[m]} = -\frac1r\partial_{r}(r\partial_{r}) - \partial_{z}^2 + \frac{m^2}{r^2},
\label{eqn:decope}
\end{equation}
where the operators $\cH_{\Gui(\theta)}^{[m]}$ are the fibers of $\cH_{\Gui(\theta)\times\mathbb{S}^1}$. The associated quadratic forms are
\[
   \cQ_{\Gui(\theta)}^{[m]}(\psi) = \int_{\Gui(\theta)} |\partial_r\psi(r,z)|^2 + 
   |\partial_z\psi(r,z)|^2 + \frac{m^2}{r^2}|\psi(r,z)|^2\,r \dd r \dd z\,.
\]
The form domains, denoted by $\Dom(\cQ_{\Gui(\theta)}^{[m]})$ depend on $m$, namely, cf.\ \cite[\S\,II.3.a]{BDM99},
\[
\Dom(\cQ_{\Gui(\theta)}^{[m]}) =
   \begin{cases}
   \{\psi: \ \psi,\ \partial_r\psi,\ \partial_z\psi\in L^2(\Gui(\theta),r \dd r \dd z),\ 
   \psi\big|_{\partial_1\Gui(\theta)}=0\}, 
   & m=0,\\
   \{\psi: \ \psi,\ \partial_r\psi,\ \partial_z\psi,\ r^{-1}\psi\in L^2(\Gui(\theta),r \dd r \dd z),\ 
   \psi\big|_{\partial_1\Gui(\theta)} =0\},
   & m\neq0,
\end{cases}
\]
with $\partial_1\Gui(\theta)$ the part of $\partial\Gui(\theta)$ that does not meet the axis $r=0$. The domain of the operator $\cH_{\Gui(\theta)}^{[m]}$ is deduced from the form domain in the standard way.

\subsection{Main results}

Let us denote by $\mathfrak{S}_{\mathsf{disc}}(\cL)$ and $\mathfrak{S}_{\mathsf{ess}}(\cL)$ the discrete and essential spectrum of the selfadjoint operator $\cL$, respectively. For our study we start from the following results on the essential and discrete spectrum of $\cH^{[m]}_{\Gui(\theta)}$ that we deduce from \cite[\S\,3]{ET11}.

\begin{theorem}[Exner and Tater \cite{ET11}]
\label{prop:exdiscspec}
Let $\theta\in(0,\frac\pi2)$. There holds, 
\[
   \mathfrak{S}_{\mathsf{ess}}(\cH^{[m]}_{\Gui(\theta)}) = [1,\infty), \quad \forall m\in\Z
\]
and
\[
   \#\mathfrak{S}_{\mathsf{disc}}(\cH^{[0]}_{\Gui(\theta)}) = \infty
   \quad\mbox{and}\quad
   \mathfrak{S}_{\mathsf{disc}}(\cH_{\Gui(\theta)}^{[m]}) = \emptyset\ \mbox{ if }\ m\neq0.
\]
\end{theorem}
Note that the minimum at $1$ of the essential spectrum is a consequence of the normalization of the meridian guides $\Gui(\theta)$ and that the absence of discrete spectrum when $m\neq0$ is due to the cancellation of functions in the domain of the operator on the axis $r=0$.

Relying on this theorem, we see that the investigation of the discrete spectrum of $-\Delta_{\Lay(\theta)}$ reduces to its axisymmetric part $\cH^{[0]}_{\Gui(\theta)}$.
To simplify the notation, we drop the exponent $0$, thus denoting this operator by $\cH_{\Gui(\theta)}$. We denote its eigenvalues by $(\mu_n(\theta))_{n\geq1}$, ordering them in non-decreasing order and repeating them according to their multiplicity.

We first state a monotonicity result about the dependence on $\theta$ of $\mu_n(\theta)$.
\begin{prop}
For all $n\in\mathbb{N}^*$, the functions $\theta \mapsto \mu_{n}(\theta)$ are non-decreasing on $(0,\frac\pi2)$ into $\R_+$.
\label{prop:nondecvp}
\end{prop}

Before stating our result on the accumulation of these eigenvalues towards the minimum of the essential spectrum, we need to introduce:
\begin{defin} 
\label{defin:nbvpsousseuil}
Let $\nu\in\R$ and $\cL$ be a self-adjoint operator semi-bounded from below and associated with a quadratic form $\cQ$. We define the counting function $\mathcal{N}_\nu(\cL)$ by:
\begin{equation*}
\mathcal{N}_{\nu}(\cL) = \# \{\mu \in \mathfrak{S}_{\mathsf{disc}}(\cL) : \mu < \nu\}.
\end{equation*}
When working with the quadratic form $\cQ$ we use the notation $\mathcal{N}_\nu(\cQ)$ instead of $\mathcal{N}_{\nu}(\cL)$.
\end{defin}

In this paper we prove:
\begin{theorem}
Let us choose $\theta\in(0,\frac\pi2)$. We have:
\begin{equation*}
\mathcal{N}_{1-E}(\cH_{\Gui(\theta)}) \underset{\underset{E>0}{E \rightarrow 0}}{\sim} \frac{\cot\theta}{4\pi}|\ln E|.
\end{equation*}
\label{th:countvp}
\end{theorem}
Theorem \ref{th:countvp} exhibits a logarithmic accumulation of the number of eigenvalues near the threshold of the essential spectrum.

In order to state our result on the spectral asymptotics of the first eigenvalues as $\theta\to0$ we need to introduce some notations on zeros of Bessel and Airy functions \cite{AS66}.

\begin{nota} We denote by $J_0$ and $Y_0$ the $0$-th Bessel functions of first and second kind, respectively. For all $n\geq1$, the $n$-th zero of $J_0$ is denoted by $j_{0,n}$. For all $n\geq1$, we also denote by $z_{\mathsf{A}}(n)$ the $n$-th zero of the reversed Airy function of first kind. The sequence $(z_{\mathsf{A}}(n))_{n\geq1}$ is increasing. 
\label{nota:specfunc}
\end{nota}

The following theorem describes the behavior of the first eigenvalues of $\cH_{\Gui(\theta)}$ in the limit $\theta\rightarrow0$. This result confirms that $\mu_{n}(\theta)$ goes to $\frac{j_{0,1}^2}{\pi^2}$ as $\theta$ goes to $0$ which is conjectured in \cite[Figure 2]{ET11}.

\begin{theorem} 
\label{th:expavp}
For all $n\in\mathbb{N}^*$, we have the asymptotic expansion:
\begin{equation*}
\mu_{n}(\theta) \underset{\theta\rightarrow0}{=} 
\frac{j_{0,1}^2}{\pi^2} + 
\frac{(2j_{0,1})^{2/3}\!\!}{\pi^2}\ z_{\mathsf{A}}(n) \theta^{2/3} + 
\cO(\theta |\ln \theta|^{3/2}),\quad n = 1,\dots,N_{0}.
\end{equation*}
\end{theorem}

In Section \ref{section:moneig} we prove Proposition \ref{prop:nondecvp} by reformulating the problem in another domain.
Section \ref{sec:counteig} deals with Theorem \ref{th:countvp}. We prove that counting the eigenvalues on the meridian waveguide $\Gui(\theta)$ reduces to consider operators on a half-strip, leading to count the eigenvalues of one-dimensional operators.
Section \ref{sec:conicalsemiclass} concerns the small aperture regime $\theta\rightarrow0$. First, we show that the problem admits a semiclassical formulation. Then, we use Agmon localization estimates to obtain Theorem \ref{th:expavp}.
Finally in Appendix \ref{sec:numresult} we illustrate the main results of this paper by numerical computations by finite element method.

\section{Monotonicity of the eigenvalues}
\label{section:moneig}
We denote by $\cQ_{\Gui(\theta)}$ the quadratic form associated with $\cH_{\Gui(\theta)}$ and by $\Dom(\cQ_{\Gui(\theta)})$ its domain. For the proof of Proposition \ref{prop:nondecvp} and further use, we consider the change of variables (rotation):
\begin{equation}
s = z \cos\theta + rÊ\sin \theta,\quad u = -z\sin\theta + r\cos\theta,
\label{eqn:rotlayer}
\end{equation}
that transforms the meridian guide $\Gui(\theta)$ into the strip with corner $\Omega(\theta)$ (see Figure \ref{fig:figureomega}), defined by:
\begin{equation*}
\Omega(\theta) = \big\{(s,u)\in\R^2 : \ s \geq -\pi\cot\theta,\quad 
\max(0,-s \tan\theta) < u < \pi\big\}.
\end{equation*}

\begin{figure}[b]
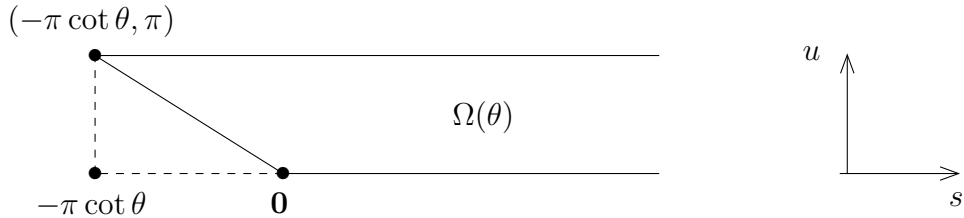

\figinit{0.5cm}
\figpt 0:(0,0)
\figpt 1:(-5,0)
\figpt 2:(10,0)

\figpt 3:(-5,3.14)
\figpt 4:(10,3.14)

\figpt 5:(4,1.5)

\figpt 10:(15,0)
\figpt 11:(15,3.2)

\figpt 12:(14.8,0)
\figpt 13:(18,0)
\figdrawbegin{}
\figdrawline [0,2]
\figdrawline [3,4]
\figdrawline [0,3]
\figdrawarrow[10,11]
\figdrawarrow[12,13]
\figset (dash=8)
\figdrawline [0,1]
\figdrawline [1,3]
\figdrawend

\figvisu{\figBoxA}{}{
\figwrites 13:{$s$} (0.5)
\figwritew 11:{$u$} (0.5)
\figwritee 5:{$\Omega(\theta)$}(0.5)
\figset write(mark=$\bullet$)
\figwrites 1:{$-\pi\cot\theta$} (0.5)
\figwrites 0:{$\mathbf{0}$} (0.5)
\figwriten 3:{$(-\pi\cot\theta,\pi)$} (0.5)
}

\centerline{\box\figBoxA}
\caption{The domain $\Omega(\theta)$.}
\label{fig:figureomega}
\end{figure}

For $\psi\in\Dom(\cQ_{\Gui(\theta)})$, we set $\tilde\psi(s,u) = \psi(r,z)$ and we have the identity $\cQ_{\Gui(\theta)}(\psi)=\cQ_{\Omega(\theta)}(\tilde{\psi})$ with the new quadratic form
\begin{equation}
\cQ_{\Omega(\theta)}(\tilde{\psi}) = \int_{\Omega(\theta)} (|\partial_{s}\tilde\psi|^2 + |\partial_{u}\tilde{\psi}|^2)\,(s\sin\theta + u\cos\theta)\dd u\dd s.
\label{eqn:defomega}
\end{equation}
The proof of the monotonicity of the eigenvalues $\mu_n(\theta)$ now follows the same track that the analogous proof in \cite[\S\,3]{DLR12} for the case of broken waveguides.
To avoid the dependence on $\theta$ of the domain $\cQ_{\Omega(\theta)}$ we perform the change of variables $(s,u)\mapsto(\hat s,\hat u) = ( s\tan\theta,u)$ that transforms 
the domain $\Omega(\theta)$ into $\Omega := \Omega(\frac\pi4)$. Setting 
$\hat{\psi}(\hat{s},\hat{u}) = \tilde{\psi}(s,u)$ we get for the Rayleigh quotients:
\begin{align*}
   \frac{\cQ_{\Gui(\theta)}(\psi)}{\|\psi\|^2} &= 
   \frac{\int_{\Omega}(\tan^2\theta|\partial_{\hat{s}}\hat\psi|^2 + |\partial_{\hat{u}}\hat\psi|^2)(\hat{s} + \hat{u})\cos\theta\cot\theta\dd \hat{s}\dd\hat{u}}{\int_{\Omega}|\hat\psi|^2(\hat{s} + \hat{u})\cos\theta\cot\theta\dd\hat{s}\dd\hat{u}} \\[0.5ex]
   &=  \frac{\int_{\Omega}(\tan^2\theta|\partial_{\hat{s}}\hat\psi|^2 + |\partial_{\hat{u}}\hat\psi|^2)(\hat{s} + \hat{u})\dd \hat{s}\dd\hat{u}}{\int_{\Omega}|\hat\psi|^2(\hat{s} + \hat{u})\dd\hat{s}\dd\hat{u}},
\end{align*}
that appear to be nondecreasing functions of $\theta$. The result follows from the min-max formulas for the eigenvalues.

\section{Counting the eigenvalues}
\label{sec:counteig}
In this section we prove Theorem \ref{th:countvp}. The idea is to reduce to a one-dimensional operator.

We will need a result adapted from Kirsch and Simon \cite{KS88}, later extended by Hassel and Marshall in \cite{HM08}. Let $c>0$, we are interested in the Friedrichs extension of the following quadratic form,
\begin{equation*}
\mathfrak{q}(\varphi)=\int_1^{+\infty}|\partial_x\varphi|^2 - \frac{c}{x^2}|\varphi|^2\;\dd x,\quad
\varphi\in\sC^\infty_0(1,+\infty).
\end{equation*}
and we denote by $\mathfrak{h}$ the corresponding operator.
\begin{theorem}[Kirsch and Simon \cite{KS88}] 
\label{th:KS88}
Let $V_0\in\sC^\infty_0(1,+\infty)$. If $c>\frac14$, there holds:
\begin{equation*}
   \mathcal{N}_{-E}(\mathfrak{h} +V_0) \underset{\underset{E>0}{E\rightarrow0}}{\sim} 
   \frac1{2\pi}\, \sqrt{c - \frac14}\ |\ln E|.
\end{equation*}
\end{theorem}

In \S\,\ref{subsec:lowbound} we find a lower bound for $\mathcal{N}_{1-E}(\cQ_{\Gui(\theta)})$. An upper bound is obtained in \S\,\ref{subsec:uppbound} and these bounds, together with Theorem \ref{th:KS88}, yield Theorem \ref{th:countvp} in \S\,\ref{subsec:proofthcount}.

For the next two subsections, instead of working with the quadratic form $\cQ_{\Gui(\theta)}$, it will be more convenient to use the quadratic form $\cQ_{\Omega(\theta)}$ introduced in \eqref{eqn:defomega}.

\subsection{A lower bound on the number of bound states}
\label{subsec:lowbound}
Let us denote by $\Sigma$ the half-strip $(1,+\infty)\times(0,\pi)$. We consider the quadratic form $\cQ_{\Sigma}(\theta)$, Friedrichs extension of the form defined for all functions $\psi\in\sC_0^\infty(\Sigma)$ by:
\begin{equation}
   \cQ_{\Sigma}(\theta)(\psi) = \int_{\Sigma}(|\partial_s\psi|^2 + |\partial_u\psi|^2)\,
   (s\sin\theta + u \cos\theta)\;\dd s\dd u,
\label{eqn:fqlower}
\end{equation}
where the $(s, u)$ variables are related to the physical domain through the change of variables \eqref{eqn:rotlayer}.
We also define the one-dimensional quadratic form $\hat{\mathfrak{q}}(\theta)$, Friedrichs extension of the form defined, for all $\hat{\varphi}\in\sC_0^\infty(1,+\infty)$, by:
\begin{equation*}
   \hat{\mathfrak{q}}(\theta)(\hat\varphi) = 
   \int_1^{+\infty}|\partial_\sigma\hat\varphi|^2 -\frac{1}{4\sigma^2\sin^2\theta}|\hat\varphi|^2
   \;\dd \sigma.
\end{equation*}

\begin{prop}
\label{prop:lowboundnbvp}
Let us fix $\theta\in(0,\frac{\pi}2)$. Let $E>0$, and set $\hat E =(1+\pi\cot\theta)^2E$. We have:
\begin{equation*}
   \mathcal{N}_{-\hat E}(\hat{\mathfrak{q}}(\theta)) \leq 
   \mathcal{N}_{1-E}(\cQ_{\Omega(\theta)}).
\end{equation*}
\end{prop}

\begin{proof}
Any $\psi\in\Dom(\cQ_{\Sigma}(\theta))$ can be extended by zero, defining $\psi_0\in\Dom(\cQ_{\Omega(\theta)})$ such that $\cQ_{\Sigma}(\theta)(\psi) = \cQ_{\Omega(\theta)}(\psi_0)$.
The min-max principle then yields
\begin{equation}
\mathcal{N}_{1-E}(\cQ_{\Sigma}(\theta)) \leq \mathcal{N}_{1-E}(\cQ_{\Omega(\theta)}).
\label{eqn:lowbound1}
\end{equation}
For $\psi\in\Dom(\cQ_{\Sigma}(\theta))$, let $\widehat\psi(s,u) = \sqrt{s\sin\theta + u\cos\theta}\,\,\psi(s,u)$. We have:
\begin{equation}\label{eqn:psihat}
\cQ_{\Sigma}(\theta)(\psi) = \displaystyle\int_{\Sigma}|\partial_{s}\widehat\psi|^2 + |\partial_{u}\widehat\psi|^2 - \frac{1}{4(s\sin\theta + u\cos\theta)^2}|\widehat\psi|^2\; \dd s \dd u
\end{equation}
and we bound $(s\sin\theta + u\cos\theta)^2$ from above by $(s\sin\theta + \pi\cos\theta)^2$, obtaining
\begin{equation}
\cQ_{\Sigma}(\theta)(\psi) 
\leq\displaystyle\int_{\Sigma}|\partial_{s}\widehat\psi|^2 + |\partial_{u}\widehat\psi|^2 - \frac{1}{4(s\sin\theta + \pi\cos\theta)^2}|\widehat\psi|^2 \;\dd s \dd u.
\label{eqn:red1Dlow}
\end{equation}
Now, we denote by ${\mathfrak{q}}(\theta)$ the quadratic form, Friedrichs extension of the form defined for $\varphi\in\sC_0^\infty(1,+\infty)$, by:
\begin{equation*}
   \mathfrak{q}(\theta)(\varphi) =
   \int_1^{+\infty} |\partial_s\varphi|^2 - \frac{1}{4(s\sin\theta + 
   \pi \cos\theta)^2}\,|\varphi|^2 \;\dd s.
\end{equation*}
So the right hand side of \eqref{eqn:red1Dlow} has two blocs with separated variables: The first one is ${\mathfrak{q}}(\theta)$ in $s$ variable and the second one is the Dirichlet Laplacian quadratic form on $H^1_0(0,\pi)$ whose eigenvalues are $k^2$, for $k\ge1$ integer. We deduce
\begin{equation}
\mathcal{N}_{-E}({\mathfrak{q}}(\theta))\leq\mathcal{N}_{1-E}(\cQ_{\Sigma}(\theta)).
\label{eqn:lowinte}
\end{equation}
Let us perform the change of variables $\sigma = \frac{s + \pi\cot\theta}{1 + \pi\cot\theta}$. For all function $\varphi$ in the domain $\Dom({\mathfrak{q}}(\theta))$, we denote $\hat{\varphi}(\sigma) = {\varphi}(s)$. We get:
\begin{equation*}
\frac{{\mathfrak{q}}(\theta)(\varphi)}{\int_1^{+\infty}|\varphi|^2\;\dd s} = 
(1 + \pi\cot\theta)^{-2}\frac{\hat{\mathfrak{q}}(\theta)(\hat\varphi)}{\int_1^{+\infty}|\hat\varphi|^2\;\dd s}\ .
\end{equation*}
Using \eqref{eqn:lowbound1}, \eqref{eqn:lowinte} and the min-max principle, this achieves the proof of PropositionÊ \ref{prop:lowboundnbvp}.
\end{proof}

\subsection{An upper bound on the number of bound states}
\label{subsec:uppbound}
To obtain an upper bound, we follow the strategy of \cite{MT05} and \cite[\S\,5]{DLR12}. Let $(\chi_0,\chi_1)$ be a $\sC^\infty$ partition of unity such that:
\begin{equation*}
\chi_0(s)^2 + \chi_1(s)^2 = 1,
\end{equation*}
with $\chi_0(s) = 1$ for $s < 1$ and $\chi_0(s) = 0$ for $s > 2$. We set $W(s) = |\chi'_0(s)|^2 + |\chi'_1(s)|^2$. Now, we consider the quadratic form $\check{q}(\theta)$, Friedrichs extension of the form defined for all $\check{\varphi}\in\sC_0^\infty(1,+\infty)$:
\begin{equation*}
   \check{q}(\theta)(\check{\varphi}) = 
   \int_1^{+\infty}|\partial_s\check{\varphi}|^2 - \Big(\frac{1}{4s^2\sin^2\theta} + W(s)\Big)|
   \check{\varphi}|^2\;\dd s.
\end{equation*}

\begin{prop}
\label{prop:uppboundnbvp}
Let us choose $\theta\in(0,\frac{\pi}{2})$. There exists a constant $C(\theta)$ (depending only on $\theta$), such that for all $E>0$, we have:
\begin{equation*}
\mathcal{N}_{1-E}(\cQ_{\Omega(\theta)}) \leq C(\theta) + \mathcal{N}_{-E}(\check{q}(\theta)).
\end{equation*}
\end{prop}

\begin{proof}
For any function ${\psi}\in\Dom(\cQ_{\Omega(\theta)})$, we deduce from the definition of $W(s)$ and from the \enquote{IMS} formula (see \cite{CFKS87}):
\begin{multline}
   \cQ_{\Omega(\theta)}(\psi) = 
   \cQ_{\Omega(\theta)}(\chi_0\psi) + \cQ_{\Omega(\theta)}(\chi_1\psi) \\ -
   \int_{\Omega(\theta)}W(s)(|\chi_{0}\tilde\psi|^2 + |\chi_{1}\tilde\psi|^2)\,
   (s\sin\theta + u\cos\theta)\;\dd u \dd s.
\label{eqn:uppboundIMS}
\end{multline}
We introduce the subdomains $\Sigma_0 = \{(s,u)\in\Omega(\theta) : s < 2\}$ and $\Sigma_1=\Sigma$ of $\Omega(\theta)$.
Then we consider the quadratic forms $\cQ_{\Sigma_0}$ and $\cQ_{\Sigma_1}$ defined for $\phi\in\Dom(\cQ_{\Sigma_j})$ by:
\begin{equation*}
   \cQ_{\Sigma_j}(\phi) = 
   \int_{\Sigma_{j}}(\vabs{\partial_{s}\phi}^2 + \vabs{\partial_{u}\phi}^2 - W(s)\vabs{\phi}^2)\,
   (s\sin\theta+u\cos\theta)\;\dd u \dd s,\quad j=0,1,
\end{equation*}
where the form domains are defined by
\begin{align*}
  \Dom(\cQ_{\Sigma_0}) &= 
  \big\{\phi\in\Dom(\cQ_{\Omega(\theta)}):\quad \phi = 0\ \text{on }[2,\infty)\times(0,\pi)\big\},\\
  \Dom(\cQ_{\Sigma_1}) &= 
  \big\{\phi\in\Dom(\cQ_{\Omega(\theta)}):\quad \phi = 0\ \text{on }(-\infty,1]\times(0,\pi)\big\}.
\end{align*}
Of course, those quantities also depend on $\theta$. As we are looking for a result for a chosen $\theta$, we have dropped the mention of $\theta$ in these notations.
Thanks to \eqref{eqn:uppboundIMS}, we get:
\begin{equation*}
   \cQ_{\Omega(\theta)}(\psi) = \cQ_{\Sigma_{0}}(\chi_{0}\psi) + \cQ_{\Sigma_{1}}(\chi_{1}\psi).
\end{equation*}
Using \cite[Lemma 5.2]{DLR12} we find
\begin{equation}
\label{eqn:uppboundNE}
   \mathcal{N}_{1-E}(\cQ_{\Omega(\theta)}) \leq 
   \mathcal{N}_{1-E}(\cQ_{\Sigma_0}) + \mathcal{N}_{1-E}(\cQ_{\Sigma_1}).
\end{equation}
We now provide upper bounds for both terms of the right hand side of inequality \eqref{eqn:uppboundNE}.

1) We have, obviously: $\mathcal{N}_{1-E}(\cQ_{\Sigma_0})\le\mathcal{N}_{1}(\cQ_{\Sigma_0})$ for all $E>0$.
Then we note that
\begin{equation*}
   \int_{\Sigma_0}(|\partial_s\psi|^2 + |\partial_u\psi|^2)\,(s\sin\theta + u\cos\theta)\;\dd s \dd u 
   - \|W\|_\infty\|\psi\|_{\Sigma_0}^2 \leq\cQ_{\Sigma_0}(\psi)\,.
\end{equation*}
The quadratic form on the left hand side is associated with the axisymmetric Dirichlet Laplacian in a bounded domain. This operator has a compact resolvent and its eigenvalue sequence tends to infinity. We deduce
\begin{equation}
\label{eqn:NS0}
   \mathcal{N}_{1-E}(\cQ_{\Sigma_0})\le\mathcal{N}_{1}(\cQ_{\Sigma_0}) = C(\theta)<\infty
\end{equation}

2) For ${\psi}\in\Dom(\cQ_{\Sigma_{1}})$, we set $\hat{\psi}(s,u) = \sqrt{s\sin\theta+u\cos\theta}\,\psi(s,u)$ and we find (see \eqref{eqn:psihat}):
\begin{equation*}
\int_{\Sigma_{1}}|\partial_s\widehat\psi|^2 + |\partial_u\widehat\psi|^2 - \frac{1}{4s^2\sin^2\theta}|\widehat\psi|^2- W(s)|\widehat\psi|^2\dd s \dd u \le \cQ_{\Sigma_1}(\psi) .
\end{equation*}
Separating the variables, we obtain:
\begin{equation}
\label{eqn:NS1}
 \mathcal{N}_{1-E}(\cQ_{\Sigma_{1}(\theta)}) \leq \mathcal{N}_{-E}(\check{\mathfrak{q}}(\theta)).
\end{equation}

Now, we can end the proof of Proposition \ref{prop:uppboundnbvp}:
With estimates \eqref{eqn:NS0} and \eqref{eqn:NS1}, we obtain an upper bound for the left hand side of inequality \eqref{eqn:uppboundNE}. This yields Proposition \ref{prop:uppboundnbvp}.
\end{proof}

\subsection{Proof of TheoremÊ \ref{th:countvp}}
\label{subsec:proofthcount}
Thanks to Propositions \ref{prop:lowboundnbvp} and \ref{prop:uppboundnbvp} we have the following inequality:
\begin{equation}
\mathcal{N}_{-(1 + \pi\cot\theta)^{2}E}(\hat{\mathfrak{q}}(\theta)) \leq \mathcal{N}_{1-E}(\cQ_{\Omega(\theta)}) \leq C(\theta) + \mathcal{N}_{-E}(\check{\mathfrak{q}}(\theta)).
\label{eqn:encKSfin}
\end{equation}
Now we use Theorem \ref{th:KS88} on each side of \eqref{eqn:encKSfin}. For the left hand side, we obtain:
\begin{equation*}
\mathcal{N}_{-(1 + \pi\cot\theta)^2E}(\hat{\mathfrak{q}}(\theta)) \underset{E\rightarrow0}{\sim} \frac{\cot\theta}{4\pi}|\ln \big((1 + \pi\cot\theta)^2 E\big)| \underset{E\rightarrow0}{\sim} \frac{\cot\theta}{4\pi}|\ln E|.
\end{equation*}
On the right hand side, because $W\in\sC_0^\infty(1,+\infty)$ we get:
\begin{equation*}
\mathcal{N}_{-E}(\check{\mathfrak{q}}(\theta)) \underset{E\rightarrow0}{\sim} \frac{\cot\theta}{4\pi}|\ln (E)|.
\end{equation*}
Together with \eqref{eqn:encKSfin}, it gives Theorem \ref{th:countvp}.
\hfill$\square$\par\medskip

\section{Conical layers in the small aperture limit}
\label{sec:conicalsemiclass}
This section is devoted to the proof of Theorem \ref{th:expavp}. We work with the meridian guide $\Gui(\theta)$ and the operator $\cH_{\Gui(\theta)}$ introduced in SubsectionÊ \ref{subsec:dirlapl}. $\Gui(\theta)$ depends on $\theta$ and it is more convenient to transfer its dependence into the coefficients of the operator. Therefore, we perform the change of variable:
\begin{equation*}
x = z\sqrt{2}\sin\theta,\quad y=r\sqrt{2}\cos\theta.
\end{equation*}
The meridian guide $\Gui(\theta)$ becomes $\Gui = \Gui(\frac{\pi}4)$ and $\cH_{\Gui(\theta)}$ is unitary equivalent to
\begin{equation}
\mathcal{D}_{\Gui}(\theta) = -2\sin^2\theta\partial_x^2 - 2\cos^2\theta\frac1y\partial_y(y\partial_y).
\label{eqn:defopD}
\end{equation}
Let $h = \tan\theta$, dividing by $2\cos^2\theta$ we obtain the partially semiclassical operator in the $x$-variable:
\begin{equation*}
\cL_{\Gui}(h) = -h^2\partial_x^2 - \frac1y\partial_y(y\partial_y).
\end{equation*}
This operator acts on $L^2(\Gui,y\dd x \dd y)$ and its eigenvalues, denoted by $(\lambda_n(h))_{n\geq1}$ satisfy:
\begin{equation*}
\lambda_n(\tan\theta) = (2\cos^2\theta)^{-1}\mu_n(\theta).
\end{equation*}
By definition, the regime $\theta\rightarrow0$ corresponds to the semiclassical regime $h\rightarrow0$. We will prove Theorem \ref{th:expavp} using the scaled operator $\cL_{\Gui}(h)$.

Besides, we introduce some notations and results that are useful in this section. We define $\Tri$, the triangular end of $\Gui$, by:
\begin{equation*}
\Tri = \{(x,y)\in\Gui : x <0\}.
\end{equation*}
We consider the operator $\cL_\Tri(h)$ defined by:
\begin{equation*}
\cL_\Tri(h)= -h^2\partial_x^2 - \frac1y\partial_y(y\partial_y),
\end{equation*}
with domain $\Dom(\cL_\Tri(h)) =\{\psi \in \Dom(\cL_{\Gui}(h)) : \psi(0,\cdot) = 0, \mathsf{supp}\ \psi\subset\Tri\}$. In fact, $\cL_\Tri(h)$ is associated with the $0$-th fiber of the Dirichlet Laplacian in a cone. In \cite[Proposition 9]{O14}, one of the authors studies the eigenpairs of cones of small aperture:

\begin{theorem}
\label{th:O14}
The eigenvalues of $\cL_\Tri(h)$, denoted by $\lambda_{n}^\Delta(h)$, admit the expansions:
\begin{equation*}
\lambda_n^\Delta(h) \underset{h\rightarrow0}{\sim}\sum_{k\geq0}\beta_{k,n}h^{k/3},\quad\text{with } \beta_{0,n} = \frac{j_{0,1}^2}{2\pi^2}, \beta_{1,n} = 0, \beta_{2,n} = \frac{(2j_{0,1}^2)^{2/3}}{2\pi^2}z_{\mathsf{A}}(n).
\end{equation*}
Here $\underset{h\rightarrow0}{\sim}$ means that for all $K\in\mathbb{N}$, there exist $C_K>0$ and $h_0>0$ such that for all $h\in(0,h_0)$ we have:
$
|\lambda_n^\Delta(h) - \sum_{k=0}^K\beta_{k,n}h^{k/3}|\leq C_K h^{(K+1)/3}.
$
\end{theorem}

This section is divided into three parts. In Section \ref{subsec:effpot} we study a one dimensional operator with electric potential $v$, the so-called Born-Oppenheimer approximation of $\cL_{\Gui}(h)$. Since this reduced operator is a lower bound of $\cL_{\Gui}(h)$ and thanks to the confining properties of $v$, we deduce, in Section \ref{subsec:Agmon}, that  the eigenfunctions of $\cL_{\Gui}(h)$ are essentially localized in $\Tri$. In Section \ref{subsec:proofexp}, we infer that the two-terms semiclassical expansions of the eigenvalues of $\cL_{\Gui}(h)$ and $\cL_\Tri(h)$ coincide.

\subsection{Effective one dimensional operator}
\label{subsec:effpot}
As the operator $\cL_{\Gui}(h)$ is partially semiclassical, we may use the strategy of the famous Born-Oppenheimer approximation (see \cite{BO27,CDS81, J14, KMSW92,M89, M07}). 
Let us denote by $v(x)$ the first eigenvalue of the operator, acting on $L^2\left(\left(\max(0,x), x+\pi\sqrt{2}\right), y\dd y\right)$ and defined by $\mathfrak{h}_{x}=- \frac1y\partial_y(y\partial_y)$ with Dirichlet conditions in $y=x+\pi\sqrt{2}$ and in $y=x$ if $x>0$. Since we are concerned with the expansion of the low lying eigenvalues, this approximation consists in replacing $- \frac1y\partial_y(y\partial_y)$ in the expression of $\cL_{\Gui}(h)$ by $v(x)$ on each slide of $\Gui$ at fixed $x$. Thus we consider the following one dimensional Schr\"odinger operator, acting on $L^2(-\pi\sqrt{2},+\infty)$,
\begin{equation}
-h^2\partial_x^2 +v(x).
\label{eqref:defBO}
\end{equation}
Let us now describe the function $v$. For each $x>-\pi\sqrt{2}$, we consider the quadratic form $\mathfrak{q}_x^\trans$ associated with the transverse operator $- \frac1y\partial_y(y\partial_y)$
\begin{equation*}
   \mathfrak{q}_x^\trans(\varphi) = 
   \int_{\max(0,x)}^{x+\pi\sqrt{2}}|\partial_y\varphi|^2\;y\dd y.
\end{equation*}
with Dirichlet boundary conditions in $y=x+\pi\sqrt{2}$ and in $y=x$ if $x>0$. Let $\Dom(\mathfrak{q}_x^\trans)$ denote its form domain. Then
\[
   v(x) = \min_{\varphi\in\Dom(\mathfrak{q}_x^\trans)}
   \frac{\mathfrak{q}_{x}^\trans(\varphi)}{\int_{\max(0,x)}^{x+\pi\sqrt{2}}|\varphi|^2y\dd y}\;.
\]
For $x\in(-\pi\sqrt{2},0)$, we have the explicit expression (first eigenvalue of the Dirichlet problem on the disk of radius $x+\pi\sqrt{2}$, cf.\ \cite[\S\,3.3]{O14}):
\begin{equation}
\label{eqn:vneg}
   v(x) = \frac{j_{0,1}^2}{(x+\pi\sqrt{2})^2}\,.
\end{equation}

\begin{prop}
\label{prop:effpot}
The effective potential $v$ has the following properties:
\begin{enumerate}[(i)]
\item $v$ is continuous and decreasing on $(-\pi\sqrt{2},0]$, continuous and non-decreasing on $(0,\infty)$.
\item The infimum of $v$ on $(-\pi\sqrt{2},\infty)$ is positive.
\item For all $x>0$, we have:
\begin{equation*}
   \frac12-\frac{1}{4x^2}\le v(x)\le \frac12\,.
\end{equation*}
\item The effective potential $v$ is continuous at $0$:
\begin{equation*}
   v(0) = \frac{j_{0,1}^2}{2\pi^2} \quad\mbox{and}\quad
   v(x) - \frac{j_{0,1}^2}{2\pi^2} \underset{x\rightarrow0^+}{\sim}
   \frac{1}{\ln x}\,\frac{j_{0,1}}{2\pi}\,\frac{Y_0(j_{0,1})}{J_0'(j_{0,1})}
   = \frac{1}{|\ln x|}\,\frac{j_{0,1}}{2\pi}\,\frac{|Y_0(j_{0,1})|}{|J_0'(j_{0,1})|}\,,
\end{equation*}
where $J_0$ and $Y_0$ are the Bessel functions of first and second kind, respectively.
\end{enumerate}
So $v$ is continuous on $(-\pi\sqrt{2},\infty)$ and attains its (unique) minimum at $x=0$, see Figure \ref{fig:allpot}.
\end{prop}

\begin{figure}
\includegraphics[scale=0.66]{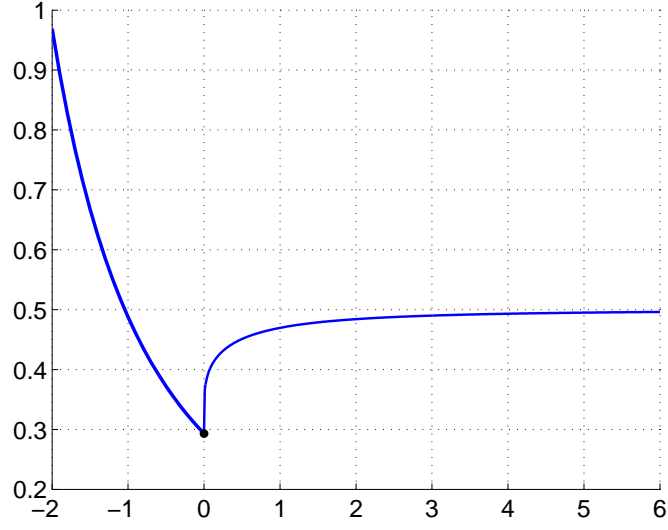}
\caption{The effective potential $v$}
\label{fig:allpot}
\end{figure}

\begin{proof}
\textit{(i)} The statement for $x\le0$ is an obvious consequence of the expression \eqref{eqn:vneg} of $v$. To tackle the case $x>0$, we perform the change of variables $\hat{y} = y - x$ that transforms the quadratic form $\mathfrak{q}_x^\trans$ into
\[
   \int_{0}^{\pi\sqrt{2}}|\partial_{\hat y}\varphi(\hat y+x)|^2\;(\hat y+x)\;\dd \hat y.
\]
To get rid of the metrics, we use the change of function defined by
\[
   \hat\varphi(\hat y) = \sqrt{\hat y+x}\;\varphi(\hat y+x).
\]
It transforms the quadratic form $\mathfrak{q}_x^\trans$ into $\hat{\mathfrak{q}}_x^\trans$ defined by:
\begin{equation*}
   \hat{\mathfrak{q}}_x^\trans(\hat\varphi) = 
   \int_{0}^{\pi\sqrt{2}} |\partial_{\hat{y}}\hat\varphi|^2
   - \frac{1}{4(\hat y+x)^2}\, |\hat\varphi|^2\; \dd \hat{y}\,,
\end{equation*}
so that we have the identities
\[
   \mathfrak{q}_x^\trans(\varphi)=\hat{\mathfrak{q}}_x^\trans(\hat\varphi)
   \quad\mbox{and}\quad
   \int_{x}^{x+\pi\sqrt{2}}|\varphi|^2\,y\,\dd y = \int_{0}^{\pi\sqrt{2}}|\hat\varphi|^2\;\dd y\,.
\]
For any positive $x$, the domain of the form $\hat{\mathfrak{q}}_x^\trans$ is $H^1_0(0,\pi\sqrt{2})$. At this point it becomes obvious that for any $\hat\varphi\in H^1_0(0,\pi\sqrt{2})$ the Rayleigh quotients $\hat{\mathfrak{q}}_x^\trans(\hat\varphi) \|\hat\varphi\|^{-2}$ are non-decreasing functions of $x$, hence the corresponding monotonicity of $v$ on $(0,\infty)$.

\textit{(ii)} For $a\in[0,\pi\sqrt{2})$, let us introduce the operator $\tilde{\mathfrak{h}}_{a}=-y^{-1}\partial_{y}y\partial_{y}$, acting on $L^2\left(\left(a,\pi\sqrt{2}\right), y\dd y\right)$ with Dirichlet conditions at $a$ and $\pi\sqrt{2}$, and denote by $\tilde v(a)$ its lowest eigenvalue, which is a non decreasing function of $a$. We note that $\tilde v(0)$ is larger than the first eigenvalue ${j_{0,1}^2}/{2\pi^2}$ of the Dirichlet Laplacian on the disc of radius $\pi\sqrt{2}$. By dilation, we get that
\[
   v(x) = 
   \frac{2\pi^2}{(x+\pi\sqrt{2})^2}\,\tilde v\Big(\frac{x\,\pi\sqrt{2}}{x+\pi\sqrt{2}}\Big)\geq
   \frac{2\pi^2}{(x+\pi\sqrt{2})^2}\,\frac{j_{0,1}^2}{2\pi^2}\,.
\]
Combining this with {\em (i)}, we obtain {\em (ii)}.

\textit{(iii)} For $x>0$, we bound $\hat{\mathfrak{q}}_x^\trans(\hat\varphi)$ from below
\[
   \hat{\mathfrak{q}}_x^\trans(\hat\varphi) \ge 
   \int_{0}^{\pi\sqrt{2}} |\partial_{\hat{y}}\hat\varphi|^2
   - \frac{1}{4x^2}\, |\hat\varphi|^2\; \dd \hat{y}
\]
Hence, since the first eigenvalue of the right hand side is $\frac12-\frac{1}{4x^2}$, we find that 
$v(x) \ge \frac12-\frac{1}{4x^2}$.

\textit{(iv)} Let $x>0$. The equation on eigenpairs $(\lambda,\varphi)$ of $\mathfrak{h}_{x}$ can be written as
\begin{gather}
\label{eqn:Bessel}
   y^2\partial^2_y\varphi + y\partial_y\varphi + \lambda y^2\varphi = 0, \\
\label{eqn:Dir}
   \varphi(x)=0 \quad\mbox{and}\quad \varphi(x+\pi\sqrt{2})=0\,.
\end{gather}
The first equation is a Bessel type equation whose general solution can be written as
\[
   A J_{0}({\lambda}^{1/2}{y}) + B Y_{0}({\lambda}^{1/2}{y}),\quad A,B\in\R.
\]
Finding a non-zero solution  to \eqref{eqn:Bessel}-\eqref{eqn:Dir} is equivalent to find a non-zero couple $(A,B)$ such that the above function satisfies both Dirichlet boundary conditions at $x$ and $x+\sqrt{2}$. An equivalent condition is that the following determinant is zero:
\[
   \begin{vmatrix}
   J_{0}({\lambda}^{1/2}{x}) & Y_{0}({\lambda}^{1/2}{x}) \\
   J_{0}({\lambda}^{1/2}(x+\pi\sqrt{2})) & Y_{0}({\lambda}^{1/2}(x+\pi\sqrt{2}))
\end{vmatrix} = 0\,,
\]
So, the effective potential $v$ satisfies the implicit equation:
\begin{equation}
   J_{0}\big(v(x)^{1/2}x\big)\,Y_{0}\big(v(x)^{1/2}(x+\pi\sqrt{2})\big) = 
   J_{0}\big(v(x)^{1/2}(x+\pi\sqrt{2})\big)\,Y_{0}\big(v(x)^{1/2}x\big).
\label{eqn:impliciteqpoteff}
\end{equation}
We know that the effective potential $v$ is bounded from above on $(0,\infty)$ by {\em(iii)} and is bounded from below by ${j_{0,1}^2}/{2\pi^2}$ by {\em(ii)}. Hence, we deduce:
\begin{equation}
   v(x)^{1/2}x \underset{ x \rightarrow 0^+}{\longrightarrow} 0
   \quad\mbox{and}\quad
   v(x)^{1/2}(x+\pi\sqrt{2)} \underset{ x \rightarrow 0^+}{\longrightarrow} \alpha\ge j_{0,1}\,.
\end{equation}
Since the left hand side of \eqref{eqn:impliciteqpoteff} is bounded and $Y_0(x)\underset{x\rightarrow0^+}{\longrightarrow}-\infty$ we get:
\begin{equation}
   J_{0}\big(v(x)^{1/2}(x+\pi\sqrt{2})\big) \underset{ x \rightarrow 0^+}{\longrightarrow} 0.
\label{eqn:limpt}
\end{equation}
Taking the limit in the left hand side of \eqref{eqn:limpt}, the only possible accumulation points of $\{v(x)\}$ as $x\to0_+$ are $\frac{j_{0,p}^2}{2\pi^2}$ ($p\in\mathbb{N}^*$). Because for all $p\geq2$, $\frac{j_{0,p}^2}{2\pi^2}>\frac12$, thanks to {\em(ii)} we get the continuity of the effective potential in $x=0$.
The $J_0$ zeros are simple, thus $J_0'(j_{0,1})\neq0$. We obtain:
\begin{equation}
\begin{array}{rcl}
   J_{0}\big(v(x)^{1/2}(x+\pi\sqrt{2})\big) &\underset{x \rightarrow0^+}{\sim}&
   \big(v(x)^{1/2}(x+\pi\sqrt{2}) - j_{0,1}\big)\,J'(j_{0,1}),\\
   &\underset{x \rightarrow0^+}{\sim}&
   \big(\sqrt{2\pi^2v(x)} - j_{0,1}\big)\,J'(j_{0,1}),\\[1.5ex]
   Y_{0}\big(v(x)^{1/2}x\big)&\underset{x \rightarrow0^+}{\sim}&
   \displaystyle\frac2\pi\, \ln\Big(\frac12v(x)^{1/2}x\Big)
   \underset{x \rightarrow0^+}{\sim}
   \displaystyle\frac2\pi \,\ln x\,.
\end{array}
\label{eqn:equivBess}
\end{equation}
Moreover, we have:
\begin{equation}
   J_{0}\big(v(x)^{1/2}x\big)\underset{x \rightarrow0^+}{\longrightarrow} 1
   \quad \mbox{and}\quad
   Y_{0}\big(v(x)^{1/2}(x+\pi\sqrt{2})\big)\underset{x \rightarrow0^+}{\longrightarrow}Y_{0}(j_{0,1}).
\label{eqn:limitBess}
\end{equation}
Equations \eqref{eqn:impliciteqpoteff}, \eqref{eqn:equivBess} and \eqref{eqn:limitBess} yield:
\begin{equation*}
   \frac2\pi \,\ln x \underset{x\rightarrow0^+}{\sim}
   \frac{Y_{0}(j_{0,1})}{\big(\sqrt{2\pi^2v(x)} - j_{0,1}\big)\,J_{0}'(j_{0,1})}\,.
\end{equation*}
This provides the proof of the asymptotic equivalence of $v-\frac{j^2_{0,1}}{2\pi^2}$ with $\frac{1}{\ln x}\,\frac{j_{0,1}}{2\pi}\,\frac{Y_0(j_{0,1})}{J_0'(j_{0,1})}$. The positivity of the quotient $\frac{1}{\ln x}\,\frac{Y_0(j_{0,1})}{J_0'(j_{0,1})}$ is then a consequence of {\em (i)} ($v$ non-decreasing on $(0,\infty)$).
\end{proof}

Here, we are not going to investigate the spectrum of the one dimensional operator defined in \eqref{eqref:defBO}, nevertheless one can see that the logarithmic behavior of $v$ near $x=0^+$ prevents us from using the classical harmonic approximation. 

\subsection{Agmon localization estimates}
\label{subsec:Agmon}
Before stating and proving Agmon localization estimates, we remark that for $N_0\in\mathbb{N}^*$, there exists $\Gamma_0>0$ and $h_0>0$ such that for all $n\in\{1,\dots,N_0\}$ and all $h\in(0,h_0]$:
\begin{equation}
\Big|\lambda_n(h) - \frac{j_{0,1}^2}{2\pi^2}\Big|\leq\Gamma_0 h^{2/3}.
\label{eqn:infirsvp}
\end{equation}
To obtain inequality \eqref{eqn:infirsvp}, we first observe that $\lambda_n(h) - \frac{j_{0,1}^2}{2\pi^2}\geq0$. This comes from the following fact: If we denote by $\cQ(h)$ the quadratic form associated with the operator $\cL_{\Gui}(h)$, for all $\psi$ in the form domain $\Dom(\cQ(h))$ we have:
\begin{equation}\label{lb-BO}
\cQ(h)(\psi) = \int_{\Gui} h^2|\partial_{x}\psi|^2 + |\partial_{y}\psi|^2 y \dd x \dd y \geq \int_{\Gui} h^{2} |\partial_{x}\psi|^2 + v(x) |\psi|^2 y \dd x \dd y\geq\frac{j_{0,1}^2}{2\pi^2}\|\psi\|^2.
\end{equation}
The min-max principle yields $\lambda_n(h) - \frac{j_{0,1}^2}{2\pi^2}\geq0$. Since $\Tri\subset\Gui$, by Dirichlet bracketing we get:
\begin{equation}
\lambda_n(h)\leq \lambda_n^\Delta(h).
\label{eqn:majvp}
\end{equation}
Now, by Theorem \ref{th:O14}, we know that $|\lambda_n^\Delta(h) -\frac{j_{0,1}^2}{2\pi^2}|\leq\Gamma_0 h^{2/3}$, which gives inequality \eqref{eqn:infirsvp}.

As a consequence of Proposition \ref{prop:effpot} {\em (iv)}, there exists $x_1>0$ such that
\begin{equation}
\label{eqn:minBO}
   \forall x\in[0,x_1],\quad
   v(x) \geq \frac{j_{0,1}^2}{2\pi^2} +\frac{c_{0}}{|\ln x|}
   \quad\mbox{with}\quad
   c_0 = \frac{j_{0,1}}{4\pi}\,\frac{|Y_0(j_{0,1})|}{|J_0'(j_{0,1})|}\;.
\end{equation}
We set
\[
   \G_1 = \Gui\cap \{(x,y) : \ x \in (0,x_1)\}
   \quad\mbox{and}\quad
   \G_2 = \Gui\cap \{(x,y) : \ x \in (x_1,\infty)\}.
\]


Our estimates of Agmon type are as follows.
\begin{prop} 
\label{prop:agmlay}
Let $\Gamma_0>0$, there exist $h_0, C_0>0$ and $\eta_1,\eta_2,\eta_3>0$ such that for all $h\in(0,h_0)$ and all eigenpair $(\lambda,\psi)$ of $\cL_{\Gui}(h)$ satisfying $|\lambda - \frac{j_{0,1}^2}{2\pi^2}|\leq\Gamma_0 h^{2/3}$, we have:
\begin{equation}
\label{eqn:agmlay}
   \int_{\Gui} e^{2\Phi/h}\big(|\psi|^2 + |h\partial_{x}\psi|^2\big)\,y\,\dd x\dd y 
   \leq C_{0}\|\psi\|^2,
\end{equation}
with the Lipschitz weight function $\Phi$ defined on $\Gui$ by
\begin{equation}
\label{eqn:agmweight}
\Phi(x,y) = \begin{cases}
   \Phi_{0}(x) = \eta_{0}|x|^{3/2} & 
   \mbox{if }\ \pi\sqrt{2}<x<0, \ \mbox{ i.e. }\ (x,y)\in\Tri \\[0.5ex]
   \Phi_{1}(x) = \eta_{1}\displaystyle\int_{0}^{x}\frac{1}{\sqrt{|\ln t|}}\dd t, & 
   \mbox{if }\ 0<x<x_1 \ \mbox{ i.e. }\ (x,y)\in\G_1\\
   \Phi_{2}(x) = \eta_{2}\big(x-x_1\big) + \Phi_{1}(x_1)
   & \mbox{if }\ x_{1}<x \ \mbox{ i.e. }\ (x,y)\in\G_2.
\end{cases}
\end{equation}
\end{prop}
Thanks to these estimates we understand the localization scales of the eigenfunctions of $\cL_{\Gui}(h)$. In the triangular end $\Tri$, they are localized at a scale $h^{2/3}$ near $x=0$ whereas they get into the layer at a scale $h\sqrt{|\ln h|}$. This is different from the two dimensional waveguides with corner where the eigenfunctions visit the guiding part at a scale $h$. In \cite{DR12}, the effective potential obtained has a jump at $x=0$, then is constant. 
Here it is different: The logarithmic behavior of $v$ at $x=0^+$ allows the eigenfunctions to leak a little more outside the triangular end.

\begin{proof} For $\Phi$ a Lipschitz function only depending on the variable $x\in(-\pi\sqrt{2},+\infty)$, we have the formula of "IMS" type:
\begin{equation*}
   \int_{\Gui} \Big(h^2|\partial_x(e^{\Phi/h}\psi)|^2 + |e^{\Phi/h}\partial_y\psi|^2 
   - (\lambda + \Phi'(x)^2)|e^{\Phi/h}\psi|^2\Big) \,y\,\dd x\dd y = 0.
\end{equation*}
Now, thanks to the first inequality in \eqref{lb-BO} we deduce:
\begin{equation*}
   \int_{\Gui} \Big(h^2|\partial_x(e^{\Phi/h}\psi)|^2 + 
   (v(x) - \lambda - \Phi'(x)^2)|e^{\Phi/h}\psi|^2\Big) \,y\,\dd x\dd y \leq 0.
\end{equation*}
By convexity of $v$ in $(-\pi\sqrt{2},0)$ and inequality \eqref{eqn:minBO} we get:
\begin{equation*}
   \int_\Gui h^2|\partial_x(e^{\Phi/h}\psi)|^2 \,y\,\dd x\dd y + I_{\Tri} + I_{\G_1} + I_{G_2}\leq 0,
\end{equation*}
with
\begin{equation*}
   I_{\Tri} = \int_{\Tri}\Big(\frac{j_{0,1}^2}{2\pi^2} + 
   \frac{j_{0,1}^2}{(\pi\sqrt{2})^3}|x| - \lambda - \Phi'(x)^2\Big)|e^{\Phi/h}\psi|^2 \,y\,\dd x\dd y,
\end{equation*}
\begin{equation*}
   I_{\G_1} = \int_{\G_1}\Big(\frac{j_{0,1}^2}{2\pi^2} +\frac{c_{0}}{|\ln x|} 
   - \lambda - \Phi'(x)^2 \Big)|e^{\Phi/h}\psi|^2 \,y\,\dd x\dd y,
\end{equation*}
\begin{equation*}
   I_{\G_2} = \int_{\G_2} \big(v(x_1) - \lambda -\Phi'(x)^2\big)|e^{\Phi/h}\psi|^2 \,y\,\dd x\dd y.
\end{equation*}
Then, using $\lambda - \frac{j_{0,1}^2}{2\pi^2}\leq\Gamma_0 h^{2/3}$ , this becomes:
\begin{equation}\label{ineq:Agmon}
   \int_\Gui h^2|\partial_x(e^{\Phi/h}\psi)|^2 \,y\,\dd x\dd y 
   + \hat{I}_{\Tri} + \hat{I}_{\G_1} + \hat{I}_{\G_2}\leq 0,
\end{equation}
with
\begin{equation*}
   \hat{I}_{\Tri} = \int_{\Tri}\Big(\frac{j_{0,1}^2}{(\pi\sqrt{2})^3}|x| 
   - \Gamma_{0}h^{2/3} - \Phi'(x)^2\Big)|e^{\Phi/h}\psi|^2 \,y\,\dd x\dd y,
\end{equation*}
\begin{equation*}
   \hat{I}_{\G_1} = \int_{\G_1}\Big(\frac{c_{0}}{|\ln x|}  
   - \Gamma_{0}h^{2/3} - \Phi'(x)^2\Big)|e^{\Phi/h}\psi|^2 \,y\,\dd x\dd y,
\end{equation*}
\begin{equation*}
   \hat{I}_{\G_2} = \int_{\G_2} \big(v(x_1) - \frac{j_{0,1}^2}{2\pi^2}
   - \Gamma_{0}h^{2/3} -\Phi'(x)^2\big)|e^{\Phi/h}\psi|^2\,y\,\dd x\dd y,
\end{equation*}
We are led to take:
\begin{equation*}
   \Phi(x,y) = \Phi_{0}(x)\mathds{1}_{(-\pi\sqrt{2},0)}(x) + 
   \Phi_{1}(x)\mathds{1}_{(0,x_1)}(x) + \Phi_{2}(x)\mathds{1}_{(x_1,+\infty)}(x),
\end{equation*}
where $\Phi_{0}$, $\Phi_1$, and $\Phi_2$ are defined by \eqref{eqn:agmweight} for some positive constants $\eta_0$, $\eta_1$, and $\eta_2$.
For $\eta_{0}$, $\eta_{1}$, and $\eta_{2}$ small enough, there exist positive $\hat\eta_{0}$, $\hat\eta_{1}$, and $\hat\eta_{2}$ such that:
\begin{equation*}
   \hat{I}_{\Tri} =\int_{\Tri}(\hat\eta_{0}|x| 
   - \Gamma_{0}h^{2/3})|e^{\Phi/h}\psi|^2 \,y\,\dd x\dd y,
\end{equation*}
\begin{equation*}
   \hat{I}_{\G_1} = \int_{\G_1}(\hat\eta_{1}|\ln x|^{-1} 
   - \Gamma_{0}h^{2/3})|e^{\Phi/h}\psi|^2 \,y\,\dd x\dd y,
\end{equation*}
\begin{equation*}
   \hat{I}_{\G_2} = \int_{\G_2} \big(\hat\eta_{2} 
   - \Gamma_{0}h^{2/3}\big)|e^{\Phi/h}\psi|^2 \,y\,\dd x\dd y.
\end{equation*}
Note that, if $h$ is small, we may find a positive constant $\gamma_{2}$ such that $\hat{I}_{\G_2}\geq \gamma_{2}\int_{\G_{2}}|e^{\Phi/h}\psi|^2 \,y\,\dd x\dd y$.
Let $\varepsilon>0$ be a chosen positive number. We define the $h$-dependent subdomains:
\begin{equation*}
   \Tri^\far = \{(x,y) \in \!\Tri, \ \hat\eta_{0}|x| - \Gamma_{0}h^{2/3} \geq \varepsilon h^{2/3}\},
   \quad\G_1^\far = \{(x,y)\in\G_1, \ 
   \frac{\hat\eta_{1}}{|\ln x|} - \Gamma_{0}h^{2/3} \geq \varepsilon h^{2/3}\},
\end{equation*}
\begin{equation*}
   \Tri^\near \!\!= \{(x,y) \in \!\Tri, \ \hat\eta_{0}|x| - \Gamma_{0}h^{2/3} \leq \varepsilon h^{2/3}\},
   \quad\G_1^\near \!\!= \{(x,y)\in\G_1, \ 
   \frac{\hat\eta_{1}}{|\ln x|} - \Gamma_{0}h^{2/3} \leq \varepsilon h^{2/3}\},
\end{equation*}
Then, we split the integrals in \eqref{ineq:Agmon} and we obtain:
\begin{multline}
\label{eqn:preagmon}
   \Gamma_{0}h^{2/3}\Big(
   \int_{\Tri^\near} |e^{\Phi_{1}/h}\psi|^2 \,y\,\dd x\dd y 
   + \int_{\G_1^\near} |e^{\Phi_{1}/h}\psi|^2 \,y\,\dd x\dd y\Big) \geq  
   \int_{\Gui} h^2|\partial_{x}(e^{\Phi/h}\psi)|^2 y \dd x \dd y\\
   + \varepsilon h^{2/3} \int_{\Tri^\far} |e^{\Phi_{1}/h}\psi|^2 \,y\,\dd x\dd y
   + \varepsilon h^{2/3} \int_{\G_1^\far}|e^{\Phi_{1}/h}\psi|^2 \,y\,\dd x\dd y 
   + \hat{I}_{\G_2}.
\end{multline}
Set $\gamma_1:=(\varepsilon+\Gamma_0)/\hat\eta_1$. In $\Tri^\near$ there holds $-h^{2/3}\gamma_1 < x < 0$. Thus we find
\begin{equation*}
   \frac{\Phi_{0}(x)}{h} \leq \eta_{1}\gamma_1^{3/2}.
\end{equation*}
Set $\gamma_2:=(\varepsilon+\Gamma_0)/\hat\eta_2$. In $\G_1^\near$, we have $|\ln x|^{-1/2}\le h^{1/3}\sqrt{\gamma_2}$ and $0<x<e^{-h^{-2/3}\gamma_{2}^{-1}}$. Thus
\begin{equation*}
   \frac{\Phi_2(x)}{h}  \leq \eta_{2} \,x \,h^{1/3}\sqrt{\gamma_2}\,\frac{1}{h}
   \leq \eta_{2} \,e^{-h^{-2/3}\gamma_{2}^{-1}} \,h^{-2/3}\sqrt{\gamma_2}.
\end{equation*}
We deduce that $\Phi/h$ is bounded independently on $h$ on $\Tri^\near\cup\G_1^\near$. 
Now, the Agmon-type estimate \eqref{eqn:agmlay} is a consequence of inequality \eqref{eqn:preagmon}.
\end{proof}

\subsection{Proof of Theorem \ref{th:expavp}}
\label{subsec:proofexp}
Let $N_0\in\mathbb{N}^*$. We consider the $N_0$ first eigenvalues of $\cL_{\Gui}(h)$. In each eigenspace associated with $\lambda_{n}(h)$ we choose a normalized eigenfunction $\psi_{n,h}$ so that $\langle\psi_{n,h},\psi_{p,h}\rangle = 0$ if $n\neq p$. Moreover, if $\cB(h)$ (resp. $\cQ(h)$) denotes the bilinear (resp. quadratic) form associated with $\cL_{\Gui}(h)$, for $n\neq p$ we also have $\cB(h)(\psi_{n,h},\psi_{p,h}) = 0$ and $\cQ(h)(\psi_{n,h}) = \lambda_{n}(h)\|\psi_{n,h}\|^2$.

Let $a\ge2$, we introduce a smooth cut-off function $\chi$ defined for $x>0$ and satisfying:
\begin{equation*}
\chi(x) = 1, \text{ if } x\leq a\quad\mbox{and}\quad \chi(x) = 0, \text{ if } x\geq 2a.
\end{equation*}
We define the cut-off function $\chi_h$ at the scale $h|\ln h|^{3/2}$ by $\chi_h(x) = \chi(x h^{-1}|\ln h|^{-3/2})$ (for all $x>0$). The \enquote{IMS} formula yields
\begin{equation*}
   \cQ(h)(\chi_h\psi_{n,h}) =\lambda_{n}(h)\|\chi_{h}\psi_{n,h}\|^2 + R(h) \,,
\end{equation*}
where the commutator term $R(h)$ is given by
\begin{equation*}
   R(h) =  \frac{1}{|\ln h|^3}\int_{\Gui}|\chi'(xh^{-1}|\ln h|^{-3/2})|^2 \,|\psi_{n,h}|^2 
   \,y\, \dd x \dd y 
\end{equation*}
With the weight function $\Phi$ defined by \eqref{eqn:agmweight} we obtain, thanks to the estimate \eqref{eqn:agmlay}:
\begin{align*}
   R(h) &\leq  |\ln h|^{-3} \,\|\chi'\|_{L^{\infty}(-\pi\sqrt{2},+\infty)}
   \int_{\Gui\cap\{(x,y) : a \leq x h^{-1}|\ln h|^{-3/2}\leq 2a\}}
   e^{2\Phi/h}e^{-2\Phi/h}|\psi_{n,h}|^2 \,y\,\dd x\dd y \\
   &\leq C |\ln h|^{-3} \,e^{-2\Phi(a h|\ln h|^{3/2})/h} \|\psi_{n,h}\|^2.
\end{align*}
Thanks to the lower bound on the weight function, we get, for $h$ small enough,
\[
   2\Phi(ah|\ln h|^{3/2}) \ge 2\eta_{1}\int_{ah/2}^{ah|\ln h|^{3/2}} \frac{1}{\sqrt{|\ln \left(ah/2\right)|}}\dd t
   \ge \eta_1 ah|\ln h|^{3/2} |\ln \left(ah/2\right)|^{-1/2},
\]
and, using that $a\ge2$, we find
\[
   2\Phi(ah|\ln h|^{3/2})/h \ge \eta_1 a |\ln h| = -\eta_1 a \ln h.
\]
Now we choose $a = \max\{2,2/\eta_1\}$ and find
$e^{-2\Phi(a h|\ln h|^{3/2})/h} \le e^{2\ln h} = h^2$. Thus we deduce:
\begin{equation}
\label{eqn:R1}
   R(h) \le C h^2 . 
\end{equation}
Similarly the functions $\chi_h\psi_{n,h}$ are almost orthogonal for the bilinear form $\cB(h)$ in the sense that there holds, for $n\neq p$:
\begin{equation*}
  \big| {\cB}(h)(\chi_h\psi_{n,h},\chi_h\psi_{p,h}) \big| \le C h^2  .
\end{equation*}
We introduce
\begin{equation*}
\mathcal{S}_{N_{0}}(h) = \mathsf{vect}(\chi_h\psi_{1,h},\dots,\chi_h\psi_{N_{0},h}).
\end{equation*}
and we get
\begin{equation}\label{ineq:SN}
\forall \psi_h\in \mathcal{S}_{N_{0}}(h),\qquad   \cQ(h)(\psi_h) \leq (\lambda_{N_{0}}(h) + Ch^2)\|\psi_h\|^2.
\end{equation}
Now, we define the triangle $\Tri(h)$ by its vertices
\[
   (-\pi\sqrt{2},0), \quad(h|\ln(h)|^{3/2},0), \quad\mbox{and}\quad 
   (h|\ln(h)|^{3/2}, h|\ln(h)|^{3/2} + \pi\sqrt{2}),
\]
and consider the operator $\cL_{\Tri(h)}(h):=-h^2\partial^2_x-\frac1y\partial_y(y\partial_y)$ on $\Tri(h)$ with Dirichlet condition on the two sides of $\Tri(h)$ that are not contained in the axis $r=0$. Since $\Tri(h)$ can be obtained by a dilation of ratio $1+\frac{h|\ln(h)|^{3/2}}{\pi\sqrt{2}}$ from $\Tri$, we find that the eigenvalues of $\cL_{\Tri(h)}(h)$ are equal to
\begin{equation*}
   \bigg(1+\frac{h|\ln(h)|^{3/2}}{\pi\sqrt{2}}\bigg)^{-2}\lambda_n^\Delta(h) .
\end{equation*}
We can extend the elements of $\mathcal{S}_{N_{0}}(h)$ by zero so that $\cQ(h)(\psi_h) = \cQ_{\Tri(h)}(h)(\psi_h)$ for $\psi_h\in\mathcal{S}_{N_{0}}(h)$. Using \eqref{ineq:SN} and the min-max principle, we get, for all $n\in\{1,\dots N_0\}$,
\begin{equation*}
\bigg(1+\frac{h|\ln(h)|^{3/2}}{\pi\sqrt{2}}\bigg)^{-2}\lambda_n^\Delta(h) \leq \lambda_n(h) + Ch^2.
\end{equation*}
Together with \eqref{eqn:majvp} and Theorem \ref{th:O14} this implies Theorem \ref{th:expavp}.\hfill$\square$\par\medskip

\appendix
\section{Numerical results}
\label{sec:numresult}


\begin{figure}[h]
\includegraphics[scale=0.66]{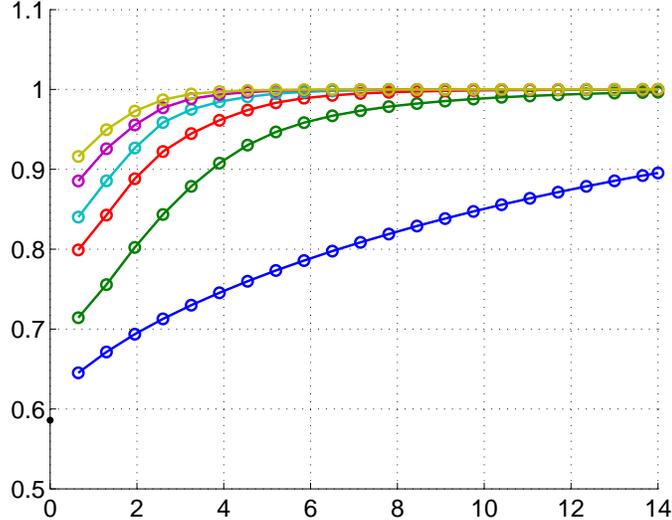}
\caption{First 6 eigenvalues $\mu_{n}(\theta)$ ($n = 1,\dots,6$) vs $\theta$ (in degrees). The black dot is the value ${j_{0,1}^2}/{\pi^2}$. Mesh with $1616$ triangles and \#DOF $=30325$.}
\label{fig:layvalpro}
\end{figure}

We illustrate properties satisfied by the eigenvalues and eigenfunctions through numerical simulations. The computations are performed with the operator $\mathcal{D}_\Gui(\theta)$ defined in \eqref{eqn:defopD}. The integration domain has to be finite: We truncate the domain $\Gui$ sufficiently far away from the origin. We use the finite element library M\'{e}lina++Ê \cite{Melina++} with interpolation degree $6$ and a quadrature rule of degree $13$. The elements are triangles and the different meshes and the number of DOF (degrees of freedom)  are mentioned in legends.

Figure \ref{fig:layvalpro} shows the six first eigenvalues as functions of $\theta$. The monotonicity of eigenvalues, their accumulation below the threshold of the essential spectrum, and their convergence to ${j_{0,1}^2}/{\pi^2}$ as $\theta\to0$ can be observed.

\begin{figure}[h]
\begin{center}
\begin{tabular}{cc}
\includegraphics[scale=0.4]{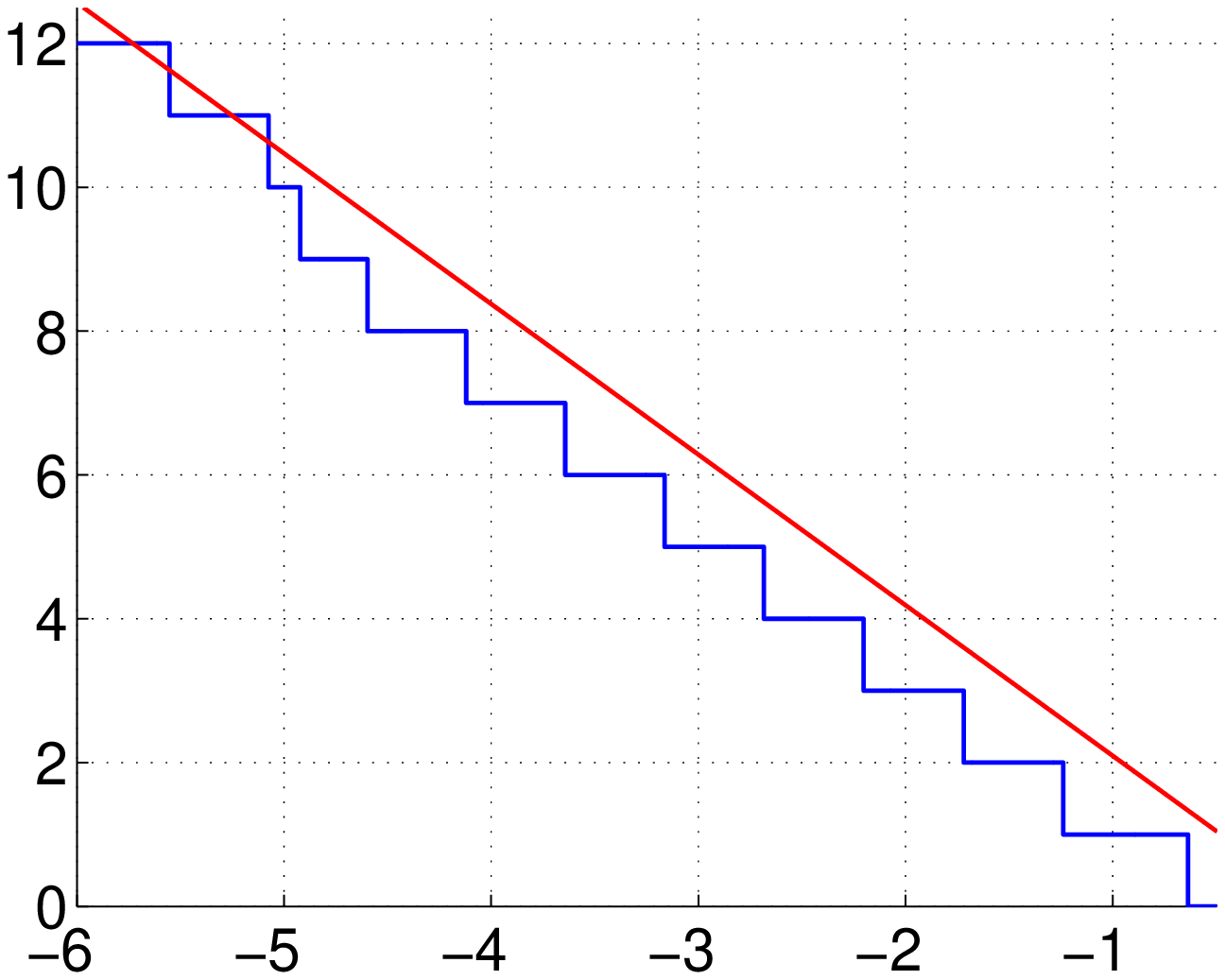} & \includegraphics[scale=0.4]{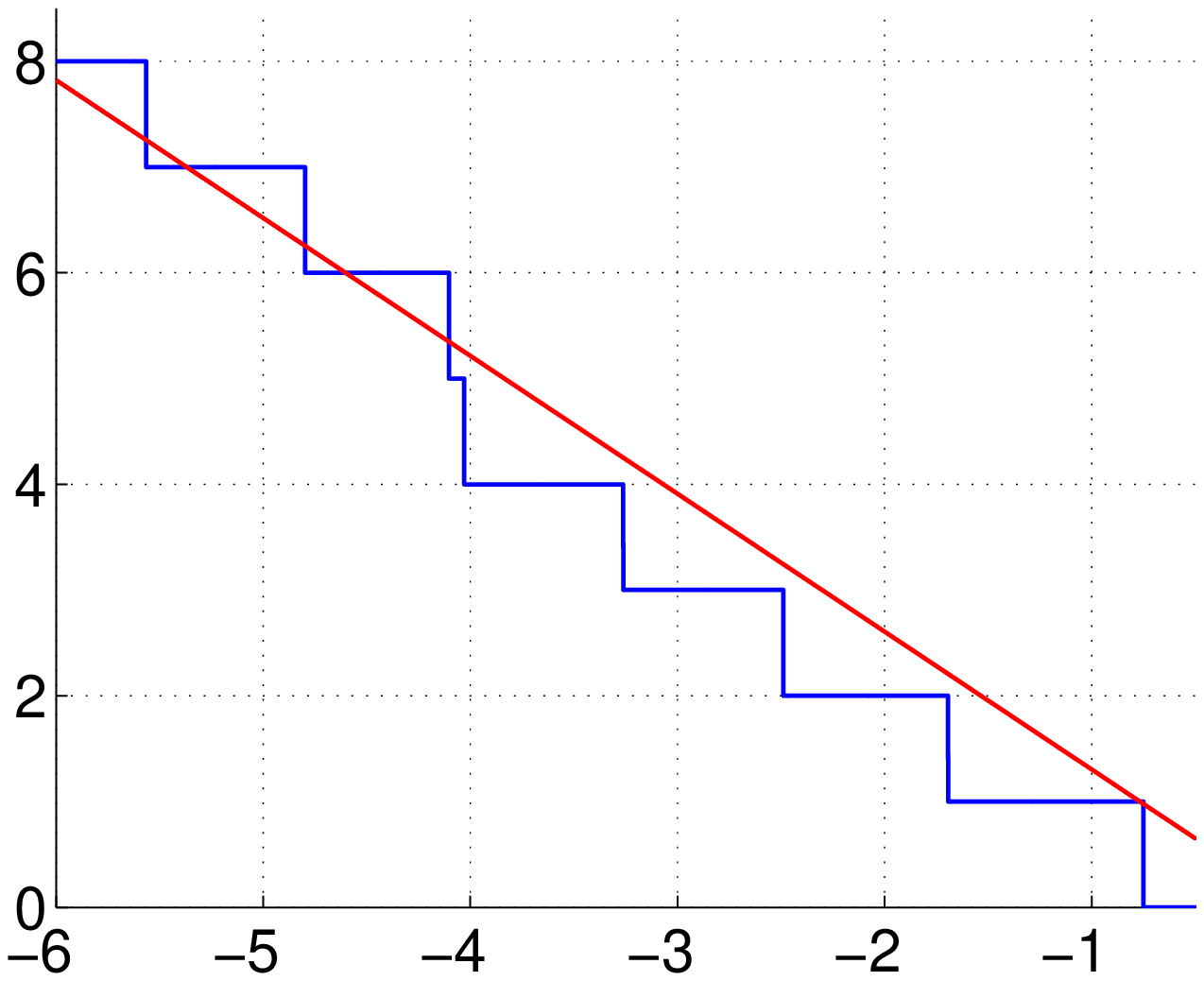}
\\
$\theta = 5^\circ$ & $\theta = 8^\circ$\\
\end{tabular}
\end{center}
\caption{The blue line is the numerical computation of $\mathcal{N}_{1-E}(\cH_{\Gui(\theta)})$ as function of $\ln_{10} E$ in abscissa. The red line is the asymptotics of Theorem \ref{th:countvp}. Mesh with $6416$ triangles and \#DOF $=120325$.}
\label{fig:curveNE}
\end{figure}
We computed the first 12 eigenvalues for $\theta=5^\circ$ and first 8 for $\theta = 8^\circ$ 
We are interested in $E_j(\theta) = 1 - \mu_j(\theta)$ because, by definition, for all $E\in[E_{j+1}(\theta), E_j(\theta))$ we have $\mathcal{N}_{1-E}(\cH_{\Gui(\theta)}) = \mathcal{N}_{1-E}(\mathcal{D}_{\Gui}(\theta)) = j$. 
The jumps occur at $E=E_j(\theta)$. In Figure \ref{fig:curveNE}, we plot the quantity $\mathcal{N}_{1-E}(\cH_{\Gui(\theta)})$ as function of $\ln_{10} E$, compared with the asymptotics of Theorem \ref{th:countvp}.

\begin{figure}[h!]
\begin{center}
\begin{tabular}{cc}
\includegraphics[scale=0.25]{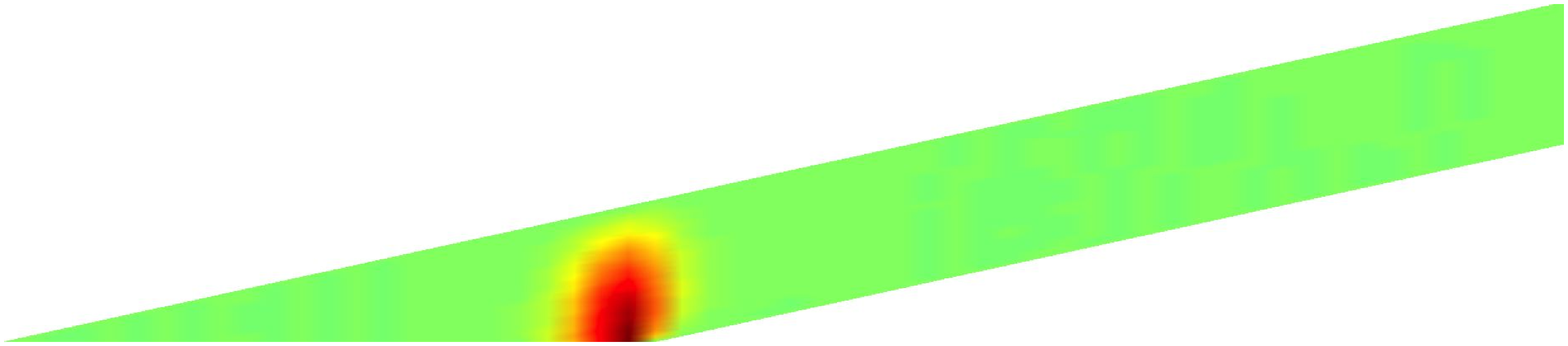} & \includegraphics[scale=0.25]{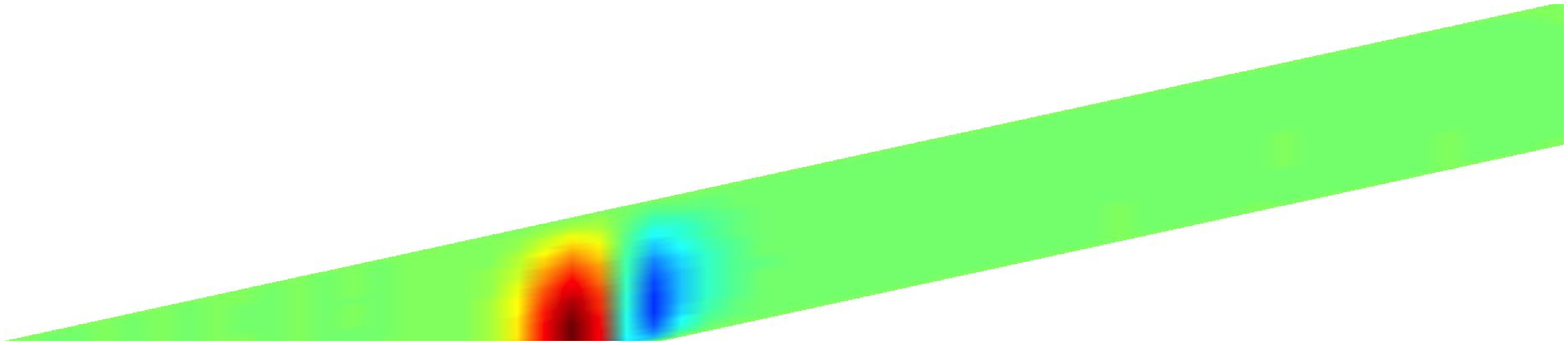}\\
$\mu_{1}(\theta) = 0.709909$ & $\mu_{2}(\theta) = 0.837417 $\\
\includegraphics[scale=0.25]{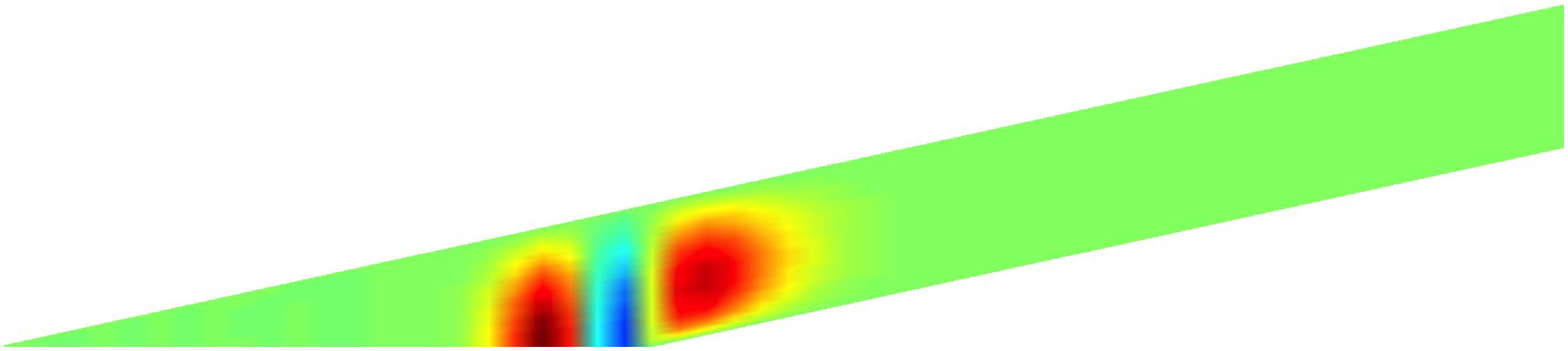} & \includegraphics[scale=0.25]{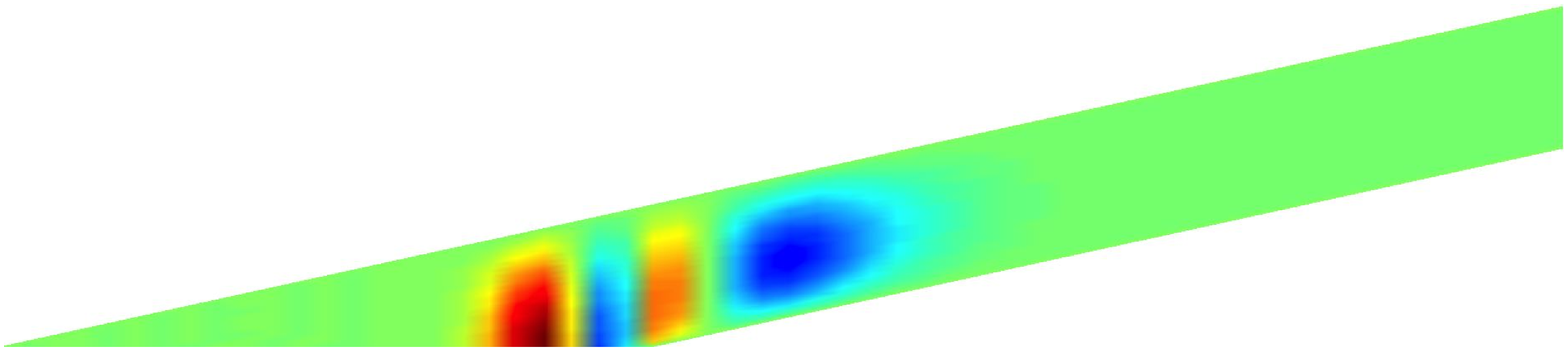}\\
$\mu_{3}(\theta) = 0.917956 $ & $\mu_{4}(\theta) = 0.954728 $\\
\includegraphics[scale=0.25]{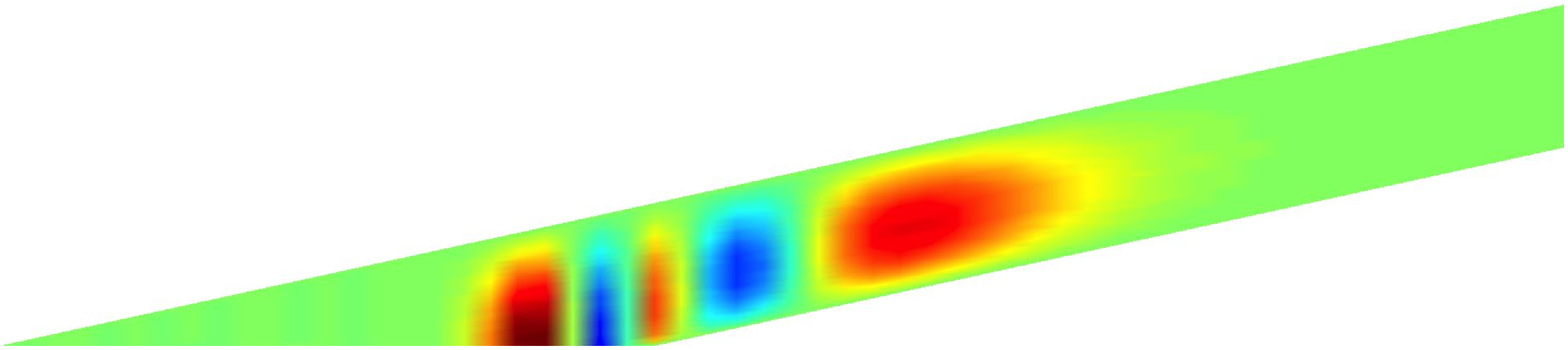} & \includegraphics[scale=0.25]{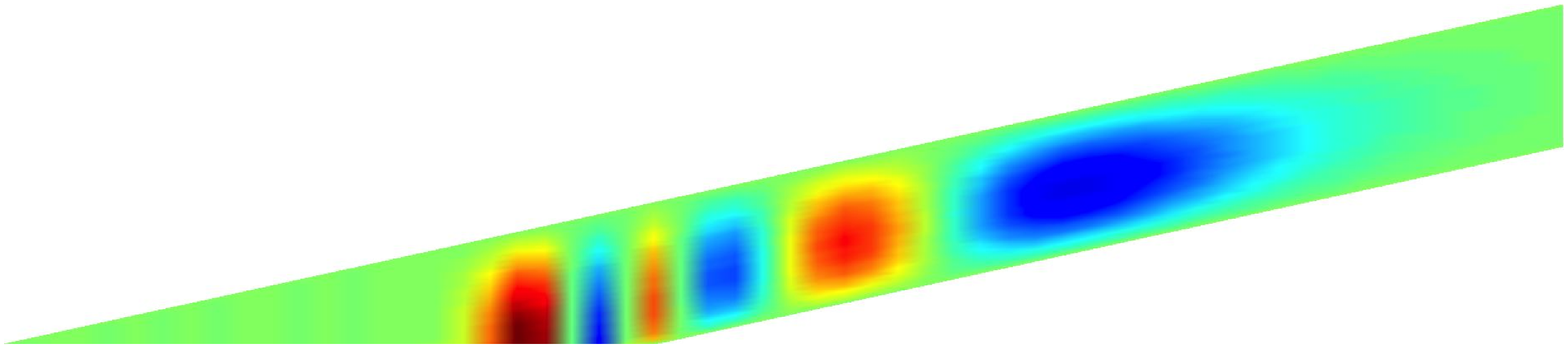}\\
$\mu_{5}(\theta) = 0.974223 $ & $\mu_{6}(\theta) = 0.985379 $\\
\end{tabular}
\end{center}
\caption{Computations for $\theta = 2.5^\circ$. The associated eigenfunctions are represented in the physical domain $\Gui(\theta)$ where the vertical axes is multiplied by $5$. Mesh with $6416$ triangles and \#DOF $=120325$.}
\label{fig:layvpsemicl}
\end{figure}

Finally, in Figure \ref{fig:layvpsemicl} we depict the first eigenfunctions for small $\theta$. This is the same numerical computations as in \cite[Figure 3]{ET11} and it enlightens the Agmon localization estimates of Proposition \ref{prop:agmlay}: The eigenfunctions penetrate into the meridian guide $\Gui(\theta)$ and, unlike the two dimensional broken waveguides, are not only localized in the triangular end of the meridian guide.


\subsection*{Acknowledgements}
Thomas Ourmi\`eres-Bonafos is supported by the Basque Government through the BERC 2014-2017 program and by Spanish Ministry of Economy and Competitiveness MINECO: BCAM Severo Ochoa excellence accreditation SEV-2013-0323.


\begin{thebibliography}{10}

\bibitem{AS66}
M.~{Abramowitz} and I.~A. {Stegun}.
\newblock {\em Handbook of mathematical functions, with formulas, graphs, and
  mathematical tables}.
\newblock Dover Publications, Inc., New York, 1966.

\bibitem{A82}
S.~Agmon.
\newblock {\em Lectures on exponential decay of solutions of second-order
  elliptic equations: bounds on eigenfunctions of {$N$}-body {S}chr\"odinger
  operators}, volume~29 of {\em Mathematical Notes}.
\newblock Princeton University Press, Princeton, NJ, 1982.

\bibitem{A85}
S.~Agmon.
\newblock Bounds on exponential decay of eigenfunctions of {S}chr\"odinger
  operators.
\newblock In {\em Schr\"odinger operators ({C}omo, 1984)}, volume 1159 of {\em
  Lecture Notes in Math.}, pages 1--38. Springer, Berlin, 1985.

\bibitem{BEL14}
J.~{Behrndt}, P.~{Exner}, and V.~{Lotoreichik}.
\newblock {Schr\"odinger operators with $\delta$-interactions supported on
  conical surfaces}.
\newblock {\em submitted, J. Phys. A}, 2014.

\bibitem{BDM99}
C.~Bernardi, M.~Dauge, and Y.~Maday.
\newblock {\em Spectral methods for axisymmetric domains}, volume~3 of {\em
  Series in Applied Mathematics (Paris)}.
\newblock Gauthier-Villars, \'Editions Scientifiques et M\'edicales Elsevier,
  Paris, 1999.
\newblock Numerical algorithms and tests due to Mejdi Aza{\"{\i}}ez.

\bibitem{BO27}
M.~Born and R.~Oppenheimer.
\newblock Zur quantentheorie der molekeln.
\newblock {\em Annalen der Physik}, 389(20):457--484, 1927.

\bibitem{CEK04}
G.~Carron, P.~Exner, and D.~Krej{\v{c}}i{\v{r}}{\'{\i}}k.
\newblock Topologically nontrivial quantum layers.
\newblock {\em J. Math. Phys.}, 45(2):774--784, 2004.

\bibitem{CDS81}
J.~Combes, P.~Duclos, and R.~Seiler.
\newblock The born-oppenheimer approximation.
\newblock In G.~Velo and A.~Wightman, editors, {\em Rigorous Atomic and
  Molecular Physics}, volume~74 of {\em NATO Advanced Study Institutes Series},
  pages 185--213. Springer US, 1981.

\bibitem{CFKS87}
H.~L. Cycon, R.~G. Froese, W.~Kirsch, and B.~Simon.
\newblock {\em Schr\"odinger operators with application to quantum mechanics
  and global geometry}.
\newblock Texts and Monographs in Physics. Springer-Verlag, Berlin, study
  edition, 1987.

\bibitem{DLR12}
M.~Dauge, Y.~Lafranche, and N.~Raymond.
\newblock {Quantum waveguides with corners}.
\newblock {\em ESAIM: Proceedings}, 35:14--45, 2012.

\bibitem{DR12}
M.~Dauge and N.~Raymond.
\newblock Plane waveguides with corners in the small angle limit.
\newblock {\em Journal of Mathematical Physics}, 53(12):123529, 2012.

\bibitem{DJ99}
M.~Dimassi and J.~Sj{\"o}strand.
\newblock {\em Spectral asymptotics in the semi-classical limit}, volume 268 of
  {\em London Mathematical Society Lecture Note Series}.
\newblock Cambridge University Press, Cambridge, 1999.

\bibitem{DE95}
P.~Duclos and P.~Exner.
\newblock Curvature-induced bound states in quantum waveguides in two and three
  dimensions.
\newblock {\em Rev. Math. Phys.}, 7(1):73--102, 1995.

\bibitem{DEK01}
P.~Duclos, P.~Exner, and D.~Krej{\v{c}}i{\v{r}}{\'{\i}}k.
\newblock Bound states in curved quantum layers.
\newblock {\em Comm. Math. Phys.}, 223(1):13--28, 2001.

\bibitem{ES89}
P.~{Exner} and P.~{\v{S}eba}.
\newblock {Bound states in curved quantum waveguides.}
\newblock {\em {J. Math. Phys.}}, 30(11):2574--2580, 1989.

\bibitem{ET11}
P.~Exner and M.~Tater.
\newblock Spectrum of {D}irichlet {L}aplacian in a conical layer.
\newblock {\em J. Phys. A}, 43(47):474023, 11, 2010.

\bibitem{ESS89}
P.~Exner, P.~\v{S}eba, and P.~\v{S}t'ovi\v{c}ek.
\newblock {On existence of a bound state in an L-shaped waveguide}.
\newblock {\em Czechoslovak Journal of Physics}, 39:1181--1191, 1989.

\bibitem{FH10}
S.~Fournais and B.~Helffer.
\newblock {\em Spectral methods in surface superconductivity}.
\newblock Progress in Nonlinear Differential Equations and their Applications,
  77. Birkh\"auser Boston Inc., Boston, MA, 2010.

\bibitem{GJ92}
J.~Goldstone and R.~L. Jaffe.
\newblock Bound states in twisting tubes.
\newblock {\em Phys. Rev. B}, 45:14100--14107, Jun 1992.

\bibitem{HM08}
A.~{Hassell} and S.~{Marshall}.
\newblock {Eigenvalues of Schr\"odinger operators with potential asymptotically
  homogeneous of degree $-2$.}
\newblock {\em {Trans. Am. Math. Soc.}}, 360(8):4145--4167, 2008.

\bibitem{H88}
B.~Helffer.
\newblock {\em Semi-classical analysis for the {S}chr\"odinger operator and
  applications}, volume 1336 of {\em Lecture Notes in Mathematics}.
\newblock Springer-Verlag, Berlin, 1988.

\bibitem{HJ84}
B.~Helffer and J.~Sj{\"o}strand.
\newblock Multiple wells in the semiclassical limit. {I}.
\newblock {\em Comm. Partial Differential Equations}, 9(4):337--408, 1984.

\bibitem{HJ85}
B.~Helffer and J.~Sj{\"o}strand.
\newblock Puits multiples en limite semi-classique. {II}. {I}nteraction
  mol\'eculaire. {S}ym\'etries. {P}erturbation.
\newblock {\em Ann. Inst. H. Poincar\'e Phys. Th\'eor.}, 42(2):127--212, 1985.

\bibitem{J14}
T.~{Jecko}.
\newblock {On the mathematical treatment of the {B}orn-{O}ppenheimer
  approximation}.
\newblock {\em to appear, J. Math. Phys.}, 2014.

\bibitem{KS88}
W.~Kirsch and B.~Simon.
\newblock Corrections to the classical behavior of the number of bound states
  of {S}chr\"odinger operators.
\newblock {\em Ann. Physics}, 183(1):122--130, 1988.

\bibitem{KMSW92}
M.~Klein, A.~Martinez, R.~Seiler, and X.~P. Wang.
\newblock On the {B}orn-{O}ppenheimer expansion for polyatomic molecules.
\newblock {\em Comm. Math. Phys.}, 143(3):607--639, 1992.

\bibitem{Melina++}
Y.~Lafranche and D.~Martin.
\newblock {\sc M\'elina++}, biblioth\`eque de calculs \'el\'ements finis.
\newblock {\em \vskip0pt{\sf\small
  http\char58//anum-maths.univ-rennes1.fr/melina/}}, 2012.

\bibitem{M89}
A.~Martinez.
\newblock D\'eveloppements asymptotiques et effet tunnel dans l'approximation
  de {B}orn-{O}ppenheimer.
\newblock {\em Ann. Inst. H. Poincar\'e Phys. Th\'eor.}, 50(3):239--257, 1989.

\bibitem{M07}
A.~Martinez.
\newblock A general effective {H}amiltonian method.
\newblock {\em Atti Accad. Naz. Lincei Cl. Sci. Fis. Mat. Natur. Rend. Lincei
  (9) Mat. Appl.}, 18(3):269--277, 2007.

\bibitem{MT05}
A.~Morame and F.~Truc.
\newblock Remarks on the spectrum of the {N}eumann problem with magnetic field
  in the half-space.
\newblock {\em J. Math. Phys.}, 46(1):012105, 13, 2005.

\bibitem{NS14}
S.~{Nazarov} and A.~{Shanin}.
\newblock {Trapped modes in angular joints of 2D waveguides.}
\newblock {\em {Appl. Anal.}}, 93(3):572--582, 2014.

\bibitem{O14}
T.~Ourmi\`eres-Bonafos.
\newblock Dirichlet eigenvalues of cones in the small aperture limit.
\newblock {\em To appear, Journal of Spectral Theory}, 2014.

\bibitem{RS78}
M.~Reed and B.~Simon.
\newblock {\em Methods of modern mathematical physics. {IV}. {A}nalysis of
  operators}.
\newblock Academic Press, New York, 1978.

\bibitem{S83}
B.~Simon.
\newblock Semiclassical analysis of low lying eigenvalues. {I}. {N}ondegenerate
  minima: asymptotic expansions.
\newblock {\em Ann. Inst. H. Poincar\'e Sect. A (N.S.)}, 38(3):295--308, 1983.

\end{thebibliography}

\end{document}